\documentclass[10pt,reqno]{amsart}
\usepackage{latexsym,amsmath,amssymb,amscd}
\usepackage[all]{xy}
\usepackage{enumerate}
\usepackage{color}
\usepackage[T1]{fontenc}
\usepackage[%pagebackref, 
colorlinks=true,linkcolor=blue,citecolor=red]{hyperref} 
\makeatletter
\newcommand\@dotsep{4.5}
\def\@tocline#1#2#3#4#5#6#7{\relax
	\ifnum #1>\c@tocdepth % then omit
	\else
	\par \addpenalty\@secpenalty\addvspace{#2}%
	\begingroup \hyphenpenalty\@M
	\@ifempty{#4}{%
		\@tempdima\csname r@tocindent\number#1\endcsname\relax
	}{%
		\@tempdima#4\relax
	}%
	\parindent\z@ \leftskip#3\relax \advance\leftskip\@tempdima\relax
	\rightskip\@pnumwidth plus1em \parfillskip-\@pnumwidth
	#5\leavevmode\hskip-\@tempdima #6\relax
	\leaders\hbox{$\m@th
		\mkern \@dotsep mu\hbox{.}\mkern \@dotsep mu$}\hfill
	\hbox to\@pnumwidth{\@tocpagenum{#7}}\par
	\nobreak
	\endgroup
	\fi}
\makeatother 
	\makeatletter
	\@addtoreset{figure}{section}
	\def\thefigure{\thesection.\@arabic\c@figure}
	\def\fps@figure{h,t}
	\@addtoreset{table}{bsection}
	
	\def\thetable{\thesection.\@arabic\c@table}
	\def\fps@table{h, t}
	\@addtoreset{equation}{section}
	
	\makeatother

\usepackage{amssymb}
\include{diagxy}
\input{xy}
\xyoption{all}
\usepackage{amscd}

\usepackage{enumerate}

\usepackage{accents}

\DeclareMathOperator{\pr}{pr}

\begin{document}

\newcommand*{\dt}[1]{%
	\accentset{\mbox{\large\bfseries .}}{#1}}
\newcommand*{\ddt}[1]{%
	\accentset{\mbox{\large\bfseries .\hspace{-0.25ex}.}}{#1}}
\renewcommand{\dot}{\dt}

\newcommand{\Z}{\mathbb Z}
\newcommand{\C}{\mathbb C}
\newcommand{\N}{\mathbb N}
\newcommand{\R}{\mathbb R}
\newcommand{\I}{\mathbbm 1}
\newcommand{\0}{\mathbbm O}
\newcommand{\CP}{\mathbb{CP}}
\newcommand{\eps}{\varepsilon}
\newcommand{\A}{\mathcal{A}}
\newcommand{\B}{\mathcal{B}}
\newcommand{\Cc}{{\mathcal C}}
\newcommand{\K}{\mathcal{K}}
\newcommand{\J}{\mathcal{J}}
\newcommand{\T}{\mathcal{T}}
\renewcommand{\P}{\mathcal{P}}
\newcommand{\G}{\mathcal{G}}
\newcommand{\AG}{\mathcal{AG}}
\newcommand{\AU}{\mathcal{AU}}
\newcommand{\U}{\mathcal{U}}
\newcommand{\V}{\mathcal{V}}
\newcommand{\Xc}{{\mathcal X}}
\newcommand{\bg}{{\mathfrak{b}}}
\renewcommand{\gg}{{\mathfrak{g}}}
\newcommand{\rg}{{\mathfrak{r}}}
\newcommand{\ug}{{\mathfrak{u}}}

\renewcommand{\H}{\mathcal{H}}
\renewcommand{\ll}{l}
\newcommand{\rr}{r}
\newcommand{\Ll}{\mathcal{L}}
\newcommand{\Ss}{\textbf{s}}
\newcommand{\Tt}{\textbf{t}}
\newcommand{\anch}{\textbf{a}}
\newcommand{\tto}{\rightrightarrows}
\newcommand{\M}{\mathfrak{M}}
\newcommand{\X}{\mathfrak{X}}
\newcommand{\Y}{\mathcal{Y}}
\renewcommand{\to}{\rightarrow}
\newcommand{\prf}[1]{\begin{proof}#1\end{proof}}

\renewcommand{\leq}{\leqslant}
\renewcommand{\geq}{\geqslant}

\newcommand{\dq}{\partial_q}
\newcommand{\To}{\longrightarrow}
\newcommand{\ol}{\overline}
\newcommand{\bra}[1]{\left\langle #1 \right|}
\newcommand{\ket}[1]{\left| #1 \right\rangle}
\newcommand{\oper}[2]{\left|#1\right\rangle\left\langle #2 \right|}

\renewcommand{\1}{{\bf 1}}
\newcommand{\Ad}{{\rm Ad}}
\newcommand{\ad}{{\rm ad}}
\newcommand{\Aut}{{\rm Aut}}
\newcommand{\Banach}{}%{Banach }
\newcommand{\bepsilon}{{\boldsymbol\epsilon}}
\newcommand{\bl}{{\boldsymbol\ell}}
\newcommand{\br}{{\boldsymbol r}}
\newcommand{\bi}{{\boldsymbol\iota}}
\newcommand{\bmu}{{\boldsymbol m}}
\newcommand{\bomega}{{\boldsymbol\omega}}
\newcommand{\bs}{{\boldsymbol s}}
\newcommand{\bt}{{\boldsymbol t}}
\newcommand{\Cas}{{\rm Cas}}
\newcommand{\Ci}{{\mathcal C}^\infty}
\newcommand{\Diff}{{\rm Diff}}
\newcommand{\de}{{\rm d}}
\newcommand{\g}{{g}}
\renewcommand{\Im}{{\rm Im}}
\newcommand{\id}{{\rm id}}
\newcommand{\Inn}{{\rm Inn}}
\newcommand{\Iso}{{\rm Iso}}
\newcommand{\inn}{{\rm inn}}
\newcommand{\Ker}{{\rm Ker}\,}
\newcommand{\Lie}{\text{\bf L}}
\newcommand{\LP}{{\text{L-P}}}
\newcommand{\lin}{ \mathfrak{l}}
\newcommand{\Oc}{{\mathcal O}}
\newcommand{\Poiss}{{\rm Poiss}}
\newcommand{\pre}{{\rm pre}}
\newcommand{\Tr}{{\rm Tr}\,}

\newcommand{\bfi}{\bfseries\itshape}

% ----------------------------------------------------------------
\vfuzz2pt % Don't report over-full v-boxes if over-edge is small
\hfuzz2pt % Don't report over-full h-boxes if over-edge is small
% THEOREMS -------------------------------------------------------
\newtheorem{thm}{Theorem}[section]
\newtheorem{cor}[thm]{Corollary}
\newtheorem{lem}[thm]{Lemma}
\newtheorem{stat}[thm]{Statement}
\newtheorem{prop}[thm]{Proposition}
\newtheorem{axiom}{Axiom}
\theoremstyle{definition}
\newtheorem{defn}[thm]{Definition}
\newtheorem{example}[thm]{Example}
%\numberwithin{example}{section}
%\theoremstyle{remark}
\newtheorem{rem}[thm]{Remark}
\newtheorem{notation}[thm]{Notation}

\numberwithin{equation}%{chapter}
{section}

%\numberwithin{section}{chapter}

%get the sub-sub-...-subsections into the table of contents: 
%\setcounter{tocdepth}{2} 

%\DeclareMathOperator{\pr}{pr}

%\renewcommand{\emph}[1]{{\bf #1}}
\newcommand{\ex}[1]{{\it #1}}

\title[Poisson geometry and modular theory]{Poisson Geometrical Aspects \\ of the Tomita-Takesaki Modular Theory}

%    author one information
\author{Daniel Belti\c t\u a}
\address{Institute of Mathematics ``Simion Stoilow'' of the Romanian Academy, 
	P.O. Box 1-764, Bucharest, Romania}
%\curraddr{}
\email{Daniel.Beltita@imar.ro, beltita@gmail.com}
%\thanks{}

%    author two information
\author{Anatol Odzijewicz$\dagger$}
\address{Institute of Mathematics, University of Bia\l ystok, 
		Cio\l kowskiego 1M, 15-245 Bia\l ystok, Poland}
%\curraddr{}
\email{aodzijew@uwb.edu.pl}
\thanks{$\dagger$Anatol Odzijewicz passed away on 18.04.2022.}

%\date{\today% . File name: \jobname.tex}

%\date{\today}

%\author{Daniel Belti\c t\u a\footnote{Institute of Mathematics ``Simion Stoilow'' of the Romanian Academy, 
%P.O. Box 1-764, Bucharest, Romania. \emph{Email}: Daniel.Beltita@imar.ro, beltita@gmail.com}  
%\ \ and 
%\setcounter{footnote}{6}
%Anatol Odzijewicz\footnote{Institute of Mathematics, University of Bia\l ystok, 
%Cio\l kowskiego 1M, 15-245 Bia\l ystok, Poland. \emph{Email}: 
%\break 
%aodzijew@uwb.edu.pl}}

% ----------------------------------------------------------------

%\begin{document}
%\maketitle

\begin{abstract}
We investigate some genuine Poisson geometric objects in the modular theory of an arbitrary von Neumann algebra $\mathfrak{M}$.  
Specifically, for any standard form realization $(\mathfrak{M},\mathcal{H},J,\mathcal{P})$, 
we find a canonical foliation of the Hilbert space $\mathcal{H}$, 
whose leaves are Banach manifolds that are weakly immersed into~$\mathcal{H}$, 
thereby endowing $\mathcal{H}$ with a richer Banach manifold structure to be denoted by~$\widetilde{\mathcal{H}}$. 
We also find that $\widetilde{\mathcal{H}}$ has the structure of a Banach-Lie groupoid $\widetilde{\mathcal{H}}\rightrightarrows\mathfrak{M}_*^+$ 
which is isomorphic to the action groupoid $\mathcal{U}(\mathfrak{M})\ast\mathfrak{M}_*^+\rightrightarrows\mathfrak{M}_*^+$ defined by the natural action of 
the Banach-Lie groupoid of partial isometries $\mathcal{U}(\mathfrak{M})\rightrightarrows\mathcal{L}(\mathfrak{M})$ on the positive cone in the predual $\mathfrak{M}_*^+$, 
where $\mathcal{L}(\mathfrak{M})$ is the projection lattice of $\mathfrak{M}$. 
There is also a presymplectic form $\widetilde{\boldsymbol\omega}\in\Omega^2(\widetilde{\mathcal{H}})$ 
that comes from the scalar product of $\mathcal{H}$ and is multiplicative in the usual sense of finite-dimensional Lie groupoid theory. 
%We further show  that the groupoid 
We further explore some aspects of reduction theory for the groupoid endowed with the multiplicative presymplectic form 
 $(\widetilde{\mathcal{H}},\widetilde{\boldsymbol\omega})\rightrightarrows \mathfrak{M}_*^+$, 
%shares several other properties of finite-dimensional presymplectic groupoids and we investigate 
including the Poisson manifold structures of its orbits and 
%the leaf space of 
the foliation defined by the degeneracy kernel of the presymplectic form~$\widetilde{\boldsymbol\omega}$. 
\\
\emph{Keywods:} Banach-Lie groupoid, multiplicative presymplectic form, Poisson bracket, reduction theory,  von Neumann algebra
\end{abstract}

\maketitle

\tableofcontents

 \section{Introduction}

There is a significant surge in the applications of modular theory of von Neumann algebras to some areas of differential geometry and representation theory of Lie groups, with a view to facilitating the construction of models of field theories that satisfy the requirements of local quantum physics in the sense of \cite{H96}. 
Samples of such areas are the theory of causal symmetric spaces or holomorphic continuation in the framework of Lie group representations. 
See e.g., \cite{Ne18},  \cite{NOOl20}, \cite{NOl19}, \cite{NOl21}, \cite{NOl23}, \cite{MNOl24}, \cite{FNO25}, and the references therein.

The above research direction was very little extended to infinite-dimensional Lie theory so far. 
One of the reasons may be that, despite the substantial development of the theory of infinite-dimensional Lie groups (see e.g., \cite{GN26} and \cite{KM97}), one still lacks a good understanding of the interaction between representation theory and symplectic geometry that governs the coadjoint actions of such groups. 
In the present paper, we take a step towards filling that gap for some infinite-dimensional linear Lie groups, specifically the unitary groups of von Neumann algebras. 
For reasons related to the specifics of Poisson geometry, we extend the scope of our discussion to certain groupoids associated to von Neumann algebras.
In particular, as a by-product, our study  contributes to the incipient theory of infinite-dimensional Lie groupoids, cf. \cite{SW15}, \cite{SW16}, \cite{S23}, \cite{BGJP19}. 

We use a considerable amount of genuine infinite-dimensional Poisson geometry in the sense of \cite{Neeb}, \cite{CP12}, and \cite{OR03}
that is inherent to the modular theory of every von Neumann algebra. 
This is independent of any semiclassical limit process that might lead from  operator algebras to Poisson manifolds. 
Instead, it turns out that for any von Neumann algebra~$\M$ realized in its standard form $(\M,\H,J,\P)$, 
{\it the original manifold structure of its corresponding Hilbert space~$\H$ has a Banach foliation~$\widetilde{\H}$ 
(i.e., the original manifold structure of $\H$ can be enriched to a Banach manifold structure 
for which the identity map $\H\to\widetilde{\H}$ is a weak immersion)  
which further leads to a Banach-Lie groupoid structure $\widetilde{\H}\tto\M_*^+$ 
on the convex cone of positive normal forms of $\M$}  
(Theorem~\ref{grpd_act}).  
Here $\M_*^+$ is the positive cone in the predual $\M_*$ of~$\M$. 
The groupoid $\widetilde{\H}\tto\M_*^+$  carries a presymplectic form $\widetilde{\bomega}\in\Omega^2(\widetilde{\H})$ 
that comes from the scalar product of $\H$, 
is multiplicative 
(Proposition~\ref{graph}) 
in the usual sense 
of finite-dimensional Lie groupoid theory, 
and \emph{gives by reduction the symplectic structure of the groupoid orbits in  $\M_*^+$} (Proposition~\ref{SQ6}).  
(See \cite{Ka86}, \cite{We87}, \cite{BuCWZ04}, \cite{We17}, and the references therein.)
Our construction thus 
leads
to nontrivial examples of multiplicative presymplectic forms on infinite-dimensional 
Lie groupoids.

In this paper we actually work with several groupoid isomorphisms defined 
by suitable polar decompositions of vectors and functionals, which can be summarized in the commutative diagram 
\begin{equation}\label{24August2019}
\xymatrix{
\H \ar@<-.5ex>[d] \ar@<.5ex>[d] &	\U(\M)\ast \M_*^+ \ar@<-.5ex>[d] \ar@<.5ex>[d] \ar[l]_{\Phi\quad}  \ar[r]^{\quad \Xi} 
	& \M_* \ar@<-.5ex>[d] \ar@<.5ex>[d]\\
\M_*^+ &	\M_*^+ \ar[l]_{\id} \ar[r]^{\id} & \M_*^+}
\end{equation}
(cf. Propositions \ref{23August2019} and 
\ref{isomorphisms}). 
Here
the coadjoint action groupoid $\U(\M)\ast \M_*^+ \tto 	\M_*^+$ is a Banach-Lie groupoid in a natural way, 
the total space $\M_*$ of the predual groupoid 
carries a natural Lie-Poisson bracket, while the total space of the standard groupoid $\H\tto \M_*^+$ 
is the Hilbert space $\H$ of the essentially unique standard representation and has a natural symplectic structure. 
(See Section~\ref{Sect6}.)

The point we want to make in this paper is that the geometrical aspects of the above three groupoids actually complement each other
and it is the interaction of these aspects which provides  Poisson geometrical information 
on the von Neumann algebra $\M$. 
These three groupoids taken together actually share suitable versions 
of the key features of any symplectic groupoid $G\tto P$ from the finite-dimensional Poisson geometry \cite{We87}: 
the base $P$ has a Poisson structure for which the source/target maps are Poisson/anti-Poisson maps, 
the graph of the groupoid multiplication is a Lagrangian submanifold of $G\times G\times\overline{G}$, 
the base $P$ is a Lagrangian submanifold of the symplectic manifold $G$, and so on. 
It looks like a worthwhile 
endeavor 
in the future to find the explanation of the occurrence of these deep genuine Poisson geometric aspects 
in the theory of von Neumann algebras. 
As already mentioned above, we find versions of these features that are adapted to multiplicative presymplectic forms on infinite-dimensional Lie groupoids.

It is noteworthy that infinite-dimensional geometric structures associated to operator algebras 
have been studied extensively, 
from a variety of perspectives including but not limited to Toeplitz operators \cite{ACL18}, 
Riemannian and Finsler geometry \cite{ALR10}, symmetric spaces \cite{Up85,Ne02}, 
representations of Banach-Lie groups \cite{NO98}, Banach-Poisson manifolds 
and particularly Banach Lie-Poisson spaces \cite{OR03,OJS2,OS}, Banach-Lie algebroids \cite{OJS1}, 
geometry of polar decomposition and generalized inverses \cite{ACM05, BGJP19, CM25} etc. 
From the point of view of these earlier studies, 
what we are doing here is to explore some of the Poisson geometric structures 
that are encoded in the standard form of von Neumann algebras. 

Before presenting the contents of the present paper, we recall that 
several striking similarities between finite-dimensional Poisson geometry and the theory of von Neumann algebras 
were pointed out by A.~Weinstein in his seminal paper~\cite{We97} and also for instance in the book \cite{CW}, 
motivated by an attempt to understand the classical limit of 
quantum theory, viewing Poisson manifolds 
as semiclassical limits of operator algebras in some sense. 
Such a collection of similarities serves at any rate as 
a very useful guide for study of each one of these two theories. 
As emphasized in the aforementioned paper, 
the investigation of the modular vector fields of Poisson manifolds may help one 
gain a geometric perspective on the modular operators of von Neumann algebras.

{\bf Outline of this paper}.  
In Section~\ref{Sect2} we have recorded notions of Poisson geometry on infinite-dimensional manifolds 
that are needed in the later sections. 
Thus, in Subsection~\ref{Subsect2.1} we introduce a suitable notion of Poisson structure that 
is obtained by adapting the notions of Poisson structure from \cite{Neeb} and sub-Poisson structures from \cite{CP12}.  
In Subsection~\ref{Subsect2.2} we show how Poisson structures in the above sense 
can be constructed from weakly symplectic structures on manifolds modeled on Banach spaces, 
thus avoiding a well known problem from the theory of infinite-dimensional Hamiltonian systems 
that was already noted in \cite{ChM}. 
Subsection~\ref{Subsect2.3} records some basic facts on the most important class of 
infinite-dimensional Poisson manifolds that was studied so far, namely the infinite-dimensional Lie-Poisson spaces, 
that is, Banach spaces $\mathfrak{b}$ for which 
the topological dual space $\mathfrak{b}^*$ has the structure of 
Banach-Lie algebra whose Lie bracket $[\cdot,\cdot]\colon \mathfrak{b}^*\times \mathfrak{b}^*\to \mathfrak{b}^*$ 
is separately continuous with respect to the weak dual topology of $\mathfrak{b}^*$. 
Basic examples of Lie-Poisson spaces are the preduals of von Neumann algebras. 
In Subsections~\ref{Subsect2.4} and \ref{Subsect2.5} 
we recall the Banach-Lie groupoid structure of the sets of closed-range elements of a von Neumann algebra, 
with its groupoids consisting of the partial isometries and the partially invertible elements. 

In Section~\ref{Sect3}, the main result is that the middle  groupoid in the diagram \eqref{24August2019}, 
that is, the coadjoint action groupoid $\U(\M)\ast \M_*^+ \tto 	\M_*^+$,  
has the natural structure of a Banach-Lie groupoid (Theorem~\ref{coadj_grpd}). 
To this end, we study the algebraic structure of that groupoid in Subsection~\ref{Subsect3.1}, while 
 Subsection~\ref{Subsect3.2} is devoted to the differential geometric structures of the coadjoint action groupoid 
and of its orbits (Theorem~\ref{thm:32}). 
 Our method of investigation of this groupoid is based on the study of its transitive subgroupoids, 
which are naturally isomorphic as Banach-Lie groupoids 
to the gauge groupoids associated to suitable principal bundles (Proposition~\ref{gauge_isom}). 

In Section~\ref{Sect4} we show that the pair consisting of a von Neumann algebra and its commutant leads to some examples of genuine dual pairs 
in the sense of infinite-dimensional Poisson geometry. 
We also investigate here some basic properties of the expectation maps that later play a key role in the construction of the standard groupoid associated to a von Neumann algebra. 

In Section~\ref{Sect5} we study the natural action of the groupoid of partial isometries $\U(\M)\tto\Ll(\M)$ on 
the positive cone $\M_*^+$ in the predual, the corresponding momentum map being given by the support projection of normal functionals. 
We call this action the coadjoint action of $\U(\M)\tto\Ll(\M)$ since its restriction to the vertex subgroupoid 
coresponds to the family of coadjoint actions of the Banach-Lie unitary groups $U(p\M p)$ for arbitrary projections $p\in\Ll(\M)$. 
Specifically, we show that every groupoid orbit has a weakly symplectic structure obtained from symplectic reduction from the Hilbert space~$\H$, 
which agrees with the Kirillov-Kostant-Souriau construction on unitary \emph{group} orbits, 
and this leads to Poisson brackets on the partial-isometry \emph{groupoid} orbits (Theorem~\ref{thm:orbit} and Corollary~\ref{cor:orbit}). 
Then, in Subsection~\ref{Subsect5.2} we briefly discuss the special properties of $\M$ encoded in these groupoid orbits, with emphasis on type~III$_1$ 
(Proposition~\ref{III1}). 

In Section~\ref{Sect6} we define the standard groupoid $\H\tto\M_*^+$ associated to any standard form $(\M,\H,J,\P)$ 
of a von Neumann algebra~$\M$ and we describe a few other realizations of that groupoid (Theorem~\ref{thm:53} and Proposition~\ref{531}). 
We further construct the natural foliation $\widetilde{\H}$, of the Hilbert space~$\H$, 
whose leaves are Banach manifolds which are weakly immersed into~$\H$ and whose underlying sets are 
the total spaces of the transitive subgroupoids of $\H\tto\M_*^+$ (Theorem~\ref{grpd_act}). 
This construction turns the standard groupoid $\H\tto\M_*^+$ into the Banach-Lie groupoid $\widetilde{\H}\tto\M_*^+$. 
Furthermore we endow that Banach-Lie grouoid with 
a presymplectic 2-form $\widetilde{\bomega}\in\Omega^2(\widetilde{\H})$ coming from the symplectic structure of the Hilbert space~$\H$, 
and one of the key technical results that we obtain is that $\widetilde{\bomega}$ is multiplicative, 
in the usual sense from finite-dimensional Lie groupoid theory (Proposition~\ref{graph}). 
This property allows us to recover several features of multiplicative presymplectic forms   
in our present infinite-dimensional setting.  
(See for instance Propositions \ref{SQ6} and~\ref{Prop}.)  

In Section~\ref{Sect7} we complete the picture obtained so far, showing that the modular flows corresponding to normal semifinite faithful weights on the von Neumann algebra~$\M$ 
lead to 1-parameter symmetry groups of our 
%presymplectic 
groupoid with multiplicative presymplectic structure $(\widetilde{\H},\widetilde{\bomega})\tto\M_*^+$, 
that is, diffeomorphisms that preserve the groupoid structural maps, the presymplectic structure, 
as well as the groupoid orbits along with their weakly symplectic structures 
(Propositions~\ref{flow} and \ref{prel_P1}). 
In Subsection~\ref{Subsect7.2} we record a few remarks and open questions on the infinitesimal aspects of these 1-parameter symmetry groups, 
as these aspects are related to the Hamiltonian structures in the infinite-dimensional Poisson geometry that we investigate in this paper. 

In Section~\ref{Sect8}, for the sake of clarity, we illustrate our general constructions by the special case when $\M$ is a type~I factor.  
We find for instance that one of the transitive subgroupoids of 
the 
%presymplectic 
Banach-Lie groupoid with multiplicative presymplectic structure  $(\widetilde{\H},\widetilde{\bomega})\tto\M_*^+$ is isomorphic to the gauge groupoid associated to a Hopf fibration, while the symplectic structure on its base is explicitly expressed in terms of a positive scalar multiple of the Fubini-Study form on an infinite-dimensional projective space.

 \section{Preliminaries}
\label{Sect2}

In this paper, smooth real Banach manifolds are called simply  \emph{manifolds} as for instance in \cite{Bou}, and in particular we use the names of Lie groups/groupoids and vector bundles for the objects that may be called elsewhere in the literature as Banach-Lie groups/groupoids and Banach vector bundles, respectively, as for instance in the general theory of Banach-Lie groupoids developed in~\cite{BGJP19}. 
This convention applies of course only to the manifolds themselves and not to their infinitesimal aspects, hence we still use expressions as Banach space or Banach-Lie algebra. 
It is often the case in this paper that, as in \cite{Bou} and \cite{Lang}, the model Banach spaces of various connected components of manifolds may be non-isomorphic for distinct connected components. 

Given any manifold $M$, it is convenient to use the name of \emph{local chart} for any diffeomorphism $\chi\colon U\to V$ whose domain $U$ is an open subset of $M$ and whose range $V$ is an open subset of another manifold $N$ that may not be a Banach space. 
For any vector bundle $\pi\colon E\to M$ we denote its space of smooth global sections by $\Gamma^\infty E$. 
We denote by $\bigwedge^k E^*$ the total space of the vector bundle over $M$ whose fiber at any point $m\in M$ 
is the Banach space consisting of 
all bounded skew-symmetric $k$-linear functions $\underbrace{E_m\times \cdots\times E_m}_{k\ \rm times}\to\R$. 
If $M$ and $N$ are manifolds, then a smooth mapping $\varphi\colon M\to N$ is called 
a \emph{weak immersion} as for instance in \cite{OR03} if for every $m\in M$ 
the tangent mapping $T_m\varphi\colon T_mM\to T_{\varphi(m)}N$ is injective, without any condition on its range. 

Much of the material presented below is based on the papers \cite{BGJP19,BR05,OJS1,OJS2,OR03,OS}.
Throughout this section, unless otherwise mentioned, $\M$ stands for an arbitrary $W^*$-algebra, $\H$ denotes an arbitrary complex Hilbert space, 
$L^\infty(\H)$ is the von Neumann algebra of all bounded linear operators on $\H$, and $L^1(\H)$ is its canonical predual consisting of all trace-class operators on $\H$. 
We also denote by $\Tr\colon L^1(\H)\to\C$ the canonical operator trace. 

 \subsection{Poisson manifolds}
\label{Subsect2.1}

If one wants to consider a Poisson structure on a \Banach smooth manifold $P$ (modeled on a  class $\mathcal{B}$ of   Banach spaces that are non-reflexive in general, see \cite{Bou,Lang}) then some problems appear which do not occur in finite dimensions. 
This is the reason why  various definitions of Poisson structures were proposed \cite{CP12, Neeb, OR03, CGM17, BGT18} for  
infinite-dimensional manifolds. 
Taking this into account we make the following definition of a Poisson structure being a slight modification of the one presented in~\cite{Neeb}.

\begin{defn}\label{def21} 
	A \emph{Poisson structure} on $P$  is a Lie algebra $(\mathcal{P}^\infty(P,\mathbb{R}),\{\cdot,\cdot\})$, where $\mathcal{P}^\infty(P,\mathbb{R})$ is an associative subalgebra of the algebra $C^\infty(P,\mathbb{R})$ of real smooth functions  with fixed $\mathbb{R}$-linear map $\#:\mathcal{P}^\infty(P,\mathbb{R})\to \Gamma^\infty TP$ such that:
\begin{enumerate}[{\rm(i)}]
\item\label{def21_item1} for any $f,g\in \mathcal{P}^\infty(P,\mathbb{R})$ one has
\begin{equation}\label{cond1} 
(\#f)(g)=\{f,g\}
\end{equation}
\item\label{def21_item2} the algebra $\mathcal{P}^\infty(P,\mathbb{R})$ separates vector fields $\xi\in \Gamma^\infty TP$ on $P$, i.e., if $\xi(f)=0$ for any $f\in \mathcal{P}^\infty(P,\mathbb{R})$ then $\xi\equiv 0$.
\end{enumerate}
\end{defn}

Let us mention that Definition \ref{def21} restricted to the category of finite dimensional manifolds leads to standard Poisson structures if one takes $\mathcal{P}^\infty(P,\mathbb{R})=C^\infty(P,\mathbb{R})$. 

From  
Definition~\ref{def21}(\eqref{def21_item1}--\eqref{def21_item2}) and the Jacobi identity for the Lie bracket $\{\cdot,\cdot\}$ one directly obtains the following facts. 

\begin{prop}\label{prop:1}
	\begin{enumerate}[{\rm(i)}]
\item\label{prop:1_item1} The bracket $\{\cdot,\cdot\}$ satisfies the Leibniz rule:
\begin{equation}\label{L} \{f,gh\}=g\{f,h\}+h\{f,g\}
\end{equation}
for $f,g,h\in \mathcal{P}^\infty(P,\mathbb{R})$
\item\label{prop:1_item2} The map $\#:(\mathcal{P}^\infty(P,\mathbb{R}),\{\cdot,\cdot\})\to (\Gamma^\infty TP,[\cdot,\cdot])$ is a morphism of Lie algebras,
 i.e., 
\begin{equation*} 
\xi_{ \{f,g\}}=[\xi_f,\xi_g]
\end{equation*}
and
\begin{equation*}
\#(f\cdot g)=f\cdot \#g+g\cdot \#f,
\end{equation*}
where $\xi_f:=\# f$ and $[\cdot,\cdot]$ is the Lie bracket of vector fields.
\item\label{prop:1_item3}  The bracket $\{f,g\}$ depends on the differentials  $\de f$ and $\de g$ only.
\end{enumerate}
\end{prop}

We recall commonly used terminology:
\begin{enumerate}[{\rm(i)}]
\item the Lie algebra $\mathcal{P}^\infty(P,\mathbb{R})$ for which the Leibniz rule is fulfilled is called a \emph{Poisson algebra};
\item the vector field $\xi_f=\#f$ is called \emph{Hamiltonian vector field};
\item the elements of kernel $\Ker \#:=\{f\in \mathcal{P}^\infty(P,\mathbb{R}): \#f=0\}$ of the $\mathbb{R}$-linear morphism  
$\#:\mathcal{P}^\infty(P,\mathbb{R})\to \Gamma^\infty TP$ are called \emph{Casimirs} and one has by the condition \eqref{cond1}
$$\{ \Ker \#, \mathcal{P}^\infty(P,\mathbb{R})\}=0,$$
 i.e. the set of Casimirs is the center of the Poisson algebra  $(\mathcal{P}^\infty(P,\mathbb{R}),\{\cdot,\cdot\})$;
\item the vector sub-distribution  $S\subset TP$ of the tangent bundle $TP$ spanned by the vector fields $\xi_f\in \#(\mathcal{P}^\infty(P,\mathbb{R}))$, i.e. $S_p:=\{(\#f)(p): f\in \mathcal{P}^\infty(P,\mathbb{R})\}\subseteq T_pP$, where $p\in P$, is called the \emph{characteristic distribution} of the Poisson structure.
\end{enumerate}

One presented in \cite{BGT18} an example of a Lie algebra $( \mathcal{P}^\infty(\ell^p,\mathbb{R}),\{\cdot,\cdot\})$ on the Banach space  of $p$-summable sequences with Lie bracket satisfying the Leibniz rule for which  the derivations $\{f,\cdot\}$ are not given by vector fields if $\ 1\leq p\leq 2$. 
This shows that the condition \eqref{cond1}  in Definition \ref{def21} is stronger than the Leibniz property \eqref{L} of~$\{\cdot,\cdot\}$.

We note here that, by Proposition \ref{prop:1}\eqref{prop:1_item2}, the characteristic distribution of the Poisson structure is involutive, i.e.,  closed with respect to the Lie bracket 
$[\cdot,\cdot]$ of vector fields. 
This is a part of the sufficient condition for the integrability of the singular vector distributions in the  
Stefan-Sussmann theorem \cite{SS} on finite-dimensional manifolds. 
See also \cite{Pe12} for a version of that integrability theorem for singular vector distributions on infinite-dimensional manifolds.

Let us  define  the sub-bundle $S^*:=\{(\de f)(p): f\in \mathcal{P}^\infty(P,\mathbb{R})\}\subseteq T^*P$ of the cotangent bundle $T^*P$ called  the \emph{co-characteristic distribution} in the following.

Since $\langle (\#f)(p),\de g(p)\rangle=-\langle (\#g)(p),\de f(p)\rangle$ the tangent vector $(\#f)(p)$ at $p\in P$ depends on the differential $\de f(p)$ only, so, the Lie algebra morphism $\#$ defines an identity covering epimorphism $\widehat{\#}:S^*\to S$ of vector bundles. 
Note here that in general the vector subspaces $S^*_p\subseteq T^*_pP$ and $S_p\subseteq T_pP$ 
may not be closed subspaces of the tangent and cotangent spaces at $p\in P$. 
Even if $S_p$ are closed subspaces, 
they may not be isomorphic as Banach spaces for different $p\in P$.
The same concerns $S^*$, too.
Using $\widehat{\#}:S^*\to S$ one can define so called Poisson tensor $\Pi\in \Gamma^\infty(\bigwedge^2 S^*)$ by
\begin{equation*} 
\Pi_p(\varphi_p,\psi_p):=\langle \widehat{\#}(p)(\varphi_p),\psi_p\rangle
=-\langle \widehat{\#}(p)(\psi_p),\varphi_p\rangle.
\end{equation*}
Unlike the finite-dimensional case, 
this geometrical object is, in our opinion, 
less useful for infinite-dimensional manifolds.  
See however the notion of Schouten bracket  for sections of the suitably defined bundle $\oplus_k\bigwedge ^k E\to P$ that  was proposed for  arbitrary vector bundles $E\to P$ in \cite{CP12}. 

Now let us introduce the notion of a morphism $\Phi:P_1\to P_2$ between two \Banach Poisson manifolds
$(P_1, \mathcal{P}^\infty(P_1,\mathbb{R}),\{\cdot,\cdot\}_1,\#_1)$   and 
$(P_2, \mathcal{P}^\infty(P_2,\mathbb{R}),\{\cdot,\cdot\}_2,\#_2)$. 
We use the two definitions of a \emph{Poisson map}, 
namely 
$\Phi^*(\mathcal{P}^\infty(P_2,\mathbb{R}))\subseteq \mathcal{P}^\infty(P_1,\mathbb{R})$ 
and one of the conditions 
\begin{equation}\label{eq2.3}
\{f\circ \Phi, g\circ\Phi\}_1
=\{f, g\}_2\circ\Phi
\text{ for all }f,g\in \mathcal{P}^\infty(P_2,\mathbb{R})
\end{equation}
and 
\begin{equation}\label{eq2.4}
(T_{p_1}\Phi) ((\#_1(f\circ \Phi))(p_1))
=(\#_2f)(\Phi(p_1))
\text{ for all }p_1\in P_1\text{ and }f\in \mathcal{P}^\infty(P_2,\mathbb{R}).
\end{equation}
For arbitrary $g\in \mathcal{P}^\infty(P_2,\mathbb{R})$, 
applying $\de g(\Phi(p_1))\in T_{\Phi(p_1)}^*P_2$ to both sides of \eqref{eq2.4}, 
we find that \eqref{eq2.4} is equivalent to 
\begin{equation}\label{eq2.5}
(\#_1(f\circ\Phi))(g\circ\Phi)=(\#_2f)(g)\circ\Phi
\end{equation}
where $f\in \mathcal{P}^\infty(P_2,\mathbb{R})$. 
Hence, assuming that $g\in \mathcal{P}^\infty(P_2,\mathbb{R})\subseteq C^\infty(P_2,\R)$, we obtain~\eqref{eq2.3}.  
Consequently, the condition~\eqref{eq2.3} is less restrictive than~\eqref{eq2.4}, and in the case when $\mathcal{P}^\infty(P_k,\mathbb{R})=C^\infty(P_k,\mathbb{R})$ for $k=1,2$ the conditions \eqref{eq2.3} and \eqref{eq2.4} are equivalent. 
In the following we will use the less restrictive condition \eqref{eq2.3} as the definition of a Poisson map. 

Let us note that the mapping 
$$\mathcal{P}^\infty(P,\mathbb{R})\times C^\infty(P,\mathbb{R})\to C^\infty(P,\mathbb{R}),\quad 
(f,g)\mapsto (\# f)(g)$$
is an action of the Lie algebra $(\mathcal{P}^\infty(P,\mathbb{R}),  \{\cdot,\cdot\})$ by derivations of the associative algebra $C^\infty(P,\mathbb{R})$. 
The definition of a Poisson map $\Phi\colon P_1\to P_2$ in the sense of~\eqref{eq2.4} ensures that the pull-back mapping 
$\Phi^*\colon C^\infty(P_2,\mathbb{R})\to C^\infty(P_1,\mathbb{R})$, 
$g\mapsto g\circ\Phi$, gives a morphism between the corresponding actions of 
$\mathcal{P}^\infty(P_2,\mathbb{R})$ and 
$\mathcal{P}^\infty(P_1,\mathbb{R})$ 
on $C^\infty(P_2,\mathbb{R})$ and $C^\infty(P_1,\mathbb{R})$, respectively.  

Two important subclasses of the category of \Banach Poisson manifolds in the above sense will be discussed in the next two subsections.

 \subsection{Weakly symplectic \Banach manifolds}
\label{Subsect2.2}

Let $P$  be a \Banach manifold in the sense of the definition presented in \cite{Bou, Lang}  with a fixed  weakly symplectic form $\omega\in \Gamma^\infty\bigwedge ^2 T^*P$ on it. 
By definition, see e.g.,\cite{ChM, OR03}, $\omega$ is a weakly symplectic form if it is closed and non-singular, i.e., $\de\omega=0$ and the identity-covering bundle map $\hat{\flat}:TP\to T^*P$ defined by 
\begin{equation*} 
T_pP\ni \xi_p\to \hat{\flat}_p(\xi_p):=\omega_p(\xi_p,\cdot)\in T^*_pP
\end{equation*}
for $p\in P$, satisfies $\Ker\hat{\flat}:=\{\xi_p\in T_pP:\ \hat{\flat}(\xi_p)=0\}=\{0\}$. 
Note here that one does not assume that $\hat{\flat}_p(T_pP)$ is a closed subspace of $T_pP$.

\begin{defn}\label{def:2.3}
\begin{enumerate}[{\rm(i)}]
\item A \Banach manifold $(P,\omega)$ endowed with a weakly symplectic form $\omega$ is called a \emph{weakly symplectic manifold}.
\item If $\hat{\flat}_p(T_pP)=T^*_pP$ for each $p\in P$ then $(P,\omega)$ is called \emph{strongly symplectic manifold}.
\end{enumerate}
\end{defn}

We now discuss when a weakly symplectic manifold $(P,\omega)$ defines a Poisson structure on $P$ in the sense of Definition \ref{def21}. 

In this case the role of $\mathcal{P}^\infty(P,\mathbb{R})$ is played by  $\mathcal{P}^\infty_\omega(P,\mathbb{R}):=\{f\in C^\infty(P,\mathbb{R}):\ \de f\in \Gamma^\infty \hat{\flat}(TP)\}$. 
Since  
$\hat{\flat}:TP\to T^*P$ is an identity-covering injective bundle morphism,  one defines the map $\#:\mathcal{P}^\infty(P,\mathbb{R})\to TP$ by the equality 
\begin{equation}
\label{omega1} \omega (\#f,\cdot)=\de f
\end{equation}
and the Poisson bracket $\{f,g\}_\omega$ of $f,g\in \mathcal{P}^\infty(P,\mathbb{R})$ by 
\begin{equation*} 
%\label{omega2}
\{f,g\}_\omega:=\omega(\#f,\#g)=-(\#f)(g)=(\#g)(f).
\end{equation*}
The Jacobi identity for $\{\cdot,\cdot\}_\omega$ follows from 
\begin{equation*} 
0= \de\omega(\#f, \#g,\#h)=3(\{f,\{g,h\}\}+\{h,\{f,g\}\}+\{g,\{h,f\}\}).
\end{equation*}
One obtains by \eqref{omega1} that $\mathcal{P}^\infty(P,\mathbb{R})$ separates the vectors of the tangent bundle $TP$ if and only if  $S_\omega^\perp=\{0\}$, where $S_\omega^\perp$ is defined by 
$$S_\omega^\perp:=\{\xi_p\in T_pP: \omega_p(\#f,\xi_p)=0 
\text{ for all }f\in \mathcal{P}^\infty(P,\mathbb{R})\},$$ 
i.e., this is the symplectic orthogonal of the characteristic distribution $S_\omega$.

One can summarize the above discussion as follows: 

\begin{prop}\label{wP} 
If $(P,\omega)$ is a weakly symplectic manifold with $S_\omega^\perp\equiv 0$, 
then the triple 
$$(\mathcal{P}^\infty(P,\mathbb{R}), \{\cdot,\cdot\}_\omega,\#_\omega)$$ 
defines a Poisson structure on $P$ in the sense of Definition~\ref{def21}.
\end{prop}

\begin{rem} 
\begin{enumerate}[{\rm(i)}] 
\item It follows from the non-singularity of $\omega$ that if the fibres $S_p\subseteq T_pP$ of the characteristic distribution are dense in $T_pP$, for $p\in P$, then $S_\omega^\perp=\{0\}$ .
\item If $(P,\omega)$ is a strongly symplectic manifold then $(C^\infty(P,\mathbb{R}), \{\cdot,\cdot\}_\omega,\#_\omega)$ defines a Poisson structure in sense of Definition \ref{def21}.
\end{enumerate}
\end{rem}

\begin{example} 
	Let $P:=\mathfrak{b}\times \mathfrak{b}^*$, where $\mathfrak{b}$ is a Banach space that need not be reflexive. 
	Since in this case the bundles $T(\mathfrak{b}\times \mathfrak{b}^*)$ and  $T^*(\mathfrak{b}\times \mathfrak{b}^*)$ are trivial one can define the differential $2$-form $\omega_0$ on $\mathfrak{b}\times \mathfrak{b}^*$ in the following way
\begin{equation*} 
\omega_0(b,\varphi)((b,\varphi, \dot{b}_1,\dot{\varphi}_1),(b,\varphi, \dot{b}_2,\dot{\varphi}_2)):=\langle \dot{\varphi}_1,\dot{b}_2\rangle-\langle \dot{\varphi}_2,\dot{b}_1\rangle
\end{equation*}
where $(b,\varphi, \dot{b}_1,\dot{\varphi}_1),(b,\varphi, \dot{b}_2,\dot{\varphi}_2)\in T_{(b,\varphi)}(\mathfrak{b}\times \mathfrak{b}^*)\cong \{(b,\varphi)\}\times\mathfrak{b}\times \mathfrak{b}^*$ are the tangent vectors to $\mathfrak{b}\times \mathfrak{b}^*$ at $(b,\varphi)\in \mathfrak{b}\times \mathfrak{b}^*$. Since $\omega_0$ is constant on $\mathfrak{b}\times \mathfrak{b}^*$ and the Banach spaces $\mathfrak{b}$ and $\mathfrak{b}^*$  separate elements of each other, i.e. $\langle\varphi, \cdot\rangle=0$ if and only if $\varphi=0$ and $\langle\cdot, b\rangle=0$  if and only if  $b=0$, so, the differential $2$-form $\omega_0$ is closed $d\omega_0=0$ and non-singular. 
Thus $(\mathfrak{b}\times \mathfrak{b}^*, \omega_0)$ is a weak symplectic manifold.

Let $\xi(b,\varphi)=(b,\varphi, \xi^{\mathfrak{b}}(b,\varphi),  \xi^{\mathfrak{b}^*}(b,\varphi)))\in \{(b,\varphi)\}\times\mathfrak{b}\times \mathfrak{b}^*$ be a vector tangent to $\mathfrak{b}\times \mathfrak{b}^*$ at $(b,\varphi)\in \mathfrak{b}\times \mathfrak{b}^*$. 
The equality \eqref{omega1} taken for $\omega_0$ gives
\begin{equation*} \langle \xi^{\mathfrak{b}^*}(b,\varphi), \dot{b}_2\rangle-\langle\dot{\varphi}_2,\xi^{\mathfrak{b}}(b,\varphi)
\rangle=\langle\frac{\partial f}{\partial b}(b,\varphi), \dot{b}_2\rangle+\langle\frac{\partial f}{\partial \varphi}(b,\varphi), \dot{\varphi}_2\rangle
\end{equation*}
for any $(b,\varphi, \dot{b}_2,\dot{\varphi}_2)\in \{(b,\varphi)\}\times\mathfrak{b}\times \mathfrak{b}^*$. 
Thus one obtains for $f\in \mathcal{P}^\infty_{\omega_0}(\mathfrak{b}\times \mathfrak{b}^*, \mathbb{R})$
\begin{equation*} -\frac{\partial f}{\partial \varphi}(b,\varphi)=\xi_f^{\mathfrak{b}}(b,\varphi)\in \mathfrak{b}\subset \mathfrak{b}^{**} 
\quad \text{and} \quad 
\frac{\partial f}{\partial b}(b,\varphi)=\xi_f^{\mathfrak{b}^*}(b,\varphi)\in \mathfrak{b}^*.
\end{equation*}
Therefore, the linear morphism $\#_{\omega_0}:\mathcal{P}^\infty_{\omega_0}(\mathfrak{b}\times \mathfrak{b}^*, \mathbb{R})\to \Gamma^\infty T(\mathfrak{b}\times \mathfrak{b}^*)$ can be written as follows:
\begin{equation*} \#_{\omega_0}f=\langle\frac{\partial}{ \partial b}\ \cdot\  ,\frac{\partial f}{\partial \varphi}\rangle-\langle\frac{\partial f}{ \partial b},\frac{\partial }{\partial \varphi}\ \cdot\ \rangle.
\end{equation*}
Noting that the functions $f_{(b_0,\varphi_0)}(b,\varphi)$ where $(b_0,\varphi_0)\in \mathfrak{b}\times \mathfrak{b}^*$, defined by 
$$f_{(b_0,\varphi_0)}(b,\varphi):=\langle b,\varphi_0\rangle+\langle b_0,\varphi\rangle,$$
 belong to $\mathcal{P}^\infty_{\omega_0}(\mathfrak{b}\times \mathfrak{b}^*, \mathbb{R})$ we find that  $S_{\omega_0}=T(\mathfrak{b}\times \mathfrak{b}^*)$. So, $\mathcal{P}^\infty_{\omega_0}(\mathfrak{b}\times \mathfrak{b}^*, \mathbb{R})$
separates tangent vectors of the tangent bundle $T(\mathfrak{b}\times \mathfrak{b}^*)$. 
Thus, summarizing, we see that one has the Poisson structure on $(\mathfrak{b}\times \mathfrak{b}^*,\omega_0)$ in the sense of  Definition \ref{def21}. 
Let us also mention that the sub-bundle $T^*\mathfrak{b}\times T_*(\mathfrak{b}^*)\subseteq T^*(\mathfrak{b})\times T^*\mathfrak{b}^*\cong T^*(\mathfrak{b}\times \mathfrak{b}^*)$ is the co-characteristic distribution $S^*_{\omega_0}$ for the above Poisson structure, 
where we have used the sub-bundle $T_*(\mathfrak{b}^*)=\mathfrak{b}^*\times\mathfrak{b}\to \mathfrak{b}^*$ 
of $T^*(\mathfrak{b}^*)$.
\end{example}

 \subsection{Lie-Poisson spaces}
\label{Subsect2.3}

In this subsection we will discuss real  Banach spaces $\mathfrak{b}$ whose duals $\mathfrak{b}^*$ are Banach-Lie algebras $(\mathfrak{b}^*,[\cdot,\cdot])$  
satisfying the condition 
\begin{equation}
\label{ad} 
\ad_x^*\mathfrak{b}\subset\mathfrak{b}
\end{equation} 
for arbitrary $x\in \mathfrak{b}^*$. 
Here we regard $\mathfrak{b}$ as a 
closed linear subspace of $\mathfrak{b}^{**}$ via the canonical linear isometric embedding $\mathfrak{b}\hookrightarrow \mathfrak{b}^{**}$. 
The above properties of $\mathfrak{b}$ allow one to define the map $\#:C^\infty(\mathfrak{b},\mathbb{R})\to \Gamma^\infty T\mathfrak{b}$ by 
\begin{equation}
\label{ad2} 
(\#f)(b):=-\ad^*_{\de f(p)}b,
\end{equation}
where $b\in \mathfrak{b}$, and the Poisson bracket 
$\{f,g\}$ of $f,g\in C^\infty(\mathfrak{b}, \mathbb{R})$ by 
\begin{equation}\label{ad3} 
\{f,g\}(b)=\langle b,[\de f(b),\de g(b)]\rangle=-((\#f)(g))(b)=((\#g)(f))(b).
\end{equation}
We note here that $\de f(b), \de g(b)\in \mathfrak{b}^*$.

The condition $\langle \xi_b,\de g(b)\rangle=0$ on a tangent vector $\xi_b\in T_b\mathfrak{b}$ fulfiled for any $g\in C^\infty(\mathfrak{b}, \mathbb{R})$ implies $\xi_b=0$. In order to see this it is enough to check it for the linear functions $g_x(b)=\langle b,x\rangle$, where $x\in \mathfrak{b}^*$, and notice that $\mathfrak{b}^*$ separates elements of $\mathfrak{b}$.

Hence we see that  $(\mathfrak{b}, C^\infty(\mathfrak{b}, \mathbb{R}), \{\cdot,\cdot\},\#)$, where $\#$ and $\{\cdot,\cdot\}$ are defined in 
 \eqref{ad2} and \eqref{ad3}, respectively, satisfies the condition of Definition \ref{def21}. 
 Therefore one has the structure of a \Banach Poisson manifold on $\mathfrak{b}$.

The Banach spaces endowed with a structure of this type are called 
\emph{\Banach Lie-Poisson spaces}. 
We refer to \cite{OR03} for the general theory of \Banach Lie-Poisson spaces.
The category of \Banach Lie-Poisson spaces, morphisms of which are continuous linear Poisson maps, is a subcategory of the category of \Banach Poisson manifolds in sense of Definition \ref{def21}.

One can reformulate the definition of a \Banach Lie-Poisson space  starting  from a Banach-Lie algebra $(\mathfrak{g},[\cdot,\cdot])$ which has a predual Banach space $\mathfrak{g}_*\subset \mathfrak{g}^{*}$ such that $\ad^*_x\mathfrak{g}_*\subseteq \mathfrak{g}_*$ for all $x\in \mathfrak{g}$. 
That is, one puts $\mathfrak{b}=\mathfrak{g}_*$ and $\mathfrak{b}^*=\mathfrak{g}$, and one defines $\#:C^\infty(\mathfrak{g}_*, \mathbb{R})\to \Gamma^\infty T\mathfrak{g}_*$ and the Poisson bracket on $C^\infty(\mathfrak{g}_*, \mathbb{R})$ by \eqref{ad2} and \eqref{ad3}, respectively.

A Banach-Lie algebra $\mathfrak{g}$ may have several non-isomorphic predual Banach spaces $\mathfrak{g}_{*1}\not\cong\mathfrak{g}_{*2}$ preserved by coadjoint action, see e.g. \cite[Cor. 2.7.8]{BL04}. 
So, the category of \Banach Lie-Poisson spaces is not a subcategory of Banach-Lie algebras possessing preduals, see \cite[Th. 4.6]{OR03}.  Here we present  the following statement from  \cite{OR03} which will be useful in the following. 

\begin{prop}\label{prop:2.6}
Let $(\mathfrak{b}_1,\{\cdot,\cdot\})$ be a \Banach Lie-Poisson space and let $\pi:\mathfrak{b}_1\to\mathfrak{b}_2$ be a continuous linear surjective map onto the Banach space $\mathfrak{b}_2$.  
Then $\mathfrak{b}_2$ carries the
\Banach Lie-Poisson structure coinduced by $\pi$ if and only if 
$\pi^*(\mathfrak{b}_2^*)\subseteq \mathfrak{b}_1^*$ is closed under
the Lie bracket $[\cdot,\cdot]_1$ of $\mathfrak{b}_1^*$.  The map $\pi^*:\mathfrak{b}_2^*\to \mathfrak{b}_1^*$ is a Banach-Lie algebra morphism
whose dual $\pi^{**}:\mathfrak{b}_1^{**}\to \mathfrak{b}_2^{**}$ maps $\mathfrak{b}_1$ into $\mathfrak{b}_2$.
\end{prop}

Let us also mention that symplectic leaves of \Banach Lie-Poisson spaces are weakly symplectic manifolds 
in sense of Subsection~\ref{Subsect2.2}, see  \cite[Thms.  7.3, 7.4, and 7.5]{OR03} and \cite[Cor. 2.10]{BR05}. 
We now explain for later use that these symplectic manifolds have Poisson structures in the sense of Definition~\ref{def21}. 

\begin{rem}[Kirillov-Kostant-Souriau construction]\label{Kir}
\normalfont
Let $\bg$ be a \Banach Lie-Poisson space 
and assume that the \Banach Lie algebra $\gg:=\bg^*$ is integrable, that is, 
there exists a connected \Banach Lie group $G$ whose Lie algebra is~$\gg$. 
Since $G$ is connected and $\bg$ is a \Banach Lie-Poisson space, 
it is straightforward to prove, using \eqref{ad}, that 
$\Ad^*(G)\bg\subseteq\bg$. 
We also assume that $\rho_0\in\bg$ has the property that its coadjoint isotropy group 
$G_{\rho_0}:=\{g\in G:\Ad_G^*(g)\rho_0=\rho_0\}$ is a \Banach Lie subgroup of $G$, 
that is, $G_{\rho_0}$ is a \Banach Lie group with respect to its topology induced from $G$, 
and moreover its Lie algebra $\gg_{\rho_0}:=\{X\in\gg: \ad_{\gg}^*(X)\rho_0=0\}$ is a closed subspace of $\gg$ 
for which there exists a closed linear subspace $\Y\subseteq\g$ satisfying the direct sum decomposition 
$\gg_{\rho_0}\dotplus\Y=\gg$. 
Then the coadjoint orbit 
$$\Oc_{\rho_0}:=\Ad_G^*(G)\rho_0\subseteq\bg$$ 
has the structure of a \Banach manifold endowed with a $G$-invariant weakly symplectic structure  
$(\Oc_{\rho_0},\omega)$ for which the mapping 
$G/G_{\rho_0}\to\Oc_{\rho_0}$, $gG_{\rho_0}\mapsto\Ad_G^*(g)\rho_0$ is a diffeomorphism. 
(See for instance \cite[Th. 7.3]{OR03} and also \cite[Ex. 4.31]{Be06}.) 

Now, for arbitrary $X\in\gg$, define $f^X\colon\Oc_\rho\to\R$, $f^X(\rho):=\langle\rho,X\rangle$, 
where we recall that $\Oc_{\rho_0}\subseteq\bg\subseteq\gg^*$. 
Then it is easily seen that 
$$\#(f^X):=\frac{\de}{\de t}\big\vert_{t=0}\Ad_G^*(\exp_G(tX))\vert_{\Oc_{\rho_0}}$$ 
is a vector field on $\Oc_{\rho_0}$ satisfying \eqref{omega1} for $f=f^X$. 
This shows that $$\{(\#f)_\rho:f\in\P_\omega^\infty(\Oc_{\rho_0},\R)\}=T_\rho(\Oc_{\rho_0})$$ 
for every $\rho\in\Oc_{\rho_0}$. 
Then clearly $S_\omega^\perp=\{0\}$, hence Proposition~\ref{wP} applies and gives us a Poisson structure on~$\Oc_{\rho_0}$ (in sense of Definition \ref{def21}) for which the inclusion map $\Oc_{\rho_0}\hookrightarrow\bg$ is a Poisson map.  
\end{rem}

Ending this subsection, we describe a subcategory of the category  of \Banach Lie-Poisson  spaces 
that is naturally related to the category of $W^*$-algebras (von Neumann algebras). 
A $W^*$-algebra $\M$ is by definition a $C^*$-algebra that has a predual Banach space $\M_*$, i.e. $\M=(\M_*)^*$. 
The predual $\M_*\subset\M^*$ is uniquely determined by the algebraic structure  of $\M$ 
as the space of normal functionals on $\M$, see for instance \cite{LP2}. 
The $W^*$-algebra  $\M$ is an associative Banach algebra, hence it is a complex Banach-Lie algebra $(\M,[\cdot,\cdot]), $ where $[x,y]:=xy-yx$ for $x.y\in \M$, and it also satisfies $\ad^*_x\M_*\subseteq \M_*$. So, $(\M_*,C^\infty(\M_*, \mathbb{C}), \{\cdot,\cdot\}_{\LP},\#)$ with $\#$ and $\{\cdot,\cdot\}_{\LP}$ defined in \eqref{ad2} and \eqref{ad3} 
 defines a complex \Banach Lie-Poisson structure on $\M_*$.

Let us recall here that the adjoint $\M_*\ni\varphi\mapsto \varphi^*\in \M_*$ of $\varphi$ is defined by the equality 
\begin{equation*}
\langle \varphi^*,x\rangle:=\ol{\langle \varphi,x^*\rangle}
\end{equation*}
satisfied for arbitrary $x\in \M$. 
The hermitian  part $\M_{*}^{\rm h}$ consists of the self-adjoint elements of $\M_*$, 
that is, $\varphi\in\M_{*}$ satisfying $\varphi=\varphi^*$.
 Let us also  note that the anti-hermitian part $\M^a:=\{x\in \M:  x+x^*=0\}$ is a real Banach-Lie algebra with respect to the commutator  $[\cdot,\cdot]$ defined above. The real Banach space $\M_{*}^{\rm h}$ is the predual of the real Banach-Lie algebra $\M^a$ and $i\de f(\varphi),i\de g(\varphi)\in \M^a$ for all $f,g\in C^\infty(\M_{*}^{\rm h}, \mathbb{R})$.

In the following we will be interested in the real \Banach Lie-Poisson structure 
on the Hermitian part $\M_{*}^{\rm h}$ of $\M_*$ which is defined by the Lie-Poisson bracket
\begin{equation} \label{bracketi}
\{f,g\}_{\LP}(\varphi):=i\langle\varphi,[i\de f(\varphi),i\de g(\varphi)]\rangle
=-i\langle\varphi,[\de f(\varphi),\de g(\varphi)]\rangle
\end{equation}
 of $f,g\in C^\infty(\M_{*}^{\rm h}, \mathbb{R})$, where $\varphi\in \M_{*}^{\rm h}$, see the definition \eqref{ad3}.

In particular case when $\M=L^\infty(\H)$ the  predual of $\M$ is $\M_*=L^1(\H)$ and the Lie-Poisson bracket \eqref{ad3} of $f,g\in C^\infty(L^1(\H),\mathbb{C})$ assumes the following form 
\begin{equation}
\label{LPbracket} 
\{f,g\}(\rho)=\Tr(\rho[\de f(\rho),\de g(\rho)])
\end{equation}
where $\rho\in L^1(\H)$. For the hermitian case the Lie-Poisson bracket \eqref{bracketi} 
	 of $f,g\in C^\infty(L^1(\H)_h,\mathbb{R})$ is defined as follows
		\begin{equation}
		\label{LPbracketi} 
		\{f,g\}_{\LP}(\rho):=i \Tr \rho[\de f(\rho),\de g(\rho)],
		\end{equation}
		where $\rho\in L^1(\H)_h$ and $df(\rho),dg(\rho)\in L^1(\H)^*_h\cong L^\infty(\H)_h$, 
		see \cite{OR03}. 
		%see [OR].  
		By $ L^1(\H)_h$ and $\ L^\infty(\H)_h$ we denote the hermitian part of  $L^1(\H)$ and $\ L^\infty(\H)$, respectively.

 \subsection{The complex Lie groupoid $\G(\M)\tto \Ll(\M)$ of partially invertible elements}
\label{Subsect2.4}

 The element $x\in \M$ is called \emph{partially invertible} if  $|x|\in \M^+$ defined by the polar decomposition 
\begin{equation}
\label{polar0} x=u|x|
\end{equation}
of $x$ satisfies $|x|\in G(p\M p)$, where $p=u^*u$ is the support $\bs(|x|)=p$ of  $|x|$ and $u\in \U(\M)$. 
By $G(p\M p)$ we have denoted the group of invertible elements of $W^*$-subalgebra $p\M p\subset \M$ and by $\U(\M)$ the set of all partial isometries of $\M$.

On the set  $\G(\M)$ of the partially invertible elements one can define the groupoid structure in a natural way. 
Namely, the source and target maps for this structure are 
the right and left  support maps  $\br:\G(\M)\to \Ll(\M)$ and   $\bl:\G(\M)\to \Ll(\M)$  
defined by $\br(x):=u^*u$ and $\bl(x):=uu^*$, 
where $u\in \U(\M)$ is defined by \eqref{polar0} and $\Ll(\M)$ is the lattice of the orthogonal projections in $\M$.

The groupoid product  of $x,y\in \G(\M)$ defined on  $\G(\M)* \G(\M):=\{(x,y)\in \G(\M)\times \G(\M):\ \br(x)=\bl(y)\}$ is given as the algebraic product $xy$ in $\M$.

The inverse map $\bi:\G(\M)\to \G(\M)$ is defined as follows
\begin{equation*}
%\label{inv} 
\bi(x):=|x|^{-1}u^*
\end{equation*}
and object inclusion map  $\bepsilon:\Ll(\M)\to \G(\M)$  by definition is the sets inclusion map $\Ll(\M)\hookrightarrow \G(\M)$. 
We will henceforth use the notation $x^{-1}:=\bi(x)$ for all $x\in\G(\M)$. 

We refer to \cite{OS} for checking that the maps defined above satisfy the conditions to be structural maps of the groupoid $\G(\M)\tto \Ll(\M)$ of partially  invertible elements.
We now recall from \cite{OS} the structure of complex \Banach Lie groupoid on $\G(\M)\tto \Ll(\M)$. 
Namely, for $p\in \Ll(\M)$ the subset $\Pi_p\subset\Ll(\M)$ by definition  consists of all $q\in\Ll(\M)$ for which one has the direct sum decomposition 
 \begin{equation*}
 %\label{split}
 \M=q\M\oplus (1-p)\M,
 \end{equation*}
 where we note that both summands in the right-hand side are right $W^*$-ideals of~$\M$. 
Using this splitting, one decomposes the projection $p\in \Ll(\M)$ as 
 \begin{equation}\label{rozklad} 
 p=x_p-y_p
 \end{equation}
 and in this way one obtains a bijection
\begin{equation}\label{varphi}
\Pi_p\ \ni\ q\mapsto \varphi_p(q):=(pq)^{-1}-p=y_p\ \in\ (1-p)\M p,
\end{equation}
and a local section
\begin{equation}\label{sigma}
\Pi_p\ \ni\ q\mapsto \sigma_p(q):= (pq)^{-1}=x_p\in\  \bl^{-1}(\Pi_p)
\end{equation}
of the target map $\bl:\G(\M)\to \Ll(\M)$. 
Let us mention here that\footnote{
	Proof of ``$\supseteq$'': 
	If $pq\in \G(\M)^p_q$ then there exists $x\in\M$ with $pqx=p$ and $xpq=q$. 
	Then $\M=q\M+(\1-p)\M$ because of $\1=xp+(\1-xp)$ with 
	$xp=x(pqx)=(xpq)x=qx\in q\M$ and $\1-xp\in(\1-p)\M$. 
	(This last property follows from $p(xp)=p(qx)=p$, which implies 
	$(\1-p)(\1-xp)=\1-xp$.) 
	Moreover $q\M\cap(\1-p)\M=\{0\}$ since, if $a,b\in \M$ and $qa=(\1-p)b$, 
then $qa=(xpq)a=xp(qa)=xp((\1-p)b)=0$. 

Proof of ``$\subseteq$'': 
In fact we show that $x_p=(pq)^{-1}$, that is, 
$(pq)x_p=p$ and $x_p(pq)=q$. 
One has  $pq=(x_pq)-(y_pq)$ by \eqref{rozklad}, while $pq=q+(\1-p)q$, hence $x_pq=q$. 
Also $x_p-y_p=p=p^2=(x_pp)-(y_pp)$ by \eqref{rozklad}, hence $x_p=x_pp$. 
Then $x_ppq=x_pq=q$. 
Furthermore, since $x_p\in q\M$ and $y_p\in(\1-p)\M$, one has $pqx_p=px_p=p(p+y_p)=p^2+py_p=p$.} 
$$\Pi_p=\{q\in\Ll(\M): pq\in \G(\M)^p_q\}$$
where 
$
\G(\M)^p_q:=\bl^{-1}(p)\cap \br^{-1}(q)$.

The charts $(\Pi_p, \varphi_p)$, where $p\in  \Ll(\M)$,
define a complex \Banach manifold structure on the lattice $\Ll(\M)$, as shown in \cite{OS}. 
The transition map $\varphi_{p'}\circ \varphi_p^{-1}:\varphi_{p}(\Pi_{p'}\cap\Pi_p)\to \varphi_{p'}(\Pi_{p'}\cap\Pi_p)$ between $(\Pi_p, \ \varphi_p)$ and $(\Pi_{p'}, \ \varphi_{p'})$ is the following:
\begin{equation}\label{trans} y_{p'}=(\varphi_{p'}\circ\varphi^{-1}_{p})(y)=(b+dy_p)(a+cy_p)^{-1},
\end{equation}
where $a=p'p$, $\ b=(1-p')p$, $\ c=p'(1-p)$ and $\ d=(1-p')(1-p)$.

The complex \Banach manifold structure on $\G(\M)$ is given by the  charts:
\begin{align} 
&\Omega_{p\tilde{p}}:=\bl^{-1}(\Pi_{p})\cap  \br^{-1}(\Pi_{\tilde{p}})\not=\emptyset, \nonumber \\
& \label{psippmap}\psi_{p\tilde{p}}:\Omega_{p\tilde{p}}\to (1-p)\M p\oplus p\M \tilde{p}\oplus(1-\tilde{p})\M\tilde{p},
 \end{align} 
where $( p,\tilde p)\in \Ll(\M)\times \Ll(\M)$ and \eqref{psippmap} is  defined for  $x\in \Omega_{p\tilde{p}}$ by
 \begin{equation}
 \label{psipp}\psi_{p\tilde{p}}(x)=(y_p,z_{ p \tilde p},\tilde y_{\tilde{p}}):=\left(\varphi_{p}(\bl(x)),(\sigma_{p}(\bl(x)))^{-1}x \sigma_{\tilde{p}}(\br(x)),
\varphi_{\tilde{p}}(\br(x)) \right).
\end{equation} 
Let us  note that $z_{p\tilde{p}}\in \G(\M)^p_{\tilde{p}}$ and $\G(\M)^p_{\tilde{p}}$ is an open subset of $p\M\tilde{p}$.
We also note that the chart map \eqref{psippmap} is a bijection from its domain $\Omega_{p\tilde{p}}$  onto the open subset $(1-p)\M p\oplus \G(\M)^p _{\tilde p}\oplus(1-\tilde{p})\M\tilde{p}$ of the Banach space $(1-p)\M p\oplus p\M \tilde{p}\oplus(1-\tilde{p})\M\tilde{p}$.

The transition  map 
 \begin{equation*}
(y_{p'},z_{p'\tilde p'},\tilde{y}_{\tilde p'})=(\psi_{p'\tilde{p'}}\circ\psi_{p\tilde{p}}^{-1})(y_p,z_{p\tilde{p}},\tilde{y}_{\tilde p})
\end{equation*} 
between two charts is given by the formulas 
\begin{align}\label{atlas2}
%\begin{array}{l} 
y_{p'}& =(b+dy_p)(a+cy_p)^{-1}, \nonumber\\
z_{p'\tilde p'}& =(a+cy_p)z_{p\tilde p}(\tilde{a}+\tilde{c}\tilde{y}_{\tilde p})^{-1}\\
\tilde{y}_{\tilde p'} &=(\tilde{b}+\tilde {d}\tilde{y}_{\tilde p})(\tilde{a}+\tilde{c}\tilde{y}_{\tilde p})^{-1}, \nonumber
% \end{array}
 \end{align}
where
$\tilde a=\tilde p'\tilde p$, $\ \tilde b=(1-\tilde p')\tilde p$, $\ \tilde c=\tilde p'(1-\tilde p)$ and $\ \tilde d=(1-\tilde p')(1-\tilde p)$. 

The structural maps  of $\G(\M)\tto \Ll(\M)$ satisfy the \Banach Lie  groupoid axioms, as shown in \cite{OS}. 
The underlying topology of the complex \Banach Lie groupoid $\G(\M)\tto \Ll(\M)$ is a Hausdorff topology, \cite{OJS2}.

Using the object inclusion map $\bepsilon:\Ll(\M)\hookrightarrow \G(\M)$ as a momentum map, one defines the inner action
\begin{equation}\label{action} 
\G(\M)*\Ll(\M)\ni (x,p)\mapsto xpx^{-1}\in \Ll(\M)
\end{equation}
of $\G(\M)\tto \Ll(\M)$ on $\Ll(\M)$. 
Note here that 
$$\G(\M)* \Ll(\M):=\{(x,p)\in \G(\M)\times \Ll(\M):\ \br(x)=p\}.$$ 
The orbit $\Ll_{p_0}(\M)$ of $p_0\in \Ll(\M)$ with respect to the inner action \eqref{action} is exactly the set of all projections $p\in \Ll(\M)$  equivalent $p\sim p_0$ to $p_0$
%satisfying  
in sense of Murray-von Neumann~\cite[Def. 2.7.15]{LP2}. 
%[Sakai?]. 
The orbits $\Ll_{p_0}(\M)$ and 
$$\G_{p_0}(\M):=\bl^{-1}(\Ll_{p_0}(\M))\cap \br^{-1}(\Ll_{p_0}(\M))$$ 
are \Banach submanifolds of $\Ll(\M)$ and $ \G(\M)$, respectively. 

The groupoid  $\G_{p_0}(\M)\tto \Ll_{p_0}(\M)$  is therefore a transitive \Banach Lie subgroupoid of $\G(\M)\tto \Ll(\M)$. 
Note here that $\Pi_p\subseteq \Ll_{p_0}(\M)$ if and only if $p\sim p_0$.

Moreover, the \Banach Lie groupoid of partially invertible elements is a disjoint union of its transitive \Banach Lie subgroupoids $\G_{p_0}(\M)\tto \Ll_{p_0}(\M)$, $p_0\in \Ll(\M)$, which are closed-open subgroupoids with respect to the topology defined by the \Banach manifold structure of  $\G(\M)\tto \Ll(\M)$, see \cite{OS}.

 \subsection{The \Banach Lie groupoid  $\U(\M)\tto \Ll(\M)$ of partial isometries}
\label{Subsect2.5}

One has the natural involution $\mathcal{J}:\G(\M)\to \G(\M)$  
defined by 
\begin{equation}
\label{J} \mathcal{J}(x):=\bi(x)^*=\bi(x^*),
\end{equation}
 i.e. $\mathcal{J}^2=\id$.  
The fixed-point set of this involution  
is the groupoid  $\U(\M)\tto \Ll(\M)$ of partial isometries  of $\M$. Note here that $\mathcal{J}$ is an automorphism of  $\G(\M)\tto \Ll(\M)$ when this is regarded as a real \Banach Lie groupoid. 
One easily sees that  $\U(\M)\tto \Ll(\M)$ is a wide subgroupoid  of $\G(\M)\tto \Ll(\M)$. 
 It is also a real \Banach Lie subgroupoid of $\G(\M)\tto \Ll(\M)$ if one consider the last one as a real \Banach Lie groupoid, see \cite{OS}. In order to see the above property we take instead of the chart 
defined in \eqref{psipp}, the following one 
\begin{equation}\label{atlas3}
\Theta_{p\tilde{p}}(x)=(y_p,u_{ p \tilde p},\tilde y_{\tilde{p}}):=\left(\varphi_{p}(\bl(x)),(u_{p}(\bl(x)))^{-1}x u_{\tilde{p}}(\br(x)),
\varphi_{\tilde{p}}(\br(x)) \right),
\end{equation}
where 
\begin{align}
\label{u_p}
u_{p}(\bl(x)):=
& \sigma_p(\bl(x))|\sigma_p(\bl(x))|^{-1}\in \U(\M)_p^{\bl(x)}, \\
\label{u_p_tilde} 
u_{\tilde{p}}(\br(x)):=
&\sigma_{\tilde{p}}(\br(x))|\sigma_{\tilde{p}}(\br(x))|^{-1} \in \U(\M)_{\tilde{p}}^{\br(x)}.
\end{align}
In the coordinates $(y_p,u_{ p \tilde p},\tilde y_{\tilde{p}})$ the involution \eqref{J} is given by 
\begin{equation}\label{Theta} 
(\Theta_{p\tilde{p}}\circ \mathcal{J}\circ (\Theta_{p\tilde{p}})^{-1})(y_p,u_{ p \tilde p},\tilde y_{\tilde{p}})=(y_p,\mathcal{J}(u_{ p \tilde p}),\tilde y_{\tilde{p}}).
\end{equation}
Hence, $x\in \U(\M)\cap\Omega_{ p \tilde p}$  if and only if  $u_{ p \tilde p}\in \U(\M)^p_{\tilde{p}}
\subset \U(\M)$, i.e. $u_{ p \tilde p}^*u_{ p \tilde p}=\tilde p$ and $u_{ p \tilde p}u_{ p \tilde p}^*= p$.

The subset $\G(\M)^p_{\tilde p}\subset p\M \tilde p$ is invariant with respect to the involution, i.e.,  $\mathcal{J}:\G(\M)^p_{\tilde p}\to \G(\M)^p_{\tilde p}$, and $\U(\M)^p_{\tilde p}\subset \G(\M)^p_{\tilde p}$ is its fixed-point set. 
The tangent map $T\mathcal{J}(u_{p\tilde p}):T_{u_{p\tilde p}}\U(\M)^p_{\tilde p}\to T_{u_{p\tilde p}}\U(\M)^p_{\tilde p}$ is thus a linear involution  of the tangent space $T_{u_{p\tilde p}}\G(\M)^p_{\tilde p}$ at $u_{p\tilde p}\in \U(\M)^p_{\tilde p}$. 
One easily obtains, see \cite{OS}, that 
\begin{equation*} 
T\mathcal{J}(u_{p\tilde p})\Delta z_{p\tilde p}=-u_{p\tilde p}(\Delta z_{p\tilde p})^*u_{p\tilde p}
\end{equation*}
for $\Delta z_{p\tilde p}\in p\M\tilde p\cong T_{u_{p\tilde p}}\G(\M)^p_{\tilde p}$. 
Since $T\mathcal{J}(u_{p\tilde p})\circ T\mathcal{J}(u_{p\tilde p})=id$, then the tangent space  $T_{u_{p\tilde p}}\G(\M)^p_{\tilde p}=T_{u_{p\tilde p}} p\M\tilde p\cong \{u_{p\tilde p}\}\times p\M\tilde p$ splits
\begin{equation}\label{split2} 
T_{u_{p\tilde p}}\G(\M)^p_{\tilde p}=T^+_{u_{p\tilde p}}\G(\M)^p_{\tilde p}\oplus T^-_{u_{p\tilde p}}\G(\M)^p_{\tilde p}
\end{equation}
on the closed subspaces which are the eigenspaces of $T\mathcal{J}(u_{p\tilde p})$ corresponding to the eigenvalues $+1$ and $-1$, respectively. 
Note here that 
\begin{equation*} 
T^\pm_{u_{p\tilde p}}\G(\M)^p_{\tilde p}\cong \{\Delta z_{p\tilde p}\in p\M\tilde p:\quad \Delta z_{p\tilde p}u_{p\tilde p}^*\pm(\Delta z_{p\tilde p}u_{p\tilde p}^*)^*=0\}
\end{equation*}
where $\Delta z_{p\tilde p}u_{p\tilde p}^*\in p\M p$, so, one has the isomorphism of Banach spaces
\begin{equation*} T^+_{u_{p\tilde p}}\G(\M)^p_{\tilde p}\cong p\M^a p\quad \text{and} 
\quad  T^-_{u_{p\tilde p}}\G(\M)^p_{\tilde p}\cong p\M^h p,
\end{equation*}
where 
$$\M^a:=\{x\in\M: x^*=-x\}$$ 
and 
$$\M^h:=\{x\in\M: x^*=x\}.$$ 
Fixing a partial isometry $u^0_{p\tilde p}\in \U(\M)^p_{\tilde p}$ one obtains the bijections $\G(\M)^p_{\tilde p}\cong G(\tilde p\M\tilde p)$ and $\U(\M)^p_{\tilde p}\cong U(\tilde p\M\tilde p)$ defined by
\begin{equation*} 
z_{p\tilde p}=u^0_{p\tilde p}\tilde g
\end{equation*}
where $\tilde g\in G(\tilde p\M\tilde p)$ for $\G(\M)^p_{\tilde p}$ and $\tilde g\in U(\tilde p\M\tilde p)$ for $\U(\M)^p_{\tilde p}$. 
Summing up,  
$\G(\M)^p_{\tilde p}$ is a complex \Banach manifold which is holomorphically diffeomorphic to  
the complex \Banach Lie group $G(\tilde p\M\tilde p)$ and $\U(\M)^p_{\tilde p}$  is a real \Banach manifold which is diffeomorphic to  the unitary group $ U(\tilde p\M\tilde p)$. 
It follows from  the splitting \eqref{split2} that $\U(\M)^p_{\tilde p}\hookrightarrow \G(\M)^p_{\tilde p}$  is a real \Banach submanifold of $\G(\M)^p_{\tilde p}$ 
with respect to its underlying real \Banach manifold structure. Hence, taking into account  the atlases defined by charts \eqref{psipp} and \eqref{atlas3} one can state the following proposition. 
(See \cite{OS} and \cite{OJS1} for additional details.)

\begin{prop}
\begin{enumerate}[{\rm(i)}]
\item The groupoid of partial isometries $\U(\M)\tto\Ll(\M)$   
is a 
\Banach Lie groupoid.
\item $\U(\M)\tto\Ll(\M)$ is a 
\Banach Lie subgroupoid of $\G(\M)\tto\Ll(\M)$ 
regarded as a  
real 
\Banach Lie groupoid.
\end{enumerate}
\end{prop}

Just as the groupoid of partially invertible elements of $\M$, 
the groupoid of partial isometries  $\U(\M)\tto\Ll(\M)$ splits into its  transitive open closed subgroupoids $\U_{p_0}(\M)\tto\Ll_{p_0}(\M)$. 

Following \cite{OJS2}, we show  that $\U_{p_0}(\M)\tto\Ll_{p_0}(\M)$ is canonically isomorphic to the gauge groupoid 
$$\frac{P_0\times P_0}{U_0}\tto P_0/U_0$$ 
of the $U_0$-principal  bundle $\bl_0\colon P_0\to \Ll_{p_0}(\M)\cong P_0/U_0$, where 
\begin{equation} \label{P0}
P_0:=\{u\in \U_{p_0}(\M): \br(u)=u^*u=p_0\}=\br^{-1}(p_0) 
\end{equation}
and 
\begin{equation*}  
U_0:=U(p_0\M p_0)
\end{equation*}
and $\bl_0:P_0\to \Ll_{p_0}(\M)$ is the restriction of the left support map~$\bl$ to~$P_0$.
For any $p\in \Ll_{p_0}(\M)$ 
its corresponding local chart on $P_0$ defined by the local chart 
$(\Theta_{pp_0},\U(\M)\cap\Omega_{pp_0})$ of $\U(\M)$ 
has the following form\footnote{
Proof: By \eqref{u_p}, 
	$u_{pp_0}=u_p(\bl(u))^{-1}u
	=\vert\sigma_p(\bl(u))\vert\sigma_p(\bl(u))^{-1}u
	=\vert(p\bl(u))^{-1}\vert p\bl(u)u
	$.
But $\vert a^{-1}\vert=\vert a^*\vert^{-1}$ for all $a\in\G(\M)$ 
hence we further obtain 
$u_{pp_0}=\vert\bl(u)p\vert pu=((pu)(pu)^*)^{1/2}pu=(puu^*p)^{1/2}pu=(puu^*p)^{1/2}u$.}
\begin{align}
\bl^{-1}(\Pi_p)\cap P_0\ni u\mapsto \theta_p(u)
= & (y_p,u_{pp_0})\nonumber  \\
\label{P_0_chart}
= & (u(pu)^{-1}-p,
(puu^* p)^{1/2}u
%pu
) \\ 
& \in (1-p)\M p \times \U(\M)^p_{p_0}. \nonumber 
\end{align}
The inverse   $\theta_p^{-1}$ of $\theta_p$  is given by
\begin{equation}
\label{inv_err}
u=\theta_p^{-1}(y_p,u_{pp_0})=u_p(p+y_p)\cdot u_{pp_0}
\end{equation}
since $p+y_p=x_p=\sigma_p(q)$ with $q=\bl(u)$, hence $p+y_p=\sigma(\bl(u))$, 
and one has the factorization 
\begin{equation}
\label{5July2019}
u=u_p(\bl(u))u_{pp_0}\text{ for all }u\in\Omega_{pp_0}.
\end{equation}
The free actions of $U_0$ an $P_0$ and on $P_0\times P_0$ are by definition 
\begin{equation}\label{actionU1} 
P_0\times U_0\ni (u,g)\mapsto ug\in P_0
\end{equation}
and 
\begin{equation}\label{actionU2} 
P_0\times P_0\times U_0\ni (u,v,g)\mapsto (ug,vg)\in P_0\times P_0.
\end{equation}
The maps 
\begin{equation*} 
P_0\ni u\mapsto 
\bl_0(u):=uu^*\in \Ll_{p_0}(\M)
\end{equation*}
and
\begin{equation*} 
P_0\times P_0\ni(u,v)\mapsto \Phi(u,v):=uv^*\in \U_{p_0}(\M)
\end{equation*}
are constant on the orbits of the actions \eqref{actionU1} and \eqref{actionU2}. 
Thus  one obtains the isomorphism of the real \Banach Lie groupoids 
 \begin{equation}\label{gaugethm}
    \xymatrix{
    	\frac{P_0\times P_0}{U_0} \ar@<-.5ex>[d] \ar@<.5ex>[d] \ar[r]^{[\Phi]} 
    	& \U_{p_0}(\M) \ar@<-.5ex>[d] \ar@<.5ex>[d]\\
    	P_0/U_0 \ar[r]^{[\bl_0]} & \Ll_{p_0}(\M)}
\end{equation}
where $[\Phi]$ and $[\bl_0] $ are defined by quotienting  $\Phi$ and $\bl_0$. 
For details see \cite{OJS1}.

Finally, we point out that one has a naturally defined $p_0\M^a p_0$-valued connection 1-form~$\alpha$ on $U_0$-principal bundle $\bl_0:P_0\to \Ll_{p_0}(\M)$:
\begin{equation}\label{alpha} 
\alpha(u):=u^*\de u
\end{equation}
where $u\in P_0$. For $x\in p_0\M^a p_0$ the fundamental vector field $\xi_x\in \Gamma^\infty TP_0$ is given by 
\begin{equation}
\xi_x(u)=ux.
\end{equation}
Hence one has 
\begin{equation}\label{230} 
\xi_x(u)\llcorner \alpha(u)=u^*ux=p_0x=x.
\end{equation}
From \eqref{actionU1} and \eqref{alpha} one also has 
\begin{equation}\label{alphaequiv} 
\alpha(ug)=g^*\alpha(u)g
\end{equation}
for $g\in U_0$.
So, the form $\alpha$ is indeed a connection 1-form on $P_0$. 
The curvature form $\Omega$ of $\alpha$ is 
\begin{equation}\label{Omega} 
\Omega:=\de\alpha+\frac{1}{2}[\alpha,\alpha]=\de u^*\wedge \de u+\frac{1}{2}[u^*\de u,u^*\de u].
\end{equation}
Note here that the curvature form $\Omega$ also satisfies 
\begin{equation}\label{Omegaequiv} 
\Omega(ug)=g^*\Omega(u)g
\end{equation}
for $g\in U_0$.

In the coordinates $(y_p,u_{pp_0})$ the connection form $\alpha$ 
is expressed as follows
\begin{align*} 
\alpha& =(p+y_p)dy_p^*+(p+y_p)u_{pp_0}du_{pp_0}^*(p+y_p)^* 
\end{align*}
The above formulas will be needed in 
the following sections of this paper.

 \section{Coadjoint action 
	groupoid $\U(\M)*\M_*^+\tto \M_*^+$ as a \Banach Lie groupoid}
\label{Sect3}

In this section we investigate the coadjoint action 
groupoid $\U(\M)*\M_*^+\tto \M_*^+$, which plays a central role in the present paper. 
In Subsection~\ref{Subsect3.1} we discuss the algebraic structure of this groupoid, and in this context it proves useful to introduce the predual groupoid $\M_*\tto \M_*^+$ that is isomorphic to $\U(\M)*\M_*^+\tto \M_*^+$ 
(Proposition~\ref{23August2019}) and 
whose structure is defined in a canonical way by the polar decompositions of elements of $\M_*$. 
Then, in Subsection~\ref{Subsect3.2}, we turn to the study of differential geometric structures related to these groupoids and their orbits, 
one of the main results being that the coadjoint action 
groupoid carries the natural structure of a \Banach Lie groupoid. 
(See Theorem~\ref{coadj_grpd}).

 \subsection{Algebraic structure of the predual/coadjoint action 
	groupoids}
\label{Subsect3.1}

We begin by the definition of the groupoid $\M_*\tto \M_*^+$ 
which
has $\M_*$ as  the space of arrows and the cone $\M_*^+$ of positive elements of $\M_*$ as its base, i.e., the set of its objects.

The left $\sigma_l(\varphi)\in \Ll(\M)$ and right $\sigma_r(\varphi)\in \Ll(\M)$ supports of $\varphi\in \M_*$ are defined as follows. 
Let $[\M\varphi]$ and $[\varphi \M]$ denote the left and right invariant subspaces of $\M_*$ generated from $\varphi\in \M_*$. 
Their annihilators $[\M\varphi]^0\subset \M$ and $[\varphi\M]^0\subset \M$ are right and left $W^*$-ideals in $\M$, 
respectively. 
Thus there exist $e,f\in \Ll(\M)$ such that  $[\M\varphi]^0=e\M$  and   $[\varphi\M]^0=\M f$. 
One then defines  
\begin{equation}
\sigma_l(\varphi):=\1-f\quad  \text{and} \quad\sigma_r(\varphi):=\1-e.
\end{equation}
If $\varphi=\varphi^*$ then we use the notation
\begin{equation}
\sigma_l(\varphi)=\sigma_r(\varphi)=:\sigma_*(\varphi).
\end{equation}
Now let us recall that $y\varphi, \varphi y\in \M_*$, where $y\in \M$ and $\varphi\in \M_*$, are defined by 
\begin{align*} 
\langle y\varphi,x\rangle:= & \langle \varphi,xy\rangle,\\
\langle \varphi y,x\rangle:=& \langle \varphi,yx\rangle, 
\end{align*}
where $x\in \M$.

Any $\varphi\in \M_*$ has the polar decomposition
\begin{equation}\label{polar} 
\varphi=u\vert\varphi\vert,
\end{equation}
%e.g. see [T], 
where $u\in \M$ and $\vert\varphi\vert\in \M^*_+$ are uniquely defined by the condition 
\begin{equation*}
u^*u=\sigma_*(\vert\varphi\vert).
\end{equation*}
Moreover, $\sigma_*(\vert\varphi\vert)=\sigma_r(\varphi)$ and $uu^*=\sigma_l(\varphi)$. 
See for instance \cite[Ch. III, Th. 4.2]{Ta02}.

We now define the following maps:
\begin{enumerate}[{\rm(a)}]
	\item\label{a} the \textsl{source} $\bs_{*}:\M_*\to \M_*^+$ and the \textsl{target} $\bt_{*}:\M_*\to \M_*^+$ maps by
	\begin{equation}\label{st} \bs_{*}(\varphi):=\vert\varphi\vert,\qquad  
	\bt_{*}(\varphi):=u\vert\varphi\vert u^*;
	\end{equation}
	
	\item%\label{b} 
	the\textsl{ product} of $(\varphi_1, \varphi_2)\in \M_* *\M_*:=\{(\varphi_1, \varphi_2):\ \bs_{*}(\varphi_1)=\bt_{*}( \varphi_2)\}$ by
	\begin{equation}
	\label{product} \varphi_1\bullet \varphi_2:=u_1u_2\vert\varphi_2\vert 
	\end{equation}
	where $\varphi_1=u_1\vert\varphi_1\vert$ and $\varphi_2=u_2\vert\varphi_2\vert $ are the respective polar decompositions of $\varphi_1$ and $\varphi_2$;
	
	\item\label{c} 
	the groupoid \textsl{inverse map} by
	\begin{equation*}
	\bi_{*}(\varphi)=\varphi^*;
	\end{equation*}
	
	\item\label{d} the \textsl{object inclusion} map $\bepsilon_{*}:\M_*^+\to \M_*$ is defined as the set inclusion map $\M_*^+\hookrightarrow \M_*$. 
\end{enumerate}

\begin{prop}
	\label{grpd_pred}
		The maps defined in \eqref{a}--\eqref{d} above 
		are the structural maps for the groupoid $\M_*\tto \M_*^+$ called here the \emph{predual groupoid} of $\M$.
\end{prop}

\begin{proof} 
	%(i) 
	Compatibility of the product $\varphi_1\bullet\varphi_2$ with $\bs_{*}$ and $\bt_{*}$ follows from
	\begin{equation*}
	\bs_{*}(\varphi_1\bullet\varphi_2)=|\varphi_2|=\bs_{*}(\varphi_2),
	\end{equation*}
	and 
	\begin{equation*}
	\bt_{*}(\varphi_1\bullet\varphi_2)=u_1u_2|\varphi_2|(u_1u_2)^*=u_1|\varphi_1|u_1^*=\bt_{*}(\varphi_1).
	\end{equation*}
	Associativity of the product \eqref{product} follows from
	\begin{align*}
	(\varphi_1\bullet\varphi_2)\bullet \varphi_3
	& =(u_1u_2|\varphi_2|)\bullet (u_3|\varphi_3|) 
	 =(u_1u_2) u_3|\varphi_3| 
	 =u_1(u_2 u_3)|\varphi_3|
	\\
	&=(u_1|\varphi_1|)\bullet(u_2 u_3|\varphi_3|) 
	 =\varphi_1\bullet(\varphi_2\bullet\varphi_3)
	\end{align*}
	For $\varphi=|\varphi|\in \M_*^+$ one has
	\begin{equation*}
	\bs_{*}(\varphi)=\bt_{*}(\varphi)=\varphi.
	\end{equation*}
	Using \eqref{st} and \eqref{product} one obtains 
	\begin{equation*}
	\varphi\bullet\bepsilon_{*}(\bs_{*}(\varphi))
	=\varphi\bullet \bs_{*}(\varphi)=u|\varphi|=\varphi
	\end{equation*}
	and
	\begin{equation*}\bepsilon_{*}(\bt_{*}(\varphi))\bullet\varphi
	=\bt_{*}(\varphi)\bullet\varphi=(u|\varphi|u^*)\bullet u|\varphi|=u|\varphi|=\varphi.
	\end{equation*}
	In order to prove the consistency properties between $\bi_{*}$, $\bs_{*}$, and $\bt_{*}$, we observe that 
	\begin{align*}
	(s_*\circ\iota)(\varphi)
	&=s_*(\varphi^*)=|\varphi^*|=u|\varphi|u^*=t_*(\varphi),\\
	(t_*\circ\iota)(\varphi)
	&=t_*(\varphi^*)=u^*|\varphi^*|u=u^*u|\varphi|u^*u=|\varphi|=s_*(\varphi),\\
	\iota(\varphi)\bullet\varphi
	&=\varphi^*\bullet \varphi=u^*u|\varphi|=|\varphi|=s_*(\varphi)=\epsilon(s_*(\varphi)),
	\\
	\varphi\bullet\iota(\varphi)
	&=\varphi\bullet \varphi^*=uu^*|\varphi|=|\varphi^*|=\epsilon(t_*(\varphi)).
	\end{align*}
This completes the proof. 
\end{proof}

Furthermore, one has the coadjoint action  
\begin{equation}\label{coaction} 
\U(\M)*\M^+_*\ni(u,\rho)\mapsto u\rho u^*\in \M^+_*,
\end{equation}
of the groupoid $\U(\M)\tto \Ll(\M)$ on $\M^+_*$ for which $\sigma_*:\M^+_*\to \Ll(\M)$ is the momentum map, where 
\begin{equation*} 
\U(\M)*\M^+_*:=\{(u,\rho)\in \U(\M)\times\M^+_*:\ u^*u=\sigma_*(\rho)\}.
\end{equation*}
Any action of a groupoid on a set defines a so-called action groupoid. 
(See for instance \cite[Def. 1.6.10]{Ma05}.)
In particular the coadjoint action \eqref{coaction} of the groupoid $\U(\M)\tto \Ll(\M)$ on $\M^+_*$ defines the coadjoint action groupoid $\U(\M)*\M^+_*\tto \M^+_*$ whose structural maps are as follows. 

\begin{enumerate}[(a)]
\item  The source and target map of $\U(\M)*\M_*^+\tto \M_*^+$ are defined by 
\begin{equation}\label{s} \bs_{*}(u,\rho):=\rho\text{ and } \bt_{*}(u,\rho):=u\rho u^*,
\end{equation}
where $(u,\rho)\in\U(\M)*\M_*^+$. 

\item The product of elements $(u, \rho),(w,\delta) \in  \U(\M)*\M_*^+$,  which are composable in the groupoid sense, i.e. such that $\rho=\bs_{*}(u, \rho)=\bt_{*}(w,\delta)=
w\delta w^*$, is defined as follows
\begin{equation}
\label{9} (u, \rho)\cdot(w,\delta):=(uw,\delta)
\end{equation}
Note here that from $\bs_{*}(u, \rho)=\bt_{*}(w,\delta)$ it follows that  $uw\in \U(\M)$.

\item The groupoid inversion map 
$$\bi_{*}:\U(\M)*\M_*^+\to \U(\M)*\M_*^+$$ 
is  defined by 
\begin{equation}
\label{iota} \bi_{*}(u,\rho):=(u^*,u\rho u^*).
\end{equation}
\item The object inclusion map $\bepsilon_{*}:\M_*^+\to \U(\M)*\M_*^+$ by
\begin{equation}\label{e} 
\bepsilon_{*}(\rho):=(\sigma_*(\rho), \rho).
\end{equation}
We recall that $\sigma_*(\rho)\in \Ll(\M)\subseteq \U(\M)$.
\end{enumerate}
The usual compatibility conditions for the above groupoid structural maps follow by straightforward verification.

\begin{prop}\label{23August2019}
One has the natural isomorphism of groupoids 
\begin{equation}
\label{23August2019_eq1}
\xymatrix{
	\U(\M)\ast \M_*^+ \ar@<-.5ex>[d] \ar@<.5ex>[d] \ar[r]^{\ \Xi} 
	& \M_* \ar@<-.5ex>[d] \ar@<.5ex>[d]\\
	\M_*^+ \ar[r]^{\id} & \M_*^+}
\end{equation}
where $\Xi(u,\rho):=u\rho$ for every $(u,\rho)\in \U(\M)\ast \M_*^+$. 
\end{prop}

\begin{proof}
%(ii) 
This follows from the polar decomposition \eqref{polar}.
\end{proof}

We point out that both groupoids considered above contain by definition the trivial groupoid $\{0\}\tto \{0\}$ as a subgroupoid.

In the sequel we will investigate  these isomorphic groupoids from the perspective of Poisson geometry.

 \subsection[Differential geometric structure]{Differential geometric structure of the coadjoint action groupoid}
\label{Subsect3.2}

We will show in Theorem~\ref{coadj_grpd} below that the coadjoint  action  groupoid  $\U(\M)*\M_*^+\tto \M_*^+$  is a \Banach Lie groupoid.
This fundamental fact provides the framework of the investigation of weakly symplectic structures in Section~\ref{Sect5}. 
 
To start with, we consider the orbit 
\begin{equation}\label{orbit} 
\Oc_{\rho_0}:=\{ u\rho_0 u^*: u^*u=p_0:=\sigma_*(\rho_0)\}
\end{equation}
of the coadjoint action through $\rho_0\in  \M_*^+$.
In order to describe 
%smooth Banach 
the manifold structure of $\Oc_{\rho_0}$  
we make the following lemma on the stabilizer 
\begin{equation} 
		U_{\rho_0}:=\{u\in \U(\M): u\rho_0 u^*=\rho_0\}.
\end{equation}
%of $\rho_0$, 
See also \cite{BR05}.%[5].
		
\begin{lem}\label{stab}
			\begin{enumerate}[{\rm(i)}]
				\item\label{stab_item1}%(i) 
				For every $\rho_0\in \M^+_*$ its stabilizer $U_{\rho_0}$ is a \Banach Lie subgroup of the Lie group $U_0=U(p_0\M p_0)$, where $p_0=\sigma_*(\rho_0)$. 
				
				\item\label{stab_item2}%(ii) 
				The quotient  $U_0/U_{\rho_0}$ with respect to the right action  of $U_{\rho_0}$ on $U_0$ is a smooth \Banach manifold and the canonical projection $U_0\to U_0/U_{\rho_0}$ is a submersion.
			\end{enumerate}
		\end{lem}
	
\begin{proof} 
	We first note that $U_{\rho_0}\subseteq U_0$  
	by the general observation that 
	$$\sigma_*(u\rho_0 u^*)=uu^*$$ 
	for all 
	$u\in\U(\M)$  with $u^*u=\sigma(\rho_0)$. 
	The Assertion~\eqref{stab_item2} follows from the Assertion~\eqref{stab_item1}, see \cite[(5.12.4)]{Bou}.
	It therefore suffices to prove~\eqref{stab_item1}.
		
		We first  note that $\rho_0$ is a faithful normal state on the $W^*$-algebra $p_0\M p_0$. 
		Therefore, according to  %Theorem 2.6. in 
		\cite[Ch. VIII, Th. 2.6]{Ta03a} the Banach-Lie algebra  $$\M_{\rho_0}:=\{x\in p_0\M p_0:\ x\rho_0-\rho_0 x=0\} $$ 
		of 
		$$G_{\rho_0}:=\{g\in G(p_0\M p_0):\ g\rho_0\ g^{-1}=\rho_0\}\subset \M_{\rho_0}$$ 
			is the centralizer of  $\rho_0$ in sense of    
		\cite[Ch. VIII, Def. 2.1]{Ta03a}, that is, 
		\begin{equation} (p_0\M p_0)_{\rho_0}=\{x\in p_0\M p_0:  (\forall t\in\R)\ \sigma_t^{\rho_0}x=x\}
		\end{equation}
		where $\sigma_t^{\rho_0}$ is the modular flow corresponding to $\rho_0$. 
		Thus there exists a conditional expectation $\mathcal{E}_{\rho_0}\colon p_0\M p_0\to (p_0\M p_0)_{\rho_0}=\M_{\rho_0} $ uniquely determined by 
		the condition $\rho_0\circ \mathcal{E}_{\rho_0}=\rho_0\vert_{p_0\M p_0}$ 
		by the Takesaki theorem 
		(see  
		\cite[Ch. IX, Th.~4.2]{Ta03a}). 
		One then obtains the direct sum decomposition 
		\begin{equation}\label{split_E} 
		p_0\M p_0=\M_{\rho_0}\oplus \Ker \mathcal{E}_{\rho_0}
		\end{equation}
		Since the conditional expectation $\mathcal{E}_{\rho_0}$ is a positive map, it preserves the anti-hermitian part $p_0\M^a p_0$ of $p_0\M p_0$. Hence, restricting \eqref{split_E} to $p_0\M^a p_0$ we obtain the splitting
		\begin{equation} 
		\mathfrak{u}_{0}=\mathfrak{u}_{\rho_0}\oplus (\Ker\mathcal{E}_{\rho_0}\cap p_0\M^a p_0)
	     \end{equation}
		of $\mathfrak{u}_0$, where $\mathfrak{u}_{0}$ and $\mathfrak{u}_{\rho_0}$ are Lie algebras of $U_{0}$ and $U_{\rho_0}$, respectively. 
		It thus follows that $U_{\rho_0}$ is a \Banach Lie subgroup of $U_{0}$. 
		\end{proof}
		
		The next statement, based on Lemma~\ref{stab}, describes the geometry of the coadjoint orbit~$\mathcal{O}_{\rho_0}$.
		
		\begin{thm}\label{thm:32} 
			If $\rho_0\in \M_*^+$ and $\sigma_*(\rho_0)=p_0$, then the following assertions hold:
	\begin{enumerate}[{\rm(i)}]
	\item\label{thm:32_item1}%(i) 
	The coadjoint orbit $\mathcal{O}_{\rho_0}$ is the base of a $U_{\rho_0}$-principal bundle as well as the total space of  a bundle as presented in the diagram
 \begin{equation}\label{P03}
    \xymatrix{ P_0 \ar[r]^{\pi_0} \ar[dr]_{\bl_0} & \Oc_{\rho_0} \ar[d]^{\sigma_*} \\
     &\Ll_{p_0}(\M)}
    \end{equation}
		where $\pi_0(u)=u\rho_0 u^*$ for all $u\in P_0$ and  $\bl_0:=\bl\vert_{P_0}$.	
	\item\label{thm:32_item2}%(ii) 
		One has the  identity-covering  bundle isomorphism 
	\begin{equation}\label{P04}
    \xymatrix{P_0\times _{U_0}(U_0/U_{\rho_0}) \ar[r]^{\quad \sim} \ar[d]_{\hat{\bl}_0} & \Oc_{\rho_0} 
    	\cong P_0/U_{\rho_0} 
    	\ar[d]^{\sigma_*}\\
    \Ll_{p_0}(\M)\ar[r]^{\quad \id} & \Ll_{p_0}(\M)	
    }
    \end{equation}
		where $U_0$ acts by the left multiplication on the right quotient $U_0/U_{\rho_0}$.
	\item\label{thm:32_item3} %(iii) 
		The inclusion map $\iota\colon\mathcal{O}_{\rho_0}\hookrightarrow \M_*$ is  smooth and its tangent map $T_\rho\iota\colon T_\rho\Oc_{\rho_0}\to T_{\iota(\rho)}\M_*$ is injective for arbitrary $\rho\in\Oc_{\rho_0}$.
	\end{enumerate}	
	\end{thm}
		
\begin{proof}
\eqref{thm:32_item1}
One has the bijective map
$P_0/ U_{\rho_0}\to\Oc_{\rho_0}$,
$u U_{\rho_0}\mapsto u\rho_0 u^*=\pi_0(u)$. 
It therefore suffices for Assertion~\eqref{thm:32_item1} to study the commutative diagram 
\begin{equation}\label{P03bis}
\xymatrix{ P_0 \ar[r]^{\widetilde{\pi}_0} \ar[dr]_{\bl_0} & P_0/ U_{\rho_0} \ar[d]^{\widetilde{\sigma}_*} \\
	& \Ll_{p_0}(\M) %P_0/ U_0
}
\end{equation}
where $\widetilde{\pi}_0(u):=uU_{\rho_0}$ and $\widetilde{\sigma}_*(uU_{\rho_0}):=uu^*=\bl(u)$, 
which is well defined since $U_{\rho_0}\subseteq U_0$.

Specifically, we endow $P_0/ U_{\rho_0}$ with a manifold structure for which the mapppings $\widetilde{\pi}_0$ and $\widetilde{\sigma}_*$ are projections on the base as indicated in the statement. 
To this end, for arbitrary $p\in\Ll_{p_0}(\M)$ we recall the local chart $$\theta_p\colon \bl^{-1}(\Pi_p)\cap P_0\to (\1-p)\M p\times \U(\M)_{p_0}^p,\quad  \theta_p(u)=(y_p,u_{pp_0})$$
where $y_p=\varphi_p(\bl(u))$, cf. \eqref{P_0_chart} and 
\eqref{rozklad}--\eqref{varphi}.  
Since 
$$\bl(ug)=ug(ug)^*=uu^*=\bl(u)$$ 
for all $u\in P_0$ and $g\in U_0$, it follows that the set $\bl^{-1}(\Pi_p)\cap P_0$ is invariant under the free action from the right of $U_0$ on $P_0$ given by $P_0\times U_0\to P_0$, $(u,g)\mapsto ug$ as in \eqref{actionU1}. 
The expression of this action in the above local chart $\theta_p$ is 
$$((y_p,u_{pp_0}),g)\mapsto (y_p,u_{pp_0}g).$$ 
For any $w\in \U(\M)_{p_0}^p$ we then define a new local chart of $P_0$ by 
$$\theta_p^w\colon \bl^{-1}(\Pi_p)\cap P_0\to (\1-p)\M p\times 
U_0, 
%\U(\M)_{p_0}^p,
\quad  \theta_p^w(u)=(y_p,w^{-1}u_{pp_0})$$
where $w^{-1}u_{pp_0}\in \U(\M)_{p_0}^{p_0}=U_0$. 

We now define 
\begin{equation}
\label{chi}
\chi_{p}^w\colon \widetilde{\sigma}_*^{-1}(\Pi_p)\to (\1-p)\M p\times (U_0/U_{\rho_0}),
\quad 
\chi_{p}^w(u U_{\rho_0})=(y_p,w^{-1}u_{pp_0} U_{\rho_0})
\end{equation}
where the quotient $U_0/U_{\rho_0}$ is a smooth manifold by Lemma~\ref{stab}. 
It is straightforward to check that all the mappings $\chi_{p}^w$ are bijective for  $p\in \Ll_{p_0}(\M)$ and $w\in \U(\M)_{p_0}^p$, 
and moreover they define a smooth atlas of $P_0/U_{\rho_0}$. 
Moreover, 
$$(\chi_{p}^w\circ \widetilde{\pi}_0\circ (\Theta_{pp_0}^w)^{-1})(y,g)=(y,gU_{\rho_0}) \text{ for all }(y,g)\in (\1-p)\M p\times 
U_0.$$
This shows via Lemma~\ref{stab} that $\widetilde{\pi}_0$ is a $U_{\rho_0}$-principal bundle. 

To show that $\widetilde{\sigma}_*$ is a locally trivial bundle with its typical fiber $U_0/U_{\rho_0}$ we just need to express it in local charts like this: 
\begin{equation}
\label{sigma_star_loc}
(\varphi_p\circ \widetilde{\sigma}_*\circ (\chi_{p}^w)^{-1})(y,gU_{\rho_0})=y
\text{ for all }(y,gU_{\rho_0})\in (\1-p)\M p\times 
(U_0/U_{\rho_0})
\end{equation}
where $\varphi_p\colon\Pi_p\to(\1-p)\M p$ is the local chart of $\Ll_{p_0}(\M)$ given by~\eqref{varphi}.

\eqref{thm:32_item2}
The homogeneous space $U_0/U_{\rho_0}$ is a manifold by Lemma~\ref{stab}. %{BR05_Th2.2} 
It is straightforward to check that the mapping 
$$\tau\colon P_0\times_{U_0} (U_0/U_{\rho_0})\to P_0/U_{\rho_0},\quad 
[(u,gU_{\rho_0})]\mapsto ug U_{\rho_0}$$
is bijective and the mapping  
$$P_0\times_{U_0} (U_0/U_{\rho_0})\to P_0/U_0,\quad 
[(u,gU_{\rho_0})]\mapsto ug U_0$$
is well defined, hence one obtains the commutative diagram 
\begin{equation}\label{P04bis}
\xymatrix{ P_0\times_{U_0} (U_0/U_{\rho_0}) \ar[r]^{\ \quad \tau} \ar[d]_{\hat{\bl}_0} & P_0/ U_{\rho_0} \ar[d]^{\widetilde{\sigma}_*} \\
\Ll_{p_0}(\M) \ar[r]^{\quad\id}&  \Ll_{p_0}(\M)}
\end{equation}
which implies~\eqref{P04}. 
Here $\hat{\bl}_0([(u,gU_{\rho_0})]):=\bl(ug)=\bl(u)$ for all $u\in P_0$ and $g\in U_0$. 

It remains to endow $P_0\times_{U_0} (U_0/U_{\rho_0})$ with a manifold structure for which the mappings in the diagram~\eqref{P04bis} have  appropriate properties. 
This is achieved by the family of local charts 
\begin{align*}
\mu_p^w\colon \hat{\bl}_0^{-1}(\Pi_p)\to 
%(\1-p)\M p
\Pi_p\times (U_0/U_{\rho_0}),
\quad 
\mu_p^w([(u,gU_{\rho_0})]):=
&(
%\varphi_p(\bl(u))
\bl(u),w^{-1}u_{pp_0}g U_{\rho_0}) \\
=
&(\bl(u),w^{-1}u_p(\bl(u))^{-1}ug U_{\rho_0})
\end{align*}
parameterized by $p\in\Ll_{p_0}(\M)$ and $w\in\U(\M)_{p_0}^p$.  
(Here we used the factorization \eqref{5July2019} which implies $u_{pp_0}=u_p(\bl(u))^{-1}u$.)
It is straightforward to check that  
the mapping~$\mu_p^w$ is bijective, and its inverse can be 
given in terms of the  mapping $u_p\colon\Pi_p\to \U(\M)$ defined by \eqref{u_p} which is a cross-section of $\bl$. 
Specifically,   
the equation $k=w^{-1}u_p(\bl(u))^{-1}ug$ is equivalent to $ug= u_p(\bl(u))wk$ and 
using this for $g=u_p(q)$ we obtain  
$$(\mu_p^w)^{-1}(q,kU_{\rho_0})=[(u_p(q)w,kU_{\rho_0})] 
\text{ for }
q\in\Pi_p\text{ and }k\in U_0.$$ 
These local charts $\mu_p^w$ parameterized by $p\in \Ll_{p_0}(\M)$ and $w\in \U(\M)_{p_0}^p$ define a smooth atlas on $P_0\times_{U_0} (U_0/U_{\rho_0})$. 
Moreover, one has 
$$(\hat{\bl}_0\circ (\mu_p^w)^{-1})(q,kU_{\rho_0})
=\hat{\bl}_0([(u_p(q)w,kU_{\rho_0})])=\bl(u_p(q)w)=q$$
and 
$$(\chi_p^w\circ\tau\circ (\mu_p^w)^{-1})(q,kU_{\rho_0})
=\chi_p^w(u_p(q)wkU_{\rho_0})
=(y_p,kU_{\rho_0}) $$
where $y_p=\varphi_p(q)$ is given by~\eqref{varphi}. 
The desired properties of $\hat{\bl}_0$ and $\tau$ then follow directly from their local expressions given by the above formulas. 

\eqref{thm:32_item3}
It remains to show that the inclusion map $\iota\colon\Oc_{\rho_0}\hookrightarrow\M_*$ is smooth and its tangent map at every point of $\Oc_\rho$ is injective.  
To this end we use the commutative diagram 
$$\xymatrix{
	P_0 \ar[d]_{\widetilde{\pi}_0} \ar[r]^{\kappa} & \M_* \\
	P_0/U_{\rho_0} \ar[r]^{\sim} & \Oc_{\rho_0} \ar@{^{(}->}[u]_{\iota}
}$$
where $\kappa\colon P_0\to\M_*$, $\kappa(u):=u\rho_0 u^*$. 
Here $\widetilde{\pi}_0$ is a submersion and the bottom arrow is a diffeomorphism hence, in order to complete the proof, it suffices to note that $\kappa$ is smooth and for every $u\in P_0$ one has 
\begin{equation*}%\label{smooth_proof_eq2}
{\rm Ker}\, (T_u\kappa)= {\rm Ker}\, (T_u (\widetilde{\pi}_0))
\end{equation*}
where the inclusion $\supseteq$ is clear while the converse inclusion can be established as follows. 
One has 
$$T_u\kappa\colon T_u P_0\to\M_*,\quad (T_u\kappa)(x)=x\rho_0 u^*+u\rho_0 x^*$$
and 
$T_uP_0=\{x\in\M p_0: u^*x\in ip_0\M^h p_0\}$
by \eqref{TuP0}. 
Therefore, for arbitrary $x\in \Ker(T_u\kappa)$ one has 
$u\rho_0 x^*=-x\rho_0 u^*$. 
Multiplying here both sides by $u^*$ from the left and by $u$ from the right, we obtain $\rho_0 x^*v=-u^*x\rho_0$, hence $\rho_0 u^*x=u^*x\rho_0$ 
(by the above description of $T_uP_0$), 
that is, $x\in {\rm Ker}\, (T_u (\widetilde{\pi}_0))$. 
This completes the proof. 
		\end{proof}

	We note now that the transitive subgroupoid $\bt_*^{-1}(\Oc_{\rho_0})
	\tto \Oc_{\rho_0}$ of $\U(\M)*\M_*^+\tto \M_*^+$ is  just the coadjoint action groupoid 
\begin{equation}\label{sub}
		\U_{p_0}(\M)\ast\Oc_{\rho_0}\tto \Oc_{\rho_0}
\end{equation}
of the groupoid $\U_{p_0}(\M)\tto \Ll_{p_0}(\M)$ that was defined in Subsection~\ref{Subsect2.5}. 
	We also note here that these groupoids 
	would not change upon replacement of $\rho_0$ 
	by any other element of $\Oc_{\rho_0}$.

The following proposition completes the picture from \eqref{gaugethm}. 

\begin{prop}\label{gauge_isom}
	Both the gauge groupoid $\frac{P_0\times P_0}{U_{\rho_0}}\tto P_0/{U_{\rho_0}}$ and the action group\-oid~\eqref{sub} are \Banach Lie groupoids.
Moreover, one has the isomorphism 
\begin{equation}
\label{isom}
\xymatrix{
    	\frac{P_0\times P_0}{U_{\rho_0}} \ar@<-.5ex>[d] \ar@<.5ex>[d] \ar[r]^{[\Psi]\quad} 
    	& \U_{p_0}(\M)\ast \Oc_{\rho_0} \ar@<-.5ex>[d] \ar@<.5ex>[d]\\
    	P_0/U_{\rho_0} \ar[r]^{[\pi_0]\quad} & \Oc_{\rho_0}
         }
\end{equation}
between these \Banach Lie groupoids, where 
%$[\Psi]$ and $[\pi_0]$ are defined by the following maps 
\begin{equation}\label{Psi} 
P_0\times P_0\ni (u,v)\mapsto \Psi(u,v):=(uv^*, v\rho_0 v^*)\in \U_{p_0}(\M)\ast\mathcal{O}_{\rho_0}
\end{equation}
		and 
\begin{equation}\label{psi} 
P_0\ni v\mapsto \pi_0(u)= u\rho_0 u^*\in \Oc_{\rho_0},
\end{equation}
respectively.
\end{prop}

\begin{proof}
 One has the actions 
\begin{equation}
\label{aPP} P_0\times P_0\ni (u,v)\mapsto (vg,ug)\in P_0\times P_0
\end{equation}
and 
\begin{equation}
\label{aP} P_0\ni u\mapsto vg\in P_0
\end{equation}
 of the group $ U_{\rho_0}\subset U_{0}$, where $g\in U_{\rho_0}$. 
 By $[u,v]$ and $[u]$ we will denote the orbits of $U_{\rho_0}$ through the elements $(u,v)$ and $u$, respectively. 
 We recall the structural maps for the gauge groupoid $\frac{P_0\times P_0}{U_{\rho_0}}\tto P_0/{U_{\rho_0}}$:
\begin{enumerate}[{\rm(a)}]
\item%(a)
the target and source maps are
$\bt([u,v]):=[u]$ and $\bs([u,v]):=[v]$
\item
the groupoid multiplication is 
\begin{equation*}
[u,v]\cdot[v',w']:=[ug,w'],
\end{equation*}
where $g\in U_{\rho_0}$ is defined by $vg=v'$,
\item 
the inverse map is 
$[u,v]^{-1}:=[v,u]$ 
\item%(d) 
the object inclusion map is  
$\bepsilon([u]):=[u,u]$.
\end{enumerate}
Let us note that that maps $\Psi$ and $\pi_0$ defined in \eqref{Psi}  and \eqref{psi} are invariant with respect to the actions \eqref{aPP} and \eqref{aP}, hence one has the well-defined 
bijections $[\Psi]:\frac{P_0\times P_0}{U_{\rho_0}}\to \U_{p_0}(\M)\ast\Oc_{\rho_0}$ and 
$[\pi_0]: P_0/{U_{\rho_0}}\to \Oc_{\rho_0}$. 
We easily see that 
\begin{enumerate}[{\rm(a)}]
	\item%(a)
	one has 
 \begin{align*}
 \bt_*([\Psi])([u,v])
 & =\bt_*(uv^*,v\rho_0 v^*) 
  =uv^*v\rho_0v^*vu^* 
  =up_0\rho_0 u^* 
 =u\rho_0 u^* 
 =\pi_0(u) \\
 & =[\pi_0]([u]) 
 =[\Psi](\bt([u,v]))
 \end{align*}
and 
 $$\bs_*([\Psi])([u,v])=\bs_*(uv^*,v\rho_0 v^*)=v\rho_0v^*=\pi_0(v)=[\pi_0]([v])=[\Psi](\bs([u,v])),$$ 
\item%(b)
moreover 
\begin{align*}
 [\Psi]([u,v]\cdot[v',w'])
 &=[\Psi]([ug,w']) 
 =(ugw^{'*},w'\rho_0w^{'*}) 
 =(uv^*vgw^{'*}, w'\rho_0w^{'*}) \\
 &=(uv^*u'w^{'*}, w'\rho_0w^{'*}) 
=(uv^*,v\rho_0 v^*)\cdot (v'w^{'*},w'\rho_0w^{'*}) \\
&=[\Psi]([u,v])\cdot [\Psi]([v',w']),
\end{align*}
\item%(c)
also 
 \begin{align*}
 \bi_*([\Psi]([u,v]))
 &=\bi((uv^*,v\rho_0v^*)) 
  =(vu^*,uv^*v\rho_0v^*(uv^*)^*) 
 =(vu^*,u\rho_0u^*)\\
 &=[\Psi]([v,u]) 
 =[\Psi]([u,v]^{-1}),
 \end{align*}
\item%(d)
and finally 
$$\bepsilon_*([\pi_0]([u]))=\bepsilon_*(u\rho_0u^*)=(u^*u,u\rho_0u^*)=[\Psi]([u,u])=[\Psi](\bepsilon ([u])).$$ 
\end{enumerate}
The above equalities show  that the bijections $[\Psi]$ and $[\pi_0]$ define the groupoid isomorphism \eqref{isom}.

We now turn to establishing the smoothness properties of the spaces and mappings in the diagram~\eqref{isom}. 
We do this in three steps: 

{\it Step 1:} Smooth structure of the gauge groupoid $\frac{P_0\times P_0}{U_{\rho_0}}\tto P_0/U_{\rho_0}$. 

Recall from Theorem~\ref{thm:32}\eqref{thm:32_item1} that $[\pi_0]$ is a diffeomorphism by the definition of the smooth manifold structure of $\Oc_{\rho_0}$ and moreover $\pi_0\colon P_0\to \Oc_{\rho_0}$ is a $U_{\rho_0}$-principal bundle, hence there exist an open covering 
$\Oc_{\rho_0}=\bigcup\limits_{p\in \Ll_{p_0}(\M)}D_p$ and a family of smooth cross-sections $\beta_p\colon D_p\to P_0$ satisfying $\pi_0\circ\beta_p=\id_{D_p}$ for all $p\in \Ll_{p_0}(\M)$. 
(Specifically, with the notation from the proof of Theorem~\ref{thm:32}\eqref{thm:32_item1}, one may take  $D_p:=[\pi_0](\widetilde{\sigma}_*(\Pi_p))\subseteq\Oc_{\rho_0}$ for every $p\in\Ll_{p_0}(\M)$, 
where $\widetilde{\sigma}_*(\Pi_p)$ is the domain of the local chart $\chi_p^w$ of $P_0/U_{\rho_0}$.) 
Then for every $p\in \Ll_{p_0}(\M)$ we define a local chart of $\frac{P_0\times P_0}{U_{\rho_0}}$ by 
\begin{equation}\label{kappa_tilde}
\widetilde{\kappa}_p\colon \bt^{-1}([\pi_0]^{-1}(D_p))\to D_p\times P_0,\quad 
\widetilde{\kappa}_p([(u,v)]):=(\pi_0(u),vu^*\beta_p(\pi_0(u))).
\end{equation}
Here we note that, since $U_{\rho_0}$ acts freely transitively on the fibers of $\pi_0$ and $\pi_0(u)=\pi_0(\beta_p(\pi_0(u)))$, 
one has $\beta_p(\pi_0(u))=ug$ 
where $g=
u^*\beta_p(\pi_0(u))\in U_{\rho_0}$ hence $$(u,v)\sim
(uu^*\beta_p(\pi_0(u)),vu^*\beta_p(\pi_0(u)))
=(\beta_p(\pi_0(u)),vu^*\beta_p(\pi_0(u)))$$
with respect to the equivalence relation $\sim$ defined by the diagonal action from the right of $U_{\rho_0}$ on $P_0\times P_0$. 
It is straightforward to check that for every $p\in \Ll_{p_0}(\M)$ 
the mapping $\widetilde{\kappa}_p$ given by \eqref{kappa_tilde} is bijective, having its inverse 
\begin{equation*}
\widetilde{\kappa}_p^{-1}\colon D_p\times P_0\to \bt^{-1}([\pi_0]^{-1}(D_p)),\quad 
\widetilde{\kappa}_p^{-1}(\rho,w):=[(\beta_p(\rho),w\beta_p(\rho))].
\end{equation*}
Moreover the mappings $\{\widetilde{\kappa}_p : p\in \Ll_{p_0}(\M)\}$  define a smooth atlas that makes the gauge groupoid $\frac{P_0\times P_0}{U_{\rho_0}}$ into a Lie groupoid. 

{\it Step 2:} Smooth structure of the action groupoid 
$\U_{p_0}(\M)\ast \Oc_{\rho_0}\tto\Oc_{\rho_0}$. 

One has 
$$\U_{p_0}(\M)\ast \Oc_{\rho_0}=\{(u,\rho)\in \U(\M)\times \Oc_{\rho_0}: \br(u)=\sigma(\rho)\}$$ 
where both maps $\br\colon \U(\M)\to\Ll(\M)$ and $\sigma_*\colon \Oc_{\rho_0}\to \Ll_{p_0}(\M)\subseteq\Ll(\M)$ are submersions. 
To construct an explicit atlas of $\U_{p_0}(\M)\ast \Oc_{\rho_0}$ 
it suffices, via the bijection $[\pi_0]\colon P_0/U_{\rho_0}\to \Oc_{\rho_0}$, $vU_{\rho_0}\mapsto\pi_0(v)$,  to construct an atlas for 
$$\U_{p_0}(\M)\ast(P_0/U_{\rho_0})=\{(u,vU_{\rho_0})\in \U_{p_0}(\M)\times (P_0/U_{\rho_0}): \br(u)=\sigma(\pi_0(v))=\bl(v)\}.$$ 
To this end we note that for every $p,\tilde{p}\in\Ll_{p_0}(\M)$, 
one has the local chart of $\Ll_{p_0}(\M)$ on the open neighbourhood 
$\Pi_{\tilde{p}}$ of $\tilde{p}\in\Ll_{p_0}(\M)$, 
$$\varphi_{\tilde{p}}\colon \Pi_{\tilde{p}}\to (\1-\tilde{p})\M\tilde{p},\quad 
$$
(see \eqref{varphi}), 
the local chart of $\U(\M)$  
\begin{align*}
\Theta_{p\tilde{p}}\colon\bl^{-1}(\Pi_p)\cap\br^{-1}(\Pi_{\tilde{p}})
 & \to(\1-p)\M p\times\U(\M)_{\tilde{p}}^p\times (\1-\tilde{p})\M\tilde{p}, \\
%\quad 
u & \mapsto (\varphi_p(\bl(u)),u_{p\tilde{p}},\varphi_{\tilde{p}}(\br(u)))
\end{align*}
(see \eqref{atlas3}) and the local chart of $P_0/U_{\rho_0}$ for any $w\in\U(\M)_{p_0}^p$, 
\begin{align*}
\widetilde{\chi}_{\tilde{p}}^w\colon 
\widetilde{\sigma}_*^{-1}(\Pi_{\tilde{p}})
& \to (\1-\tilde{p})\M\tilde{p}\times(U_0/U_{\rho_0}),\\ 
vU_{\rho_0}
& \mapsto (\varphi_{\tilde{p}}(\bl(u)),w^{-1}v_{\tilde{p}p_0} U_{\rho_0})
\end{align*}
(see \eqref{chi}). 
Moreover, one has 
\begin{align*}
\varphi_{\tilde{p}}\circ \br\circ \Theta_{p\tilde{p}}^{-1}
\colon 
(\1-p)\M p\times\U(\M)_{\tilde{p}}^p\times (\1-\tilde{p})\M\tilde{p}
 & \to 
(\1-\tilde{p})\M\tilde{p},\\
(y_p,u_{p\tilde{p}},y_{\tilde{p}}) 
& \mapsto y_{\tilde{p}}
\end{align*}
and, as noted in \eqref{sigma_star_loc},
\begin{align*}
	\varphi_{\tilde{p}}\circ \sigma_*\circ (\widetilde{\chi}_{\tilde{p}}^w)^{-1}
\colon 
(\1-\tilde{p})\M\tilde{p}\times(U_0/U_{\rho_0})
& \to 
(\1-\tilde{p})\M\tilde{p},\\ 
(y_{\tilde{p}},gU_{\rho})
& \mapsto y_{\tilde{p}}.
\end{align*}
It follows that, denoting 
\begin{align*}
	(\U(\M)\ast(P_0/U_{\rho_0}))_{p\tilde{p}}
	:=\{(u,vU_{\rho_0}) & \in \U(\M)\times (P_0/U_{\rho_0}): \\ 
	& \br(u)=\sigma_*(\pi_0(v))=\bl(v)\in\Pi_{\tilde{p}},\ \bl(u)\in\Pi_p\},
\end{align*}
the mapping 
\begin{align*}
\widetilde{\nu}_{p\tilde{p}}^w\colon (\U(\M)\ast(P_0/U_{\rho_0}))_{p\tilde{p}}
& \to 
(\1-p)\M p\times\U(\M)_{\tilde{p}}^p\times (\1-\tilde{p})\M\tilde{p}\times(U/U_{\rho_0}),\\
(u,vU_{\rho_0})
&\mapsto 
(\varphi_p(\bl(u)),u_{p\tilde{p}},\varphi_{\tilde{p}}(\br(u)),w^{-1}v_{p\tilde{p}}U_{\rho_0})
\end{align*}
is bijective. 
Moreover, the family of mappings 
$$\{\widetilde{\nu}_{p\tilde{p}}^w : p,\tilde{p}\in\Ll_{p_0}(\M),w\in\U(\M)_{p_0}^p\}$$ 
is a smooth atlas on 
$\U_{p_0}(\M)\ast(P_0/U_{\rho_0})$ for which the action groupoid 
$\U_{p_0}(\M)\ast(P_0/U_{\rho_0})\tto P_0/U_{\rho_0}$ is a Lie groupoid. 

{\it Step 3:} The mapping $[\Psi]\colon \frac{P_0\times P_0}{U_{\rho_0}} \to
\U(\M)\ast \Oc_{\rho_0}$ is a diffeomorphism. 

As above, it suffices to show that the mapping 
$$[\widetilde{\Psi}]\colon \frac{P_0\times P_0}{U_{\rho_0}} \to
\U_{p_0}(\M)\ast (P_0/U_{\rho_0})$$
is a diffeomorphism, where 
\begin{equation}\label{Psi_tilde} 
P_0\times P_0\ni (u,v)\mapsto \widetilde{\Psi}(u,v):=(uv^*, vU_{\rho_0})\in \U(\M)\ast (P_0/U_{\rho_0}).
\end{equation}
Using the above local charts, one obtains 
\begin{align*}
(\widetilde{\nu}_{p\tilde{p}}^w\circ [\widetilde{\Psi}]\circ \widetilde{\kappa}_p^{-1})(\rho,v)
& =(\widetilde{\nu}_{p\tilde{p}}^w\circ [\widetilde{\Psi}])(\beta_p(\rho),v\beta_p(\rho)) \\
&=\widetilde{\nu}_{p\tilde{p}}^w(v^*,v\beta_p(\rho)U_{\rho_0}) \\
&=(\varphi_p(\br(v)),(v^*)_{p\tilde{p}},\varphi_{\tilde{p}}(\bl(v)),
w^{-1}(v\beta_p(\rho))_{p\tilde{p}}U_{\rho_0})
\end{align*}
which shows that $[\widetilde{\Psi}]$ is smooth. 
One can similarly show that $[\widetilde{\Psi}]^{-1}$ is smooth, and this completes the proof. 
\end{proof}

We mention that the \Banach manifold structure of the quotient sets that occur in Theorem~\ref{thm:32} and Proposition~\ref{gauge_isom} could have been less explicitly described using the general results on quotients of \Banach manifolds. 
Applications of this alternative method in other instances can be found e.g., in \cite{ACM05}, \cite{AV05}, \cite{BGJP19}, and the references therein. 

Now we establish the main result of this section.

\begin{thm}\label{coadj_grpd} 
	The coadjoint action groupoid $\U(\M)*\M_*^+\tto \M_*^+$ is a \Banach Lie groupoid. 
\end{thm}

\begin{proof} 
The coadjoint action groupoid $\U(\M)*\M_*^+\tto \M_*^+$ is the disjoint union of its transitive subgroupoids 
$\U_{p_0}(\M)\ast\Oc_{\rho_0}\tto\Oc_{\rho_0}$ parameterized by $\rho_0\in\M_*^+$.  
As established in Proposition~\ref{gauge_isom} each one of these transitive groupoids is a Lie groupoid, hence their disjoint union 
is 
a Lie groupoid.
\end{proof}

We point out once again that the \Banach manifold structures of total space and the unit space of the coadjoint action groupoid $\U(\M)*\M_*^+\tto \M_*^+$ from Theorem~\ref{coadj_grpd} provide nontrivial examples of manifolds in the sense of \cite{Bou}, 
whose various connected components are modeled on different Banach spaces. 
Such manifolds are not uncommon both in finite dimensions and in infinite dimensions, for instance the Grassmann manifolds consisting of the orthogonal projections in various $C^*$-algebras. 
However, what is special in the case of the groupoid $\U(\M)*\M_*^+\tto \M_*^+$ is that both its total space and its base have original topologies of connected spaces that fail to be manifolds, and therefore these topologies need to be enriched in order to obtain spaces that  carry \Banach manifold structures.

 \section{Symplectic dual pair related to a von Neumann algebra}\label{Sect4}

In the preceding sections, $\M$ was regarded as an abstract $W^*$-algebra. 
In this section we will consider it realized as a von Neumann algebra on a complex Hilbert space~$\H$. 
The commutant $\M'$ thus appears on the stage, 
and the main point of this section is the interaction between the Poisson structures on the preduals of $\M$ and $\M'$, as well as their relation to the canonical symplectic structure of the Hilbert space $\H$. 
See for instance Propositions \ref{diagdual_prop1}, \ref{prop:4.4}, and \ref{HP}. 

More specifically, let $\iota(\M)=\iota(\M)''\subset L^\infty (\H)$, 
where  $\iota:\M\hookrightarrow L^\infty (\H)$ is  an inclusion map of $\M$ into  the $W^*$-algebra $L^\infty (\H)$
of bounded operators on the  
%separable 
complex Hilbert space~$\H$.  
Further by $\iota':\M'\hookrightarrow  L^\infty (\H)$ we will denote the inclusion map for the commutant $\M'$ of $\M$ in $L^\infty (\H)$, i.e. $\iota '(\M ')=\iota (\M)'$. 
The corresponding predual maps $\iota_*:L^1(\H)\to \M_*$ and $\iota'_*:L^1(\H)\to \M '_*$ are defined by 
\begin{equation} \label{iota*}\langle \iota_*(\rho),x\rangle=\Tr(\rho\iota(x))
\end{equation}
for any $\rho\in L^1(\H)$ and $x\in \M$. 
For the definition of $\iota'_*$ one replaces  $x\in \M$ by $x'\in \M'$ and $\iota $ by $\iota '$ in \eqref{iota*}. 
We will use the following notation:
\begin{enumerate}[(i)]
\item $(\iota_*)^*(C^\infty(\M_*,\mathbb{C})):=\{F\circ \iota_*: F\in C^\infty(\M_*,\mathbb{C})\}$;
\item $(\iota_*)^*(C^\infty(\M_*,\mathbb{C}))'$ stands for the commutant of $(\iota_*)^*(C^\infty(\M_*,\mathbb{C}))$ in the function algebra $C^\infty(L^1(\H),\mathbb{C})$ with respect to the Lie-Poisson bracket $\{\cdot,\cdot\}_{L^1}$ defined in \eqref{bracketi};
\item analogous notation will be used for $\M '_*$;
\item although one has the equality  $(\iota_*)^*(\M)=\iota(\M)$ we will consider $(\iota_*)^*(\M)'$ as the commutant of $(\iota_*)^*(\M)$ in $C^\infty(L^1(\H),\mathbb{C})$ while $\iota(\M)'$ as the commutant of $\iota(\M)$ in $L^\infty(\H)$ (note that $\M\subset C^\infty(\M_*,\mathbb{C}))$.
\end{enumerate}

\begin{prop}\label{diagdual_prop1}
\begin{enumerate}[{\rm(i)}]
\item\label{diagdual_prop1_item1}
One has the following surjective Poisson morphisms 
 \begin{equation}\label{diagdual}
     \xymatrix
     { 	& L^1(\H) 	\ar[dl]_{\iota_* ' } \ar[dr]^{\iota_*} &     \\
     	\M'_*  &	 &  \M_* 
     }
     \end{equation}	
			of complex Lie-Poisson spaces.
\item\label{diagdual_prop1_item2} 
In the Poisson algebra $C^\infty(L^1(\H),\C)$ one has the Poisson commutation relations 
\begin{equation}\label{43a} ((\iota_*)^*(\M))'=(\iota_*)^*(C^\infty(\M_*,\mathbb{C}))'
\end{equation}
and
\begin{equation} \label{44a}(\iota_*)^*(C^\infty(\M_*,\mathbb{C}))\subset  (\iota_*)^*(C^\infty(\M_*,\mathbb{C}))''\subset (\iota '_*)^*(C^\infty(\M_*',\mathbb{C}))'
\end{equation}
These relations are satisfied for $\M'$, too.

\item\label{diagdual_prop1_item3} 
Assertions similar to  \eqref{diagdual_prop1_item1}--\eqref{diagdual_prop1_item2}
%(i) and (ii) 
hold true for the hermitian part  
			 \begin{equation}\label{diagdual2}
     \xymatrix
     { 	& L^1_h(\H) 	\ar[dl]_{\iota_* ' } \ar[dr]^{\iota_*} &     \\
     	\M'_{h*}  &	 &  \M_{h*} 
     }
     \end{equation}	
of the diagram \eqref{diagdual} when one considers the corresponding real Poisson algebras, see \eqref{bracketi} and \eqref{LPbracketi}.
\end{enumerate}
\end{prop}

\begin{proof}
%(i) 
\eqref{diagdual_prop1_item1} 
This statement follows by Proposition~\ref{prop:2.6}. 
We give however an alternative, more direct proof, 
as this is useful for proving the items \eqref{diagdual_prop1_item2}--\eqref{diagdual_prop1_item3}.   
For arbitrary $F,G\in C^\infty(\M_*,\C)$ one has
\begin{align*}
\{F\circ \iota_*, G\circ \iota_*\}_{L^1}(\rho)
&= \Tr (\rho[\de(F\circ \iota_*)(\rho),\de( G\circ \iota_*)(\rho)]_{L^1}) \\
& =\Tr (\rho[\de F(\iota_*(\rho))\circ \iota_*,\de G (\iota_*(\rho))\circ\iota_*]_{L^1}) \\
&=\Tr (\rho[\iota(\de F(\iota_*(\rho))),\iota (\de G (\iota_*(\rho)))]_{L^1}) \\
&=\Tr (\rho\iota[\de F(\iota_*(\rho)),\de G (\iota_*(\rho))]_{\M}) \\
& =\langle\iota_*(\rho),[\de F(\iota_*(\rho)),\de G (\iota_*(\rho))]_{\M}\rangle \\
&=\{F,G\}_{\M_*}(\iota(\rho_*))
\end{align*}
The above shows that $\iota_*$ is a Poisson map. Note here that proving the above sequence of equalities  we used $X\circ \iota^*=\iota(X)$ and $\iota([X,Y]_\M)=[\iota(X), \iota(Y)]_{L^\infty}$ which are valid for $X, Y\in \M$. For $\iota_* ' :L^1\to \M_*'$ the proof is analogous.

%(ii) 
\eqref{diagdual_prop1_item2} The condition that 
$$0
=\{F\circ \iota_*,g\}_{L^1}(\rho)
=\Tr(\rho[\iota(\de F(\iota_*(\rho))),\de g (\rho)])
=\Tr([\rho,\iota(\de F(\iota_*(\rho)))] \de g (\rho))$$
for any $F\in C^\infty(\M_*,\C)$ is equivalent to the condition 
$$ 0=\Tr([\rho,\iota(X)]\de g(\rho))$$
fulfilled for any $X\in \M$. 
The above means that $g\in (\iota_*)^*(C^\infty(\M_*,\C))'$ if and only if  $g\in (\iota_*)^*(\M)'$. 
So, the equality \eqref{43a} is valid. 
From $\iota '(\M ')=\iota(\M)'$ one obtains 
$$(\iota '_*)^*(\M ')\subset (\iota_*)^*(\M)'$$
which gives 
\begin{equation}\label{46}
(\iota _*)^*(\M )''\subset (\iota'_*)^*(\M ')'.
\end{equation}
Substituting into \eqref{46} the equality \eqref{43a} we obtain \eqref{44a}.

%(iii) 
\eqref{diagdual_prop1_item3} 
The proofs of 
%(i) and (ii) 
\eqref{diagdual_prop1_item1} and \eqref{diagdual_prop1_item2}
can be adapted to the real Poisson algebras $C^\infty(\M_{h*},\mathbb{R})$, $C^\infty(\M '_{h*},\mathbb{R})$ and $C^\infty(L^1_h(\H),\mathbb{R})$, whose Poisson brackets are given by \eqref{bracketi} and \eqref{LPbracketi},  respectively.
\end{proof}

	We now recall that 
	for any subset $\A$ of a Poisson algebra one has 
	\begin{equation}
	\label{diagdual_prop1_proof_open}
	\A'=\A'''.
	\end{equation}
(In fact, for any subset~$\B$ of the Poisson algebra under consideration one has the obvious inclusion $\B\subseteq\B''$. 
		Applying this for $\B=\A'$, we obtain $\A'\subseteq\A'''$. 
		On the other hand, taking commutants in both sides of the obvious inclusion $\A\subseteq\A''$, we obtain $\A'''\subseteq\A'$. 
		Thus \eqref{diagdual_prop1_proof_open} is proved.)
	Nevertheless, a general ``bicommutant theorem'' $\A=\A''$ fails to be true for general Poisson subalgebra $\A$. 
	More specifically, we show below that the first inclusion in \eqref{44a} in general fails to be an equality. 
	
	To this end let $\iota(\M)= L^\infty(\H)$ and $\iota '(\M')=\C\1$. 
Then in \eqref{diagdual} one has 
$$\iota'_*=\Tr\colon L^1(\H)\to\C\ 
\text{ and }\ \iota_*=\id\colon L^1(\H)\to \M_*=L^1(\H).$$
Consequently $(\iota_*)^*(C^\infty(\M_*),\C)=C^\infty(L^1(\H),\C)$ and 
\begin{equation}
\label{contr_eq1}
(\iota'_*)^*(C^\infty(\M'_*),\C)=\{g\circ\Tr: g\in C^\infty(\C,\C)\}.
\end{equation}

\begin{prop}
The center of the Poisson algebra $(C^\infty(L^1(\H),\C), \{\cdot,\cdot\}_{L^1})$ 
(that is, the set of Casimir functions), to be denoted by 
$$\Cas(L^1(\H),\C):=\{f\in C^\infty(L^1(\H),\C):\ \{f,C^\infty(L^1(\H),\C)\}_{L^1}=\{0\}\},$$
satisfies 
\begin{align}
\label{c1}
(\iota'_*)^*(C^\infty(\M'_*,\C))\subsetneqq & \Cas(L^1(\H),\C)=(\iota'_*)^*(C^\infty(\M'_*,\C))'',\\
\label{c2}
(\iota_*)^*(C^\infty(\M_*,\C))= &\Cas(L^1(\H),\C)'=(\iota'_*)^*(C^\infty(\M'_*,\C))'.
\end{align}
\end{prop}

\begin{proof}
In fact, denoting $\A:=(\iota'_*)^*(C^\infty(\M'_*,\C))$, 
equation~\eqref{c1} takes the form 
$$\A\subsetneqq \Cas(L^1(\H),\C)=\A''.$$
In order to prove this fact, we first note that if $g\in C^\infty(\C,\C)$, 
then  
 for arbitrary $F,G\in C^\infty(L^1(\H))$ we have 
$\{g\circ F,G\}_\LP=(g'\circ F)\cdot \{F,G\}_\LP$ as a direct consequence of \eqref{LPbracket}. 
In particular, if $F\in \Cas(L^1(\H),\C)$ then $g\circ F\in \Cas(L^1(\H),\C)$. 

Now let us define 
$$F_k\colon L^1(\H)\to\C, \quad F_k(\rho):=\Tr(\rho^k),$$ 
for any positive integer $k\ge 1$, as in \cite[page 45]{OR03}. 
It is straightforward to check that the differential of $F_k$ at an arbitrary point $\rho\in L^1(\H)$ is 
$$\de F_k(\rho)\colon L^1(\H)\to\C, \quad (\de F_k(\rho))(\delta)=\Tr(k\rho^{k-1}\delta)\text{ for all }\delta\in L^1(\H).$$ 
That is, using the trace-duality identification $L^1(\H)^*\cong L^\infty(\H)$, 
we obtain $\de F_k(\rho)=k\rho^{k-1}$ for all  $\rho\in L^1(\H)$. 
Then for any $G\in C^\infty(L^1(\H),\C)$ and $\rho\in L^1(\H)$ we have 
$\{F_k,G\}_\LP(\rho)= \Tr(\rho[k\rho^{k-1},\de G(\rho)])=0$, 
using the well-known fact that  $\Tr(\rho_1x)=\Tr(x\rho_1)$ for all $\rho_1\in L^1(\H)$ and $x\in L^\infty(\H)$. 
Thus 
$$F_k\in\Cas(L^1(\H),\C)\text{ for every }k\ge 1.$$ 
In particular $\Tr=F_1\in\Cas(L^1(\H),\C)$, and it then follows by the above paragraph along with \eqref{contr_eq1} 
that $\A\subseteq\Cas(L^1(\H),\C)$. 
By the definition of $\Cas(L^1(\H),\C)$ we have  $\Cas(L^1(\H),\C)'=C^\infty(L^1(\H),\C)$ 
hence also $\A'=C^\infty(L^1(\H),\C)$. 
Also 
 $C^\infty(L^1(\H),\C)'=\Cas(L^1(\H),\C)$, hence 
$\A''=\Cas(L^1(\H),\C)$. 

We now check that $F_k\not\in (\iota'_*)^*(C^\infty(\M'_*),\C)$ if $k\ge 2$. 
To this end we argue by contradiction and thus, by \eqref{contr_eq1}, 
let us assume that there exists $g\in C^\infty(\C,\C)$ with $F_k(\rho)=g(\Tr \rho)$ for all $\rho\in L^1(\H)$. In particular, for $\rho=zp$, where $z\in\C$ is arbitrary and $p\in L^1(\H)$ is a rank-one projection, we obtain 
$z^k=F_k(zp)=g(\Tr(zp))=g(z)$ for all $z\in\C$. 
Hence the assumption $F_k(\rho)=g(\Tr \rho)$ implies $\Tr(\rho^k)=(\Tr\rho)^k$ for all $k\in L^1(\H)$, which actually fails to be true for any operator $\rho=p_1+p_2$, where $p_1$ and $p_2$ rank-one projections with $p_1p_2=0$, which always exist if $\dim\H\ge 2$. 
This shows that the inclusion in \eqref{c1} is indeed strict. 
\end{proof}

In the following, for the greater transparence of our considerations, 
we assume that the Hilbert space $\H$ is separable, as this allows us to use a simpler coordinate description. 
Nevertheless, it is easily seen that this separability assumption is not necessary for the validity of our results. 
More specifically, let us fix an orthonormal basis $\{\ket{n}\}_{n\in \mathbb{N}}$  in the Hilbert space $\H$ 
and express $|\gamma\rangle\in \H$ and $\rho\in L^1_h(\H)$ in this basis
\begin{align}
\label{43} 
\ket{\gamma} & =\sum^\infty_{n=1}z_n\ket{n}, \\
\label{44} 
\rho & =\sum^\infty_{k,l=1}\rho_{kl}\ket{k}\bra{l},
\end{align}
where $\rho_{lk}=\ol{\rho_{kl}}$.  
Recalling that $\langle\gamma_1\mid \gamma_2\rangle$ denotes the scalar product of  $\gamma_1,\gamma_2\in \H$,  
we have used Dirac's notation in \eqref{43} and \eqref{44}.  
Using that scalar product, we define  the symplectic form $$\omega:=\de\Gamma\in\Omega^2(\H,\R),$$ 
where
\begin{equation}
\label{Gamma} 
\Gamma:=i\langle \gamma\mid \de\gamma\rangle\in\Omega^1(\H,\C),
\end{equation}
recalling that $\gamma$ stands for a generic point of $\H$. 
We note that 
$$\Gamma-\overline{\Gamma}=\de(i\langle\gamma\mid\gamma\rangle),$$ 
hence $(\H,\omega)$ is a real symplectic manifold. 
The forms $\Gamma$ and $\omega$ in the coordinates $(z_k, \ol{z}_k)$ (where $k\ge 1$) can be written as 
\begin{equation*} 
\Gamma =i\sum_{k=1}^\infty \ol{z}_k \de z_k 
\text{ and }
\omega 
=i\sum ^\infty _{k=1}\de\ol{z} _k\wedge \de z_k
\end{equation*}
and, thus the Poisson bracket of $F,G\in C^\infty(\H,\mathbb{R})$ assumes the form 
		\begin{equation}\label{braomega}
		\{F,G\}_\omega=-i\sum^\infty_{k=1}\Bigl(\frac{\partial F}{\partial z_k}\frac{\partial G}{\partial \ol{z}_k}-
		\frac{\partial G}{\partial z_k}\frac{\partial F}{\partial \ol{z}_k}\Bigr).
		\end{equation}
		Let us mention here that $(\H,\omega)$ is a real strong symplectic manifold in sense of Definition \ref{def:2.3}.
		
		The Lie-Poisson bracket \eqref{LPbracket}  of $f,g\in C^\infty(L^1_h(\H),\mathbb{R})$ in the coordinates $\rho_{kl}=\ol {\rho_{lk}}$, where $k,l\in \mathbb{N}$ is given  by 
	\begin{align} 
	\{f,g\}_{\LP}
	&=i\sum^\infty_{k,l=1}\rho_{kl}\sum ^\infty_{m=1}
	\Bigl(\frac{\partial f}{\partial \rho_{km}}\frac{\partial g}{\partial \rho_{ml}}
		-\frac{\partial g}{\partial \rho_{km}}\frac{\partial f}{\partial \rho_{ml}}\Bigr) \nonumber \\
	\label{LP2}
	&=i\sum^\infty_{k,l=1}\rho_{kl}\sum ^\infty_{m=1}\Bigl(\frac{\partial f}{\partial \rho_{km}}\frac{\partial g}{\partial \ol{\rho_{lm}}}
		-\frac{\partial g}{\partial \rho_{km}}\frac{\partial f}{\partial \ol{\rho_{lm}}}\Bigr).
			\end{align}
		Using Dirac notation we define the smooth map $\mathcal{E}:\H\to L^1_h(\H)$ by
		\begin{equation}\label{E_cal} \mathcal{E}(\gamma):=\ket{\gamma}\bra{\gamma}.
		\end{equation}
		Expressing $\mathcal{E}$ in the coordinates $(z_k,\ol{z_k})$ we obtain 
		\begin{equation}
		\label{rho} 
		\rho_{kl}=z_k\ol{z_l}.
		\end{equation}
		
\begin{prop}\label{HP}
	The map \eqref{E_cal} is a Poisson map of the real symplectic manifold $(\H,\omega)$ into the real Lie-Poisson space $(L^1_h(\H),\{\cdot,\cdot\}_{\LP})$.
\end{prop}
		
\begin{proof} Substituting $F=\rho_{kl}$ and $G=\rho_{mr}$ given by \eqref{rho} into \eqref{braomega} we obtain 
		\begin{equation}\label{HP_proof_eq1} 
		\{f\circ \mathcal{E},g\circ \mathcal{E}\}_\omega=\{f,g\}_{\LP}\circ \mathcal{E}
		\end{equation}
	for any $f,g\in C^\infty(L^1_h(\H),\mathbb{R})$.	
\end{proof}

		So, the Poisson map  $\mathcal{E}:\H\to L^1_h(\H)$ is a  symplectic realization of the real \Banach Lie-Poisson space $(L^1_h(\H),\{\cdot,\cdot\}_{\LP})$ in the sense 
	of \cite[\S 6.3]{CW}. 
	The fibre $\mathcal{E}^{-1}(\mathcal{E}(\gamma))$ of $\mathcal{E}:\H\to L^1_h(\H)$ through $\gamma\in \H$ is the orbit  $\{\lambda\gamma:\ \lambda\in {\rm U}(1)\}$ of ${\rm U}(1)$. 
	Hence, the pull-back $\mathcal{E}^*(C^\infty(L^1_h(\H), \mathbb{R}))$ of $ C^\infty(L^1_h(\H), \mathbb{R})$ is contained in the Lie algebra $(C^\infty_{{\rm U}(1)}(\H,\mathbb{R}),\{\cdot,\cdot\}_\omega)$ of smooth functions invariant with respect to the action of ${\rm U}(1)$ on $\H$.
	
	For later use in the proof of Proposition~\ref{diagdual_prop2} we make the following remark.
	
	\begin{rem}\label{prop4.03}
		Assume the above notation.  
		\begin{enumerate}[{\rm(i)}]
			\item It follows from the definition~\eqref{E_cal} and from~\eqref{HP_proof_eq1} that the mapping $\mathcal{E}:\H\to L^1_h(\H)$ is a ${\rm U}(\H)$-equivariant momentum map, where the equivariance property means  
			$\mathcal{E}(u\gamma)=u\mathcal{E}(\gamma)u^*$ for every unitary operator $u\in{\rm U}(\H)$. 
			See for instance \cite{CW} for the notion of momentum map. 
			\item For any $X=X^*\in L^\infty(\H)$, let us define the function 
			$F_X\in C^\infty(L^1_h(\H),\R)$, given by 
			$$F_X(\rho):=\Tr(X\rho)\text{ for all }\rho\in L^1_h(\H).$$
			The function $F_X\circ\mathcal{E}\in C^\infty(\H,\R)$ in the coordinates~\eqref{43} assumes the form 
			$$(F_X\circ\mathcal{E})(\gamma)=\langle\gamma \mid X\gamma\rangle
			=\sum_{k,l=1}^\infty \ol{z}_k\bra{k}X\ket{l}z_l.$$
			\item The Hamilton equation associated to $F_X\circ\mathcal{E}$ is 
			\begin{align*}
			\frac{\de z_k}{\de t}
			&=\{F_X\circ\mathcal{E},z_k\} 
			=-i\sum_{l=1}^\infty\Bigl(
			\frac{\partial(F_X\circ\mathcal{E})}{\partial z_l} \frac{\partial z_k}{\partial\ol{z}_l} 
			-\frac{\partial z_k}{\partial z_l}\frac{\partial(F_X\circ\mathcal{E})}{\partial\ol{z}_l}\Bigr) \\
			&=i\frac{\partial(F_X\circ\mathcal{E})}{\partial\ol{z}_k} 
			=\sum_{l=1}^\infty \bra{k}X\ket{l}z_l
			\end{align*}
			that is, 
			$$\frac{\de \gamma}{\de t}=iX\gamma\text{ for all }\gamma\in\H.$$
			Thus, the Hamiltonian flow $\R\ni t\mapsto \sigma_X(t)\in C^\infty(\H,\H)$ 
			generated by the function $\mathcal{E}^*(F_X)\in C^\infty(\H,\R)$ on the symplectic manifold $(\H,\omega)$ is the one-parameter unitary group  
			$$\sigma_X(t)=\exp(itX)\in{\rm U}(\H).$$ 
		\end{enumerate}
	\end{rem}
	
	According to the definition presented in \cite[\S 9.3]{CW} 
	a symplectic manifold $M$ and Poisson manifolds $P_1$ and $P_2$ form a \textbf{symplectic dual pair} if one has Poisson maps
		\begin{equation}\label{sdp}
     \xymatrix
     { 	& M 	\ar[dl]_{J_1 } \ar[dr]^{J_2} &     \\
     P_1  &	 &  P_2 
     }
     \end{equation}
		with symplectically  orthogonal fibres. 
The next proposition gives an example of a symplectic dual pair.
		
	\begin{prop}
		\label{diagdual_prop2} 
		Assuming trivial Poisson structure on $\mathbb{R}$, one obtains the pair of Poisson maps
			\begin{equation}\label{diagdualH}
     \xymatrix
     { 	& \H 	\ar[dl]_{\mathcal{E} } \ar[dr]^{\Tr\circ \mathcal{E}} &     \\
     	L^1_h(\H)  &	 &  \mathbb{R} 
     }
     \end{equation}
		such that $\Ker T_\gamma\mathcal{E} $ and $\Ker T_\gamma(\Tr\circ\mathcal{E}) $ for $\gamma\in \H$ are symplectically orthogonal to each other, i.e., the diagram \eqref{diagdualH} defines a symplectic dual pair. 
		One also has
	\begin{align}
	\label{E1} 
	\mathcal{E}^* (C^\infty(L_h^1(\H),\mathbb{R}))'
	& =(\Tr\circ \mathcal{E})^*(C^\infty(\mathbb{R},\mathbb{R})),\\
	\label{E2} 
	(\Tr\circ \mathcal{E})^*(C^\infty(\mathbb{R},\mathbb{R}))'
	&\supseteq\mathcal{E}^* (C^\infty(L_h^1(\H),\mathbb{R})),
	\end{align}
			where the commutants are taken with respect to  $\{\cdot,\cdot\}_\omega$.
				\end{prop}

	\begin{proof} 
		Let us take smooth curves 
		$$\gamma(t)\in \mathcal{E}^{-1}(\mathcal{E}(\gamma))
		\text{ and }\gamma'(t)\in (\Tr\circ \mathcal{E})^{-1}((\Tr\circ \mathcal{E})(\gamma))$$ 
		such that $\gamma(0)=\gamma'(0)=\gamma$. 
		One easily observes that $\gamma(t)=\lambda(t)\gamma$, where $\vert\lambda(t)\vert=1$ and $\lambda(0)=1$, 
		and $\langle \gamma'(t)\mid\gamma'(t)\rangle=\langle \gamma\mid\gamma\rangle$. 
		Hence, the tangent vectors 
		$$\dot\gamma:=\frac{\de}{\de t}\gamma(t)\vert_{t=0}\in T_\gamma \mathcal{E}^{-1}(\mathcal{E}(\gamma))$$ 
		and 
		$$\dot\gamma ':=\frac{\de}{\de t}\gamma '(t)|_{t=0}\in T_\gamma (\Tr\circ \mathcal{E})^{-1}((\Tr\circ \mathcal{E})(\gamma))$$ 
		satisfy 
			\begin{equation} \label{dotgamma}
			\dot \gamma=\dot\lambda \gamma
			\quad {\rm and}\quad 
			\langle\dot\gamma '\mid\gamma\rangle+\langle\gamma \mid\dot\gamma '\rangle=0
			\end{equation}
		where $\dot\lambda:=\frac{\de}{\de t}\lambda(t)|_{t=0}=-\ol{\dot\lambda}$. 
			Using \eqref{dotgamma} we find that 
			\begin{equation} \omega(\dot\gamma,\dot\gamma')
			=\frac{1}{2i}(\langle\dot\gamma \mid\dot\gamma '\rangle
			-\langle\dot\gamma ' \mid\dot\gamma \rangle
			=
			-\frac{\dot\lambda}{2i}(\langle\dot\gamma' \mid \gamma\rangle
			+\langle\gamma \mid \dot\gamma '\rangle)=0
			\end{equation}
			The above gives $(\Ker T_\gamma \mathcal{E})^\perp=\Ker T_\gamma(\Tr\circ \mathcal{E})$. 
			
We now prove \eqref{E1}--\eqref{E2}. 
To this end, we use Remark~\ref{prop4.03} above. 
So for arbitrary $X=X^*\in L^\infty(\H)$, we define the function 
$F_X\in C^\infty(L^1_h(\H),\R)$, given by 
$F_X(\rho):=\Tr(X\rho)$ for all $\rho\in L^1_h(\H)$. 
Now let us denote by $\xi_X\in \Gamma^\infty T\H$ the Hamiltonian vector field on the symplectic manifold $(\H,\omega)$ that corresponds to the function $\mathcal{E}^*(F_X)\in C^\infty(\H,\R)$, that is, the infinitesimal generator of the Hamiltonian flow $t\mapsto \sigma_X(t)$. 
Then for every $f\in \mathcal{E}^* (C^\infty(L_h^1(\H),\mathbb{R}))'$ we have 
$$%\begin{align*}
0=\{\mathcal{E}^*(F_X),f\}_\omega(\gamma)
=(\xi_X(f))(\omega)
=\frac{\de}{\de t}\Bigl\vert_{t=0}f(\sigma_X(t)\gamma). 
$$%\end{align*}
Since $t\mapsto \sigma_X(t)$ is a complete flow, we then obtain 
$$
f\in\{\mathcal{E}^*(F_X): X\in L^\infty_h(\H)\}'
\iff 
f\circ\sigma_X(t)=f\text{ for all }t\in\R.$$
Recalling the formula $\sigma_X(t):=\exp(itX)\in{\rm U}(\H)$ from Remark~\ref{prop4.03} and using the fact that the exponential map of the unitary group ${\rm U}(\H)$ is surjective, we further obtain 
\begin{align*}
f\in\{\mathcal{E}^*(F_X): X\in L^\infty_h(\H)\}'
& \iff 
f\circ u=f\text{ for all }u\in {\rm U}(\H) \\
& \iff f(\cdot)=\chi(\Vert\cdot\Vert^2)\text{ for some }\chi\in C^\infty(\R,\R)
\end{align*}
where we also used the fact that the unitary group ${\rm U}(\H)$ acts transitively on any sphere that is centered at $0\in\H$. 
On the other hand, it is easily seen that 
$$(\Tr\circ \mathcal{E})^*(C^\infty(\mathbb{R},\mathbb{R}))
=\{\chi(\Vert\cdot\Vert^2): \chi\in C^\infty(\mathbb{R},\mathbb{R})\}$$
hence we obtain $\mathcal{E}^* (C^\infty(L_h^1(\H),\mathbb{R}))'
\subseteq (\Tr\circ \mathcal{E})^*(C^\infty(\mathbb{R},\mathbb{R}))$. 
For the converse inclusion we note that if $F\in C^\infty(L_h^1(\H),\mathbb{R})$
and $\chi\in C^\infty(\R,\R)$, then, using Proposition~\ref{HP}, 
\begin{align*}
\{F\circ\mathcal{E},\chi\circ\Tr\circ\mathcal{E}\}_\omega(\gamma)
&=\{F,\chi\circ\Tr\}_\LP(\mathcal{E}(\gamma)) \\
& = \chi'(\Tr(\mathcal{E}(\gamma)))\cdot\{F,\Tr\}_\LP(\mathcal{E}(\gamma)) 
=0
\end{align*}
for all $\gamma\in\H$, where the last equalities hold true just as in the proof of~\eqref{c1}. 
This completes the proof of \eqref{E1}. 

We now obtain 
$$\mathcal{E}^* (C^\infty(L_h^1(\H),\mathbb{R}))\subseteq 
\mathcal{E}^* (C^\infty(L_h^1(\H),\mathbb{R}))''
=(\Tr\circ \mathcal{E})^*(C^\infty(\mathbb{R},\mathbb{R}))'$$
where the last equality is obtained by just taking the commutants of both sides of~\eqref{E1}. 
This completes the proof of \eqref{E2}, and we are done. 
\end{proof}
	
	There is another symplectic dual pair canonically related to the von Neumann algebra $\iota:\M\hookrightarrow L^\infty(\H)$ (the one presented in the diagram  \eqref{diagdualH} is related to $\id:\M\to L^\infty(\H)$). 
	In order to construct it  for $\M$ and $\M'$ we define the expectation maps
	\begin{align}
	\label{E} 
	E:= & \iota_*\circ \mathcal{E}:\H\to \M_*^+ \\
	\label{Eprim} 
	E':= & \iota_*'\circ \mathcal{E}:\H\to {\M_*'}^+
	\end{align}
	Since $\mathcal{E}$ is a smooth map (a quadratic polynomial, actually) and $\iota_*$ as well as $\iota _*'$ are linear continuous maps, so, the expectation maps defined above are smooth.
	
	\begin{prop} \label{prop:44} The support map $\sigma_*: \M_*^+\to \Ll(\M)$ has the following properties
\begin{equation}\label{suppvect}
\sigma_*(E(\gamma))=[\M'\gamma]\quad \text{and} \quad \sigma_*(E'(\gamma))=[\M\gamma]
\end{equation}
for any $\gamma\in \H$, where $[\M'\gamma]$ and $[\M\gamma]$ are the orthogonal projections on the closed subspaces $\ol{\M'\gamma}$ and $\ol{\M\gamma}$, respectively.
\end{prop}

\begin{proof} 
	By definition the support $\sigma_*(E(\gamma))\in \M$ of $E(\gamma)$ is the smallest projection in $\M$ such that $\sigma_*(E(\gamma))E(\gamma)=E(\gamma)\sigma_*(E(\gamma))=E(\gamma)$. Since $[\M'\gamma]\gamma=\gamma$ one has $[\M'\gamma]E(\gamma)=E(\gamma)[\M'\gamma]=E(\gamma)$. 
	
	From the above two facts we obtain that $\sigma_*(E(\gamma))\leq [\M'\gamma]$. 
	On the other hand from $\sigma_*(E(\gamma))E(\gamma)=E(\gamma)$  we have    $\sigma_*(E(\gamma))\gamma=\gamma$. 
	Thus we obtain $\sigma_*(E(\gamma))\M'\gamma=\M'\gamma$ which implies that $[\M'\gamma]\leq\sigma_*(E(\gamma)).$
\end{proof}

The following fact goes back to \cite[Lemma 4.2(2)]{Ha73}, 
but we give it here in a form that is suitable for our purposes in this paper.

	\begin{prop}\label{prop:2} 
	If $\gamma_1,\gamma_2\in\H$ and $E(\gamma_1)=E(\gamma_2)\in \M_*^+$ (resp. $E'(\gamma_1)=E'(\gamma_2)\in {\M'}_*^+$), then there exists a unique $u'\in \U(\M')$ (resp. $u\in \U(\M)$) satisfying $u'\gamma_1=\gamma_2$ and $u^{'*}u'=[\M \gamma_1]$ (resp. $u\gamma_1=\gamma_2$ and $u^*u=[\M'\gamma_1]$). 
	Moreover $u'u^{'*}=[\M \gamma_2]$ (resp. $u^*u=[\M'\gamma_1]$).
\end{prop}

\begin{proof} 
	If $E(\gamma_1)=E(\gamma_2)$ then for arbitrary $x\in\M$ one has 
$E(\gamma_1)(x^*x)=E(\gamma_2)(x^*x)$, 
which is equivalent to $\Vert x\gamma_1\Vert=\Vert x\gamma_2\Vert$. 
This shows that there exists a unique partial isometry $u':\H\to\H$ with $u'^*u'=[\M\gamma_1]%p'_{\gamma_1}
$   
and $u'(x\gamma_1)=x\gamma_2$ for all $x\in\M$. 
Moreover, $u'u'^*=
[\M\gamma_2]
$. 

If $y=y^*\in\M$, then it is clear that both $\overline{\M\gamma_1}$ 
and its orthogonal complement are invariant with respect 
to the operator $y$. 
Now if $\delta\in \Ker u'$, then $y\delta\in \Ker u'$, 
hence $u'y\delta=yu'\delta=0$. 
On the other hand, for arbitrary $x\in\M$ one has 
$$u'y(x\gamma_1)=u'(yx\gamma_1)=yx\gamma_2=yu'(x\gamma_1),$$ 
and therefore $u'y\delta=yu\delta$ for all $\delta\in\overline{\M\gamma_1}=(\Ker u')^\perp$. 
Consequently $u'y=yu'$ for all $y=y^*\in\M$, hence $u'\in\M'$. 

To check uniqueness of $u'$, let  $w'\in\U(\M')$ satisfying $w'\gamma_1=\gamma_2$ and $w'^*w'=[\M\gamma_1]$. 
In particular, $\Ker w'=\Ker u'=\overline{\M\gamma_1}^\perp$. 
Moreover, for arbitrary $x\in\M$ one has $w'(x\gamma_1)=xw'\gamma_1=x\gamma_2$, hence $w'=u'$ on $\overline{\M\gamma_1}=(\Ker u')^\perp=(\Ker w')^\perp$. 
Thus $u'=w'$, and this completes the proof of the statement in the case $E(\gamma_1)=E(\gamma_2)$. 
The case $E'(\gamma_1)=E'(\gamma_2)$ then follows from the preceding case, 
interchanging $\M$ and $\M'$. 
\end{proof}

From Proposition \ref{prop:2} we conclude: 
					
\begin{cor}\label{cor:1}
\begin{enumerate}[{\rm(i)}]
\item\label{cor:1_item1} 
The groupoid of partial isometries $\U(\M)\tto\Ll(\M)$ acts on 
$\H$ 
by 
\begin{equation}\label{actionUH} 
\U(\M)*_\mu\H\ni(u,\gamma)\to u\gamma\in \H
\end{equation}
where the momentum map $\mu:\H\to \Ll(\M)$ is 
\begin{equation}
\label{mom}
\mu(\gamma):=[\M'\gamma]
\end{equation}
 and $(u,\gamma)\in \U(\M)*_\mu\H$ if and only if $\br(u)=u^*u=\mu(\gamma)$. 
 This action is free.  
\item\label{cor:1_item2} 
The orbits of the groupoid action \eqref{actionUH} are the fibres $E'^{-1}(E'(\gamma))$ of the expectation map $E'\colon \H\to \M_*'^+$.
\item\label{cor:1_item3} 
The expectation map $E:\H\to \M_*^+$ is equivariant with respect to the actions  of the groupoid $\U(\M)\tto \Ll(\M)$ on $\H$ and on $\M_*^+$, see \eqref{actionUH} and \eqref{coaction} respectively, i.e.,
\begin{equation*}
E(u\gamma)=uE(\gamma)u^*.
\end{equation*}
\item The above assertions \eqref{cor:1_item1}--\eqref{cor:1_item3}
hold also for the groupoid $\U(\M')\tto\Ll(\M')$ if one defines the momentum map $\mu':\H\to \Ll(\M')$ by $\mu'(\gamma):=[\M\gamma]$ and replaces $E'$ by $E$.
\end{enumerate}
\end{cor}

Summarizing the statement of Corollary \ref{cor:1} we obtain the family of fibre bundles presented in the following diagram:
\begin{equation}
\label{diag1}
\xymatrix{
	\H \ar[d]_{E} & E'^{-1}(\rho'_0) \ar@{_{(}->}[l] \ar[d]_{E}  &  P_0 \ar[l]^{\ }_{\ \ \ \iota_{\gamma_0}} \ar[d]^{\bl} \ar@{^{(}->}[r]& \U(\M) \ar[d]^{\bl} \\
	\M_*^+ &  \Oc_{\rho_0} \ar@{_{(}->}[l] \ar[r]^{\sigma_*} & \Ll_{p_0}(\M) \ar@{^{(}->}[r] & \Ll(\M)
}
\end{equation}
where $\rho_0=E(\gamma_0)$, $\rho_0'=E'(\gamma_0)$, $p_0=[\M'\gamma_0]=\sigma_*(\rho_0)$ and $p_0'=[\M\gamma_0]$ are defined by choice of $\gamma_0\in \H$.
The lower horizontal arrow in the middle of  \eqref{diag1} is the support map $\sigma_*:\M_*^+\to \Ll(\M)$, 
see Section~\ref{Sect3}. 
		
The upper horizontal arrow in \eqref{diag1} is defined for $u\in P_0$, for the definition $P_0$ see \eqref{P0}, by
\begin{equation}
		\label{pi} 
		%\pi
		\iota_{\gamma_0}(u):=u\gamma_0.
\end{equation}
We recall here that for $u\in P_0$ one has $u^*u=p_0=\sigma_*(\rho_0)=[\M'\gamma_0]$. 
We also recall that $\br(u)=u^*u$ and $\bl(u)=uu^*$ are right and left supports of $u\in \U(\M)$, respectively. 
It follows from Proposition~\ref{prop:2} that the definition \eqref{pi} is correct 
and $\iota_{\gamma_0}\colon P_0\stackrel{\sim}{\to} E'^{-1}(\rho_0')$ is a bijection.
		
\begin{prop}
\begin{enumerate}[{\rm(i)}]
\item The natural actions of $\U(\M)\tto \Ll(\M)$ on all objects of \eqref{diag1} are transitive and its actions on $E'^{-1}(\rho_0')$ and on $P_0$ are also free.
\item All arrows in \eqref{diag1} are equivariant maps with respect to these actions.
\end{enumerate}
\end{prop}

\begin{proof} 
Straightforward. 
\end{proof}
		
Let us recall that $P_0$ is the total space of a $U_{0}$-principal bundle over $\Ll_{p_0}(\M):=\{up_0u^*: u\in P_0\}$. 
		The \Banach smooth manifold structures on $P_0$ and on $\Ll_{p_0}(\M)$ were described in \cite{OJS1} and \cite{OJS2},  
		where it was shown also that $P_0(\Ll_{p_0}(\M), \bl_0, U_{0})$ is a $U_{0}$-principal bundle in sense  of the category of smooth \Banach manifolds. 
		See also Subsection~\ref{Subsect2.5}.

Regarding $(\H,\omega)$ as a real symplectic Hilbert manifold let us investigate the injection 
$
\iota_{\gamma_0}\colon P_0\to \H$ as a smooth map of these \Banach manifolds. 
To this end 
let us take a smooth curve $]-\epsilon,\epsilon[\ni t\mapsto u(t)\in P_0$ through $u(0)=u$. 
The corresponding curve in $\H$ is $\gamma(t):=\iota_{\gamma_0}(u(t))=u(t)\gamma_0$. 
It is important to remember that 
\begin{equation*}
u(t)^*u(t)=p_0=[\M'\gamma_0]
\end{equation*}
for $t\in ]-\epsilon,\epsilon[$. 
Thus $\dot{u}:=\frac{\de}{\de t}u(t)|_{t=0}\in T_uP_0$ satisfies 
\begin{align}
\label{uu} 
u^*\dot{u}+\Bigl(u^*\dot{u}\Bigr)^*=0,\\
\dot{u}p_0=\dot{u}. \nonumber
 \end{align}
From \eqref{uu} we see that $x:=u^*\dot{u}\in p_0\M p_0$ and $x+x^*=0$.  
Thus the real Banach space $T_uP_0$ tangent to $P_0$ at $u$ is 
\begin{equation}
\label{TuP0}
T_uP_0=\{ \dot{u}\in \M p_0:\ u^*\dot{u}\in ip_0\M^h p_0\},
\end{equation}
where $\M^h$ is the Hermitian part  of $\M$. 
So, for the map $T_u(\iota_{\gamma_0})\colon T_uP_0\to T_{\iota_{\gamma_0}(u)}\H$ tangent to  
$\iota_{\gamma_0}\colon P_0\to \H$ at $u\in P_0$ one has 
\begin{equation*}%\label{TuP00}
T_u(\iota_{\gamma_0})(T_uP_0)=\{\dot{u}\gamma_0=\dot{u}u^*\gamma: \quad \dot{u}\in T_uP_0\},
\end{equation*}
where $\gamma=u\gamma_0$.

\begin{prop}\label{P0imm}
 For the principal bundle $(P_0,\Ll_{p_0}(\M),\bl_0,U_0)$, 
 its total space $P_0$ is a weakly immersed  submanifold of $\H$ 
 via $\iota_{\gamma_0}:P_0\hookrightarrow\H$. 
\end{prop}

\begin{proof}
 For proving $\Ker T_u(\iota_{\gamma_0})=\{0\}$, we note that for $v\in \Ker T_u(\iota_{\gamma_0})$ one has $v\gamma_0=0$.
Hence, we have $v\M'\gamma_0=0$ what implies  $0=vp_0=v$.
The above shows that $\iota_{\gamma_0}\colon P_0\hookrightarrow\H$ is an injective weak immersion.
\end{proof}

Recalling from Section~\ref{Sect2} the definition of a weak immersion, 
we emphasize that in general the real vector subspace $T_u(\iota_{\gamma_0})(T_uP_0)$ 
is not closed in $T_\gamma\H\cong\H$. 
Assume for simplicity that $\H$ is separable and let $T_u(\iota_{\gamma_0})(T_uP_0)$ 
be closed in $T_\gamma\H\cong\H$. 
Then 
$T_u(\iota_{\gamma_0})\colon T_uP_0\to T_{\iota_{\gamma_0}(u)}\H$ 
is an injective operator whose range is a closed $\R$-linear subspace of the Hilbert space $\H$, 
and this implies that the Banach space $T_uP_0$ is separable (and it is actually topologically isomorphic 
to a separable real Hilbert space). 

On the other hand, it follows by \eqref{TuP0} for $u=p_0$ that $T_uP_0$ is a closed $\R$-linear subspace of $\M$ 
(with $ip_0\M^h p_0\subseteq T_uP_0$), hence we obtain that the $W^*$-algebra $p_0\M p_0$ is separable, 
and then $\dim(p_0\M p_0)<\infty$. 
Thus, if $\dim(p_0\M p_0)=\infty$ (which is always the case for instance 
if $\M$ is a type~II or type~III factor and $0\ne p_0\in\Ll(\M)$), then  $T_u(\iota_{\gamma_0})(T_uP_0)$ 
fails to be closed in $T_\gamma\H\cong\H$. 

We now recall from the discussion following \eqref{pi} that one has the bijection 
$$\iota_{\gamma_0}\colon P_0\to E'^{-1}(E'(\gamma_0)).$$
We use this bijection to transport the manifold structure of $P_0$ to $E'^{-1}(E'(\gamma_0))$, 
and then the inclusion map $E'^{-1}(E'(\gamma_0))\hookrightarrow\H$ is a weak immersion by Proposition~\ref{P0imm}. 
One can similarly define a manifold structure on $E^{-1}(E(\gamma_0))$ for arbitrary $\gamma_0\in\H$. 
We use these manifold structures in Proposition~\ref{prop:4.4} below. 

	\begin{prop}\label{prop:4.4}
	The diagram
	\begin{equation}\label{diagdual3}
\xymatrix
{ 	& \H 	\ar[dl]_{E } \ar[dr]^{E'} &     \\
	\M_{h*}  &	 &  {\M'}_{h*} 
}
\end{equation}
gives a symplectic dual pair (see \eqref{sdp}), 
in the sense that for any $\gamma\in \H$ the fibres tangent spaces $\Ker T_\gamma E'=T_\gamma E'^{-1}(E'(\gamma))$ and $\Ker T_\gamma E=T_\gamma E^{-1}(E(\gamma))$ are symplectically orthogonal. 
One also has 
\begin{align}\label{E3} E'^*(C^\infty(\M'_{h*}, \mathbb{R}))
& \subseteq E^*(C^\infty(\M_{h*}, \mathbb{R}))', \\
\label{E4} 
E^*(C^\infty(\M_{h*}, \mathbb{R}))
& \subseteq E '^*(C^\infty(\M '_{h*}, \mathbb{R}))'
\end{align}
\end{prop}

\begin{proof} 
	Let us consider two smooth curves $\gamma(t)\in E^{-1}(E(\gamma))$ and $\gamma'(t)\in  E'^{-1}(E'(\gamma))$, where $t\in ]-\epsilon,\epsilon[$ through $\gamma$, i.e., $\gamma(0)=\gamma'(0)=\gamma$. According to Proposition \ref{prop:2} we can represent them as $$\gamma(t)=u'(t)\gamma\text{ and }\gamma'(t)=u(t)\gamma,$$ 
	where $u'(t)\in \U(\M ')$ and $u(t)\in \U(\M)$ satisfy:
\begin{align}
\label{1} 
&u(0)=[\M'\gamma]=:p_\gamma \qquad u(t)^* u(t)=p_\gamma\\
\label{2} 
&u'(0)=[\M\gamma]=:p_\gamma ' \qquad u'(t)^* u'(t)=p_\gamma '
\end{align}
From \eqref{1}--\eqref{2} one obtains the following relations
\begin{align}
\label{3} 
(p_\gamma\dot u p_\gamma)^*+p_\gamma\dot u p_\gamma
&=0,\quad \dot\gamma=\dot u'\gamma, \quad p_\gamma\gamma=\gamma\\
\label{4} 
(p'_\gamma\dot u' p'_\gamma)^*+p'_\gamma\dot u' p'_\gamma
&=0,\quad \dot\gamma '=\dot u\gamma, \quad p_\gamma'\gamma=\gamma
\end{align}
where we use the notation 
$$\dot\gamma:=\frac{\de}{\de t}\gamma(t)\Bigl\vert_{t=0},\  
\dot\gamma ':=\frac{\de}{\de t}\gamma '(t)\Bigl\vert_{t=0},\  
\dot u:=\frac{\de}{\de t}u(t)\Bigl\vert_{t=0},\  
\dot u ':=\frac{\de}{\de t}u '(t)\Bigl\vert_{t=0}$$ 
for tangent vectors. 
Using \eqref{3}--\eqref{4} we have
\begin{align} 
\omega (\dot\gamma,\dot\gamma ')
&=\frac{1}{2i}\langle\gamma\mid p_\gamma' p_\gamma(\dot u'^*\dot u-\dot u^*\dot u')p_\gamma p_{\gamma '}\gamma\rangle \nonumber \\
%\label{5}
&=\frac{1}{2i}\langle\gamma\mid ((p'_\gamma\dot u' p'_\gamma)^*p_\gamma\dot u p_\gamma-p_\gamma\dot u p_\gamma(p'_\gamma\dot u' p'_\gamma)^*)\gamma\rangle 
\nonumber \\
&=
\frac{1}{2i}\langle\gamma\mid [p_\gamma\dot u p_\gamma,p'_\gamma\dot u' p'_\gamma]\gamma\rangle \nonumber \\
\label{5}
&=0,
\end{align}
i.e. the tangent spaces $T_\gamma(E^{-1}(E(\gamma))$  and $T_\gamma(E'^{-1}(E'(\gamma))$ are symplectically orthogonal. 
One obtains the last equality in \eqref{5} since $p_\gamma\dot u p_\gamma\in \M$ and $p'_\gamma\dot u' p'_\gamma\in \M'$. 
The inclusions \eqref{E3}--\eqref{E4} follow from \eqref{44a} 
%if one notes that 
since $E$ and $E'$ are Poisson maps.
\end{proof}

In Section~\ref{Sect6} we will investigate the symplectic dual pair presented in \eqref{diagdual3} for a standard form of a von Neumann algebra $\iota:\M\hookrightarrow L^\infty(\H)$.

 \section{Coadjoint orbits of the groupoid $\U(\M)\tto \Ll(\M)$}
\label{Sect5}

In this section we investigate the orbits of the natural action of the groupoid $\U(\M)\tto \Ll(\M)$ on the positive cone $\M_*^+$ in the predual $\M_*$. 
That action is called here the \emph{coadjoint action of the groupoid $\U(\M)\tto \Ll(\M)$} because of its close relation to the coadjoint action of the unitary group of~$\M$. 
In Section~\ref{Subsect5.1} we endow these groupoid coadjoint orbits with invariant weakly symplectic structures obtained by the reduction procedure whose input is the symplectic structure of the Hilbert space~$\H$ (Theorem~\ref{thm:orbit}). 
Then, in Subsection~\ref{Subsect5.2}, we show that the type of the von Neumann algebra~$\M$ can be read off the coadjoint orbits of the \Banach Lie groupoid $\U(\M)\tto \Ll(\M)$.

 \subsection{Weakly symplectic structure of the coadjoint orbits of $\U(\M)\tto \Ll(\M)$}
\label{Subsect5.1}
					
As shown in Section~\ref{Sect4}, the groupoid $\U(\M)\tto \Ll(\M)$ acts on $(\H,\omega)$ in a free and symplectic way. 
The reduction of the symplectic form $\omega$ to the orbits of this action will be described in this subsection.

Recall from the diagram \eqref{diag1} that $E'^{-1}(\rho'_0)\cong P_0$ is a weakly immersed submanifold of the real symplectic manifold $(\H,\omega)$. 
Therefore, one can consider the reduction of the symplectic form $\omega=\de\Gamma$ to $E'^{-1}(\rho'_0)$.

Also recall the 1-form $\Gamma$ defined in 
 \eqref{Gamma}. 
  By pull-back of $\Gamma$  
 to $E'^{-1}(\rho'_0)$, 
 i.e., 
 using the change of variables 
 $\iota_{\gamma_0}\colon P_0\to E'^{-1}(\rho'_0)$, 
 $u\mapsto \gamma:=u\gamma_0$
 in~\eqref{Gamma}, 
 %where $u\in P_0$, 
 we obtain the differential forms 
\begin{align}
\Gamma_0(u)
:=& (\iota_{\gamma_0}^*\Gamma)(u) 
=  i\langle u\gamma_0\mid \de(u\gamma_0)\rangle 
=  i\langle \gamma_0\mid u^*\de(u\gamma_0)\rangle 
=  i\langle \rho_0,u^*\de u\rangle \nonumber \\
\label{Gamma0}
= & i\langle\rho_0,\alpha(u)\rangle
\end{align}
and 
\begin{equation}
\label{dGamma}
\de\Gamma_0(u)=i\langle\rho_0,\de\alpha(u)\rangle
=i\langle\rho_0,\Omega(u)\rangle-i\langle\rho_0,\frac{1}{2}[\alpha(u),\alpha(u)]\rangle
\end{equation}
on $P_0$, where  $\alpha\in \Gamma^\infty(T^*P_0, p_0i\M^h p_0)$ is the connection form defined in \eqref{alpha} and $\Omega$ is its curvature form defined in \eqref{Omega}.

\begin{lem} \label{lemma1}
\begin{enumerate}[{\rm(i)}] 
\item\label{lemma1_item1} 
The horizontal $T^h_uP_0$ and vertical $T^v_uP_0$ components of $$T_uP_0=T^h_uP_0\oplus T^v_uP_0 $$ 
with respect to the connection form $\alpha$ are given by 
\begin{equation} \label{Th}  
T^h_uP_0:=\{(1-uu^*)\dot{u}:\ \dot{u} \in T_uP_0\}
\end{equation}
and
\begin{align} \label{Tv} 
T^v_uP_0:=
&\Ker T_u(\bl_0)=\{\dot{u}\in T_uP_0:\ u^*\dot{u}+\dot{u}^*u=0\} \\
=
&\{ux:\ x\in ip_0\M^h p_0\}=\{uu^*\dot{u}:\  \dot{u}\in T_uP_0\}, \nonumber
\end{align}
respectively.
\item\label{lemma1_item2} 
One has the orthogonality relation $T^h_uP_0\perp T^v_uP_0$ with respect to $d\Gamma_0$, i.e., 
\begin{equation} 
\de\Gamma_0(u)(\dot{u}^h, \dot{u}^v)=0
\end{equation}
for any $\dot{u}^h\in T^h_uP_0$ and $\dot{u}^v\in T^v_uP_0$
\item\label{lemma1_item3} 
The curvature of $\alpha$ is the  $2$-form given by 
\begin{equation}\label{Omega1} 
\Omega(u)(\dot{u}_1,\dot{u}_2)
=\frac{1}{2}(\dot{u}^*_1(1-uu^*)\dot{u}_2
-\dot{u}^*_2(1-uu^*)\dot{u}_1).
\end{equation}
One has also 
\begin{equation}\label{dGamma1} 
\de\Gamma_0(u)(\dot{u}_1, \dot{u}_2)
=\frac{i}{2} 
\langle \rho_0, \dot{u}_1^{h*} \dot{u}_2^{h}
-\dot{u}_2^{h*} \dot{u}_1^{h} \rangle
-i\langle\rho_0,[x_1,x_2]\rangle
\end{equation}
where $x_1=u^*\dot{u}_1$, $x_2=u^*\dot{u}_2$.
\end{enumerate}
\end{lem}

\begin{proof} 
%(i) 
\eqref{lemma1_item1} 
By definition one has 
\begin{align} 
T^h_uP_0
& =\{\dot{u}\in T_uP_0:\ \langle\alpha,\dot{u}\rangle=0\} 
 =\{\dot{u}\in T_uP_0:\ u^*\dot{u}=0\} \nonumber \\
\label{ThuP0}
& =\{\dot{u}\in \M p_0:\ u^*\dot{u}=0\}
\end{align}
where the last equation follows by \eqref{TuP0}. 
Hence by the decomposition
\begin{equation}
\label{decu*} 
\dot{u}=(1-uu^*)\dot{u}+uu^*\dot{u}
\end{equation}
of $\dot{u}\in T_uP_0$ one obtains $T_u^hP_0=(1-uu^*)T_hP_0$. 
We recall that $\bl_0:P_0\to\Ll_{p_0}(\M)$ is given by  $\bl_0(u)=uu^*$. 
Hence we have $T_u(\bl_0)\dot{u}=\dot{u}u^*+u\dot{u}^*$. 
This gives the first equality in \eqref{Tv}. 
Since $U_0$ acts on $\bl_0^{-1}(uu^*)$ in transitive  and free way, we obtain that $\dot{u}\in \Ker T_u(\bl_0)$ if and only if $\dot{u}=ux$, where $x\in ip_0\M^h p_0$. 
So, the second equality in \eqref{Tv} is valid. 
In order to obtain the last equality in \eqref{Tv} we use the decomposition \eqref{decu*} and note that $u^*\dot{u}\in ip_0\M^h p_0$.

%(ii) 
\eqref{lemma1_item2}
This follows from \eqref{Th}, \eqref{Tv} and  \eqref{decu*}.

%(iii) 
\eqref{lemma1_item3} By definition one has 
\begin{align} 
\Omega(u)(\dot{u}_1, \dot{u}_2)
:=
&
\de\alpha(\dot{u}_1^h, \dot{u}_2^h) 
=
 (\de u^*\wedge \de u)(\dot{u}_1^h, \dot{u}_2^h) 
=
\frac{1}{2}(\dot{u}_1^{h*} \dot{u}_2^h-\dot{u}_2^{h*} \dot{u}_1^h)
\nonumber \\
\label{Omega2}
= &
\frac{1}{2}({\dot{u}_1}^*(1-uu^*) \dot{u}_2-\dot{u}_2^*(1-uu^*) \dot{u}_1).
\end{align}
In order to prove \eqref{Omega1} we substitute $\dot{u}_1, \dot{u}_2\in T_uP_0$ into \eqref{dGamma} decomposed according to \eqref{decu*} using next \eqref{Th} and \eqref{Tv}.
\end{proof}

From the $U_0$-equivariance properties \eqref{alphaequiv} and \eqref{Omegaequiv} and \eqref{Gamma0}--\eqref{dGamma}  it follows that $\Gamma_0$ and $d\Gamma_0$ are $U_{\rho_0}$-invariant differential forms on  $P_0$.  
They also satisfy
\begin{equation}
\label{cos}
\xi_x\llcorner \Gamma_0=i\langle \rho_0,x\rangle\quad 
\text{and} \quad \xi_x\llcorner \de\Gamma_0=0,
\end{equation}
where $\xi_x$ is the fundamental vector field generated by $x\in \ug_{\rho_0}$ (= the Lie algebra of $U_{\rho_0}$). 
The property \eqref{cos} follows from \eqref{230} and
\begin{equation} \label{cos2}
0=\mathcal{L}_{\xi_x}\Gamma_0=\xi_x\llcorner \de\Gamma_0+\de(\xi_x\llcorner \Gamma_0)=
\xi_x\llcorner \de\Gamma_0+i\de\langle\rho_0,x\rangle=\xi_x\llcorner \de\Gamma_0
\end{equation}
From the above we conclude that there exists on $\mathcal{O}_{\rho_0}$ a closed differential $2$-form $\widetilde{\omega}_{\rho_0}$ such that $\tilde\pi_0^*\widetilde{\omega}_{\rho_0}=\de\Gamma_0$, where $\tilde\pi_0:P_0\to P_0/U_{\rho_0}\cong\mathcal{O}_{\rho_0}$ is the quotient map. 

The invariance property of the 2-form $\widetilde{\omega}_{\rho_0}\in\Omega^2(\Oc_{\rho_0})$ pointed out in the following theorem 
implies that $\widetilde{\omega}_{\rho_0}$ actually depends only on the groupoid orbit $\Oc_{\rho_0}$ and not on the choice of the point $\rho_0$ in that orbit, just as it is the case with coadjoint group orbits in Remark~\ref{Kir}. 
Nevertheless, just as in the case of the orbit $\Oc_{\rho_0}$ we choose to indicate explicitly the point $\rho_0$ that is used for the construction of the differential form~$\widetilde{\omega}_{\rho_0}$.

\begin{thm} \label{thm:orbit} The coadjoint orbit 
$\Oc_{\rho_0}$ 
of $\U(\M)\tto \Ll(\M)$ is a weak symplectic manifold whose symplectic structure is given by 
$\widetilde{\omega}_{\rho_0}$. 
The weak symplectic form 
$\widetilde{\omega}_{\rho_0}$ 
is invariant with respect to the coadjoint actions of $\U(\M)\tto \Ll(\M)$.
\end{thm}

\begin{proof}
Substituting $vu\in \U(\M)$, where $\br(v)=\bl(u)$, such that $v$ is a fixed element of $\U(\M)$,  into \eqref{Gamma0} instead of $u\in \U(\M)$ and using $\br(v)=v^*v=\bl(u)=uu^*$, we find that
\begin{align*}
\Gamma_0(vu)
&=i\langle \gamma_0\mid (vu)^*\de (vu)\gamma_0\rangle 
=i\langle \gamma_0\mid u^*v^*v\de u\gamma_0\rangle 
=i\langle \gamma_0\mid u^*uu^*\de u\gamma_0\rangle \\
&=i\langle \gamma_0\mid u^*\de u\gamma_0\rangle 
=\Gamma_0(u).
\end{align*}
The above shows, that $\de\Gamma_0$ is a differential $2$-form on $P_0$ invariant with respect to the left action of $\U(\M)\tto\Ll(\M)$ on $P_0$. 
Thus, since $E:E'^{-1}(\rho_0')\to \mathcal{O}_{\rho_0}$ is an equivariant map with respect to the actions of $\U(\M)\tto\Ll(\M)$ on $E'^{-1}(\rho_0')$  and on $\mathcal{O}_{\rho_0}$, we find that 
$\widetilde{\omega}_{\rho_0}$ 
is also a invariant closed differential $2$-form on the orbit $\mathcal{O}_{\rho_0}$.

In order to prove that 
$\widetilde{\omega}_{\rho_0}$ 
is nondegenerate 
we use Lemma~\ref{lemma1}. 
Namely, the $\de\Gamma_0$-orthogonality of subspaces $T^h_uP_0$ and  $T^v_uP_0$ allows us to consider this question for each of these components separately. 
Firstly let us consider the horizontal  component $T^h_uP_0$. 
Substituting $\dot{u}_1^h=\dot{u}^h\in T^h_uP_0$ and $\dot{u}_2^h=i\dot{u}^h\in T^h_uP_0$ into \eqref{Omega2} 
(where we recall from the last equality in \eqref{ThuP0} that $T^h_uP_0$ is a complex linear subspace of $\M p_0$)
we obtain that 
\begin{equation}
\label{rho0}
\langle \rho_0,\dot{u}^{h*}\dot{u}^h\rangle=0
\iff 
\dot{u}^h=0.
\end{equation}
The above fact follows from the positivity of $\rho_0$ on $p_0\M p_0$ and from $\sigma_*(\rho_0)=p_0$. 
We note also that $\dot{u}^{h*}\dot{u}^h\in p_0\M p_0$. 
From \eqref{rho0} we conclude that the horizontal component of the right-hand side  of equality \eqref{dGamma1} is non-singular.

Rewriting the vertical components of the right-hand side of equality \eqref{dGamma1}  as follows 
\begin{equation} 
-i\langle\rho_0,[x_1,x_2]\rangle=i\langle \ad ^*_{x_1}\rho_0,x_2\rangle
\end{equation}
and assuming that $\langle\rho_0,[x_1,x_2]\rangle=0$ for all $x_2\in ip_0\M^h p_0$ we obtain $\ad ^*_{x_1}\rho_0=0$ i.e., $x_1\in \ug_{\rho_0}\subseteq \ug_0$, 
where $\ug_{\rho_0}$ is Lie algebra of $U_{\rho_0}$ and $\ug_0=ip_0\M^hp_0$ is the Lie algebra of $U_0=U(p_0\M p_0)$. 
This means that the degeneracy vectors of $\de\Gamma_0(u)$ restricted to the vertical part  $T^v_uP_0$ are tangent to the fibres of $\tilde\pi_0:P_0\to P_0/U_{\rho_0}\cong \mathcal{O}_0$. 
Summarizing the facts from above,  we conclude  from \eqref{cos2} that 
$\widetilde{\omega}_{\rho_0}$ 
is a weakly symplectic form.
\end{proof}

\begin{rem}\label{Kir2}
	\normalfont
	We conclude this subsection by comparing the results presented in Lemma~\ref{lemma1} and Theorem~\ref{thm:orbit} with the ones obtained in 
	\cite{BR05} and \cite{OR03} 
	for the coadjoint orbits  of $U(\M)$ in $\M^+_*$. 
	For this reason  let us consider  the subgroupoid $\U(\M)_{\rm unit}\tto\Ll(\M)$ of the groupoid 
	$\U(\M)\tto\Ll(\M)$ defined as follows. 
	By definition the partial isometry $u\in \U(\M)$ belongs to $\U(\M)_{\rm unit}$ if and only if it has an extension to a unitary element $U\in U(\M)$ of $\M$. 	
	For a finite $W^*$-algebra $\M$ that extension is always possible, that is, $\U(\M)_{\rm unit}=\U(\M)$.
	(A $W^*$-algebra $\M$ is \emph{finite} if and only if any two projections in $\M$ are unitary equivalent whenever they are Murray-von Neumann equivalent, see e.g., \cite{Sakai}. 
	If this is not the case, then $\M$ is called \emph{infinite}.)
	
	So, if the $W^*$-algebra $\M$ is finite
	the coadjoint orbits of $\U(\M)\tto\Ll(\M)$ are the same as the coadjoint orbits of $U(\M)$ in $\M_*^+$.
	For an infinite $W^*$-algebra $\M$ the coadjoint orbits of $\U(\M)\tto\Ll(\M)$ split into disjoint unions of coadjoint orbits of the unitary group ${\rm U}(\M)$, as shown in Proposition~\ref{comp} below. 
	
	The weakly symplectic form 
	$\widetilde{\omega}_{\rho_0}$ 
	%$\tilde\omega_0$ 
	after restriction to 
	the connected component of $\rho_0\in\Oc_{\rho_0}$, that is,   
	$$\Ad^*_{{\rm U}(\M)}\rho_0\cong {\rm U}(\M)/{\rm U}(\M)_{\rho_0}\subseteq P_0/{\rm U}(\M)_{\rho_0}\cong\U(\M)_{\rho_0}$$  
	agrees with the one defined by Kirillov-Kostant-Souriau construction recalled in Remark~\ref{Kir}. 
	
	In order to see this, we 
	need the surjective mapping 
	$$\rg_{p_0}\colon {\rm U}(\M)\to P_0\cap \U(\M)_{\rm unit},\quad  \rg_{p_0}(u):=up_0$$
	and 
	the mapping
	$$\kappa\colon P_0\cap \U(\M)_{\rm unit}\to \Ad^*_{{\rm U}(\M)}\rho_0,\quad  
	\kappa(v):=v\rho_0 v^*.$$ 
	For any $\dot{u}_1,\dot{u}_2\in T_{\1}({\rm U}(\M))=\ug(\M)=\M^a$ one has 
	\begin{align*} 
	\bigl((\iota_{\gamma_0}\circ\rg_{p_0})^*(\omega)\bigr)(\dot{u}_1,\dot{u}_2)
	&=
	\omega(\dot{u}_1\gamma_0,\dot{u}_2\gamma_0) 
	=2\Im\langle\dot{u}_1\gamma_0\mid\dot{u}_2\gamma_0\rangle \\
	&=-2i(\langle\dot{u}_1\gamma_0\mid\dot{u}_2\gamma_0\rangle
	-\langle\dot{u}_2\gamma_0\mid\dot{u}_1\gamma_0\rangle) \\
	&=2i\langle\gamma_0\mid(\dot{u}_1\dot{u}_2-\dot{u}_2\dot{u}_1)\gamma_0\rangle
	\end{align*}
	and therefore 
	\begin{equation}
	\label{kir2}
	\bigl((\iota_{\gamma_0}\circ\rg_{p_0})^*(\omega)\bigr)(\dot{u}_1,\dot{u}_2)
	=\langle 2iE(\gamma_0),[\dot{u}_1,\dot{u}_2]\rangle.
	\end{equation}
	Since the 2-form $\widetilde{\omega}_{\rho_0}\in\Omega^2(\Oc_{\rho_0})$ is invariant under the transitive action of the groupoid $\U(\M)\tto\Ll(\M)$  on $\Oc_{\rho_0}$, 
	it follows 
	that 
	the restriction of 
	$\widetilde{\omega}_{\rho_0}$
	to the unitary coadjoint orbit $\Ad^*_{U(\M)}\rho_0$ is invariant with respect to the coadjoint action of the unitary group $U(\M)$. 
	It then follows by \eqref{kir2} that restriction of 
	$\widetilde{\omega}_{\rho_0}$
	to the unitary coadjoint orbit $\Ad^*_{U(\M)}\rho_0$ agrees with the weakly symplectic form obtained from $\rho_0$ by the Kirillov-Kostant-Souriau construction from Remark~\ref{Kir}. 
\end{rem}

\begin{cor}\label{cor:orbit}
The weakly symplectic form 
$\widetilde{\omega}_{\rho_0}$ 
from Theorem~\ref{thm:orbit} satisfies $S^\perp_{\widetilde{\omega}_{\rho_0}}=\{0\}$ 
hence it gives rise to a Poisson structure in the sense of Definition~\ref{def21}, for which the inclusion map $\Oc_{\rho_0}\hookrightarrow\M_{*}^{\rm h}$ is a Poisson map. 
\end{cor}

\begin{proof}
It follows by Remark~\ref{Kir2} that the restriction of 
$\widetilde{\omega}_{\rho_0}$  
to an arbitrary connected component of the groupoid orbit~$\Oc_{\rho_0}$ agrees with the weakly symplectic form obtained as in Remark~\ref{Kir} for the unitary group~$G=U(\M)$. 
Then the assertion follows by the conclusion of Remark~\ref{Kir}. 
\end{proof}

 \subsection[On $W^*$-algebras with special properties]{Special properties of $W^*$-algebras encoded in coadjoint orbits of $\U(\M)\tto \Ll(\M)$}
\label{Subsect5.2}

The following proposition  shows that for any $W^*$-algebra $\M$ the coadjoint orbits of the unitary group $U(\M)$ are the connected components of the orbits of the groupoid  $\U(\M)\ast \M_*^+\tto\M_*^+$.  
Here we use the topology on the groupoid orbits which corresponds to their \Banach manifold structures constructed above. 

\begin{prop}\label{comp}
	For an arbitrary $W^*$-algebra $\M$, if $\rho,\rho_0\in\M_*^+$ with $\rho$ in the coadjoint orbit $\mathcal{O}_{\rho_0}=\U(\M).\rho_0$ of the groupoid $\U(\M)\tto\M_*^+$, then the following properties are equivalent: 
	\begin{enumerate}[{\rm(i)}]
		\item\label{comp_item1}
		$\rho$ belongs to the connected component of $\rho_0$ in $\U(\M).\rho_0$. 
		\item\label{comp_item2} $\rho$ belongs to the unitary orbit  $U(\M).\rho_0$. 
		\item\label{comp_item3} The support projections $\sigma_*(\rho_0),\sigma_*(\rho)\in\Ll(\M)$ are unitary equivalent. 
	\end{enumerate}
\end{prop}

\begin{proof}
	Let $p_0:=\sigma(\rho_0)$. 
	
	``\eqref{comp_item1}$\Leftrightarrow$\eqref{comp_item3}'': 
	We know from Theorem~\ref{thm:32}\eqref{thm:32_item2} that the support mapping 
	$$\sigma_*\colon \Oc_{\rho_0}\to\Ll_{p_0}(\M)$$ 
	is a locally trivial fibration having its typical fiber $U(p_0\M p_0)/U_{\rho_0}$. 
	Since the unitary group of any $W^*$-algebra is connected, it follows that the fibers of the fibration $\sigma_*$ are connected. 
	Then it is straightforward to show that a subset $C\subseteq \Ll_{p_0}(\M)$ is connected if and only if its preimage $\sigma_*^{-1}(C)$ is a connected subset of $\Oc_{\rho_0}$. 
	On the other hand it is well known that the connected component of $p_0$ in $\Ll_{p_0}(\M)$ is the unitary orbit of $p_0$
	
	``\eqref{comp_item2}$\Rightarrow$\eqref{comp_item3}'': 
	This is clear since $\sigma_*(u\rho u^*)=u\sigma_*(\rho)u^*$ for every unitary element $u\in U(\M)$. 
	
	``\eqref{comp_item3}$\Rightarrow$\eqref{comp_item2}'': 
	Denoting $p:=\sigma_*(\rho)\in\Ll(\M)$, it follows by \eqref{comp_item3} that there exists $u\in U(\M)$ with $p=up_0u^*$. 
	
	On the other hand $\rho\in\U(\M).\rho_0$ by hypothesis, hence there exists $v\in\U(\M)$ with $\br(v)=p_0$ and $\rho=v\rho_0 v^*$. 
	It is straightforward to check that $u^*v\in\U(\M)$ with $\br(u^*v)=\bl(u^*v)=p_0$ and $vu^*\in\U(\M)$ with $\br(vu^*)=\bl(vu^*)=p$. 
	Using these relations, it is then easily seen that one can extend the partial isometry $v$ to the unitary operator 
	$$\widetilde{v}:=v+u(\1-p_0).$$
	More specifically,  one has $\widetilde{v}\in U(\M)$ 
	and $\widetilde{v}\rho_0\widetilde{v}^*=\rho$, 
	hence $\rho$ belongs to the orbit of $\rho_0$ with respect to the action of the unitary group $U(\M)$ on $\M_*^+$. 
\end{proof}

\begin{cor}\label{conn}
	If $\M$ is a $W^*$-algebra, then then the following assertions are equivalent: 
	\begin{enumerate}[{\rm(i)}]
		\item\label{conn_item1} Every orbit of $\U(\M)\ast \M_*^+\tto\M_*^+$ is connected.  
		\item\label{conn_item3} The orbits of the groupoid $\U(\M)\ast \M_*^+\tto\M_*^+$ are exactly the orbits of the coadjoint action of the \Banach Lie group $U(\M)$ on $\M_*^+$. 
		\item\label{conn_item2} The $W^*$-algebra $\M$ is finite. 
	\end{enumerate}
\end{cor}

\begin{proof}
	Use Proposition~\ref{comp}. 
\end{proof}

For the following corollary we recall that a $W^*$-algebra $\M$ is  \emph{type~III} 
if for every projection $p\in\Ll(\M)\setminus\{0\}$ there exists 
another projection $q\in\Ll(\M)$ which is Murray-von Neumann equivalent to $p$ and satisfies $q\le p$ and $q\ne p$. 

\begin{cor}\label{III}
	Let the $W^*$-algebra $\M$ be a factor, and consider the following assertions: 
	\begin{enumerate}[{\rm(i)}]
		\item\label{III_item1} Every orbit of the groupoid $\U(\M)\ast\M_*^+\tto\M_*^+$ is either $\{0\}$ or the disjoint union of exactly two 
		coadjoint orbits of the unitary group $U(\M)$. 
		\item\label{III_item2} The factor $\M$ is type~III. 
	\end{enumerate}
	Then  \eqref{III_item2}$\Rightarrow$\eqref{III_item1} and, if $\M$ has separable predual, then also \eqref{III_item1}$\Rightarrow$\eqref{III_item2}. 
\end{cor}

\begin{proof}
	``\eqref{III_item2}$\Rightarrow$\eqref{III_item1}'': 
	We must prove that if $0\ne \rho_0\in\M_*^+\setminus\{0\}$ then there exists $\rho_1\in\M_*^+$ for which one has the disjoint union of unitary group orbits
	$$\U(\M).\rho_0=U(\M).\rho_0\sqcup U(\M).\rho_1.$$ 
	Denoting $p_0:=\sigma(\rho_0)\in\Ll(\M)\setminus\{0\}$ 
	and $U(\M).p_0:=\{up_0u^*: u\in U(\M)\}$, one obtains the disjoint union
	$$
	\U(\M).\rho_0 =A\sqcup B$$
	where $A:=\{v\rho_0 v^*: \br(v)=p_0,\ \bl(v)\in U(\M).p_0\}$ 
	and $B:=\{v\rho_0 v^*:\br(v)=p_0,\ \bl(v)\not\in U(\M).p_0\}$. 
	Since $\M$ is a factor, it is well known that $\M$ is type~III if and only if any two projections from $\Ll(\M)\setminus\{0,\1\}$ are unitary equivalent. 
	Hence all functionals from the set $B$ have unitary equivalent supports. 
	On the other hand, all functionals from $A$ have unitary equivalent supports by definition of~$A$. 
	Then, by Proposition~\ref{comp}, one has $A=U(\M).\rho_0$ and also $B=U(\M).\rho_1$ for arbitrary $\rho_1\in B$. 
	
	``\eqref{III_item1}$\Rightarrow$\eqref{III_item2}'': 
	Since $\M$ has separable predual, there exists $\rho\in\M_*^+$ with $\sigma(\rho)=1$, that is, $\rho$ is faithful. 
	For arbitrary $p\in\Ll(\M)$ if we define $\rho_p\in\M_*^+$ by $\rho_p(x):=\rho(pxp)$, then $\sigma(\rho_p)=p$. 
	Moreover, for every $v\in\U(\M)$ with $\br(v)=p$, one has $\sigma(v\rho_pv^*)=\bl(v)$. 
	Hence 
	$$\{\sigma_*(\varphi):\varphi\in\U(\M).\rho_p\}=\Ll_p(\M)\text{ for every }p\in\Ll(\M).$$
	Then, taking into account the assumption \eqref{III_item1} and Proposition~\ref{comp}, we obtain the following conclusion: 
	for every $p\in\Ll(\M)$, if $p_1,p_2\in\Ll_p(\M)$ and neither $p_1$ nor $p_2$ is unitary equivalent to $p$, then $p_1$ is unitary equivalent to $p_2$. 
	It is easily seen that this condition is not satisfied by factors of type~I or type~II, hence $\M$ is type~III. 
\end{proof}

In the following  we use the notation $$(\M_*^+)_r:=\{\rho\in\M_*^+:\Vert\rho\Vert=r\}$$ 
for every $r\in[0,\infty)$. 
We also recall from \cite[\S 4]{Co73} that if $\M$ is a factor with separable predual, then $\M$ is called \emph{type~III$_1$} 
if it is type~III and for every $\rho\in\M_*^+$ with $\Vert\rho\Vert=1$ and $\sigma_*(\rho)=\1$ 
its corresponding modular operator $\Delta_\rho$ has its spectrum equal to~$[0,\infty)$. 

\begin{prop}\label{III1}
	If the $W^*$-algebra $\M$ is a factor with separable predual, 
	then the following conditions are equivalent: 
	\begin{enumerate}[{\rm(i)}]
		\item\label{III1_item1} $\M$ is type~III$_1$. 
		\item\label{III1_item2} For every $\rho_0\in\M_*^+$ its unitary orbit $U(\M).\rho_0$ is dense in $(\M_*^+)_{\Vert\rho_0\Vert}$. 
		\item\label{III1_item3} For every $\rho_0\in\M_*^+$ its groupoid orbit $\U(\M).\rho_0$ is dense in $(\M_*^+)_{\Vert\rho_0\Vert}$. 
	\end{enumerate}
\end{prop}

\begin{proof}
	``\eqref{III1_item1}$\Leftrightarrow$\eqref{III1_item2}'': 	
	This is the Connes-St\o rmer theorem \cite[Th. 4]{CoSt78} (or \cite[Ch. XII, Th. 5.12]{Ta03a}). 
	
	``\eqref{III1_item2}$\Rightarrow$\eqref{III1_item3}'': 	
	This is clear since 
	$U(\M).\rho_0\subseteq\U(\M).\rho_0$. 
	
	``\eqref{III1_item3}$\Rightarrow$\eqref{III1_item1}'': 	
	If $\M$ is finite (hence type I$_n$ for $n=1,2,\dots$ or type II$_1$), then it has a faithful trace $\tau_0\in \M_*^+$, hence $\sigma(\tau_0)=\1$ and $\Vert\tau_0\Vert=1$. 
	For every $v\in\U(\M)$ with $v^*v=\sigma(\tau_0)=1$ we also have $vv^*=\1$, since $\M$ is finite, and then $v\tau_0 v^*=\tau_0$ by the trace property of~$\tau_0$. 
	Thus $\U(\M).\tau_0=\{\tau_0\}$, which is not dense in $(\M_*^+)_1$. 
	
	If $\M$ is a semifinite factor (hence type I$_\infty$ or type II$_\infty$), then it has a faithful semifinite normal trace $\tau\colon\M^+\to[0,\infty]$. 
	For any $p\in\Ll(\M)$ with $\tau(p)<\infty$ define $\tau_p\in(\M_*^+)_1$ by $\tau_p(x):=\frac{1}{\tau(p)}\tau(px)$ for all $x\in\M$, hence $\sigma(\tau_p)=p$. 
	If $v\in\U(\M)$ with $v^*v=p$  
	then, using the trace property of $\tau$, we obtain  
	\begin{align*}
	(v\tau_p v^*)(x)
	&=\tau_p(v^*xv)  
	=\frac{1}{\tau(p)}\tau(pv^*xv) 
	=\frac{1}{\tau(p)}\tau(vpv^*x) 
	=\frac{1}{\tau(p)}\tau(vv^*x) \\
	&=\frac{1}{\tau(vv^*)}\tau(vv^*x)
	\end{align*}
	hence $v\tau_p v^*=\tau_{vv^*}$. 
	This shows that 
	$$\U(\M).\tau_p=\{\tau_q:q\in\Ll_p(\M)\}.$$
	On the other hand, as noted in \cite[page 92]{CoStHa85}, for all $p,q\in\Ll(\M)$ with $p\le q$ and $\tau(q)<\infty$ one has 
	$\Vert \tau_p-\tau_q\Vert\ge 2\Bigl(1-\frac{\tau(p)}{\tau(q)}\Bigr)$. 
	It then follows that $\U(\M).\tau_p$ is not dense in $(\M_*^+)$ for any $p\in\Ll(\M)$ with $\tau(p)<\infty$. 
	
	If $\M$ is type~III, then let $\rho_0\in\M_*^+\setminus\{0\}$ with $\Vert\rho_0\Vert=:r_0$. %and $\sigma(\rho_0)=\1$. 
	Using Corollary~\ref{III}, one then obtains the disjoint union into two unitary orbits 
	$$\Oc_{\rho_0}=\U(\M).\rho_0\sqcup \U(\M).\rho_{00}$$
	for suitable $\rho_{00}\in \Oc_{\rho_0}$
	Since $\Oc_{\rho_0}$ is dense in $(\M_*^+)_{r_0}$ by hypothesis, it follows that at least one of these unitary orbits is dense in $(\M_*^+)_{r_0}$. 
	
	We now prove that actually all unitary orbits of elements of $(\M_*^+)_{r_0}$ are dense in $(\M_*^+)_{r_0}$. 
	To this end, for any $\rho_1,\rho_2\in\M_*^+$ let us denote 
	$$d([\rho_1],[\rho_2]):=\inf\{\Vert u_1\rho_1u_1^*-u_2\rho_2u_2^*\Vert :u_1,u_2\in U(\M)\}.$$
	The following facts are straightforward: 
	\begin{itemize}%[{\rm(i)}]
		\item For all $\rho_1,\rho_2,\rho_3\in\M_*^+$ one has 
		$d([\rho_1],[\rho_3])\le d([\rho_1],[\rho_2])+d([\rho_2],[\rho_3])$. 
		\item If $\rho_1\in\M_*^+$ and $r:=\Vert \rho_1\Vert$, then the unitary orbit $U(\M).\rho_1$ is norm-dense in $(\M_*^+)_r$ if and only if for all $\rho_2\in\M_*^+$ with $\Vert\rho_2\Vert=r$ one has $d([\rho_1],[\rho_2])=0$. 
	\end{itemize}
	It then easily follows by these facts that if there exists $\rho_1\in\M_*^+$ whose unitary orbit $U(\M).\rho_1$ is norm-dense in $(\M_*^+)_r$, where $r:=\Vert \rho_1\Vert$, then for every $\rho_2\in\M_*^+$ with $\Vert\rho_2\Vert=r$ the unitary orbit $U(\M).\rho_2$ is norm-dense in $(\M_*^+)_r$. 
	
	Consequently, for every $\rho\in (\M_*^+)_{r_0}$ its unitary orbit $U(\M).\rho$ is dense in $(\M_*^+)_{r_0}$. 
	It then follows by \eqref{III1_item1}$\Leftrightarrow$\eqref{III1_item2} that $\M$ is type~III$_1$. 
\end{proof}

 \section{%The standard 
 	Presymplectic  
 	structure of the 
 standard	Lie groupoid $\widetilde{\H}\tto\M_*^+$}
\label{Sect6}

This section is the core of the present paper. 
We consider the symplectic dual pair~\eqref{diagdual3}  
in the case when the von Neumann algebra $\M\subseteq L^\infty(\H)$ is in standard form in the sense of \cite{Ha73} and \cite{Ha75}, as recalled below. 
In this setting of a standard form, we show that the complex Hilbert space $\H$ has the structure of a groupoid $\H\tto\M_*^+$ which we call  the \emph{standard groupoid} of $\M$ and actually has several isomorphic realizations that involve various data of the standard form of $\M$ (Theorem~\ref{thm:53} and Proposition~\ref{isomorphisms}). 
What is crucial for the present paper is that the standard groupoid actually has the structure of a \Banach Lie groupoid that is denoted $\widetilde{\H}\tto\M_*^+$ in order to emphasize that the manifold structure of its total space is different from the ordinary Hilbert space structure of $\H$. 
More specifically, $\widetilde{\H}$ is a Banach foliation of $\H$, 
in the sense that the identity map is a weak immersion $\widetilde{\H}\to\H$ (Theorem~\ref{grpd_act}). 
Finally, we show that the symplectic structure of $\H$ is compatible with the Lie groupoid structure $\widetilde{\H}\tto\M_*^+$ in the following sense: 
By pull-back via the weak immersion mentioned above, one obtains a presymplectic 2-form $\widetilde{\bomega}\in\Omega^2(\widetilde{\H})$ that is multiplicative in the usual sense of finite-dimensional Lie groupoid theory (Proposition~\ref{graph}).  
We introduce the notation  
$$(\widetilde{\H},\widetilde{\bomega})\tto\M_*^+$$ 
for the Lie groupoid $\widetilde{\H}\tto \M_*^+$ endowed with its multiplicative presymplectic form~$\widetilde{\bomega}$. 
We also obtain quite complete information on the leaf space of the degeneracy-foliation of $\widetilde{\bomega}$ (Proposition~\ref{SQ6}). 

 \subsection{The standard  groupoid $\H\tto\M_*^+$}
\label{Subsect6.1}

A sufficient condition for the standard form is to have a separating cyclic vector $\Omega\in\H$. 
This assumption allows us to define 
\begin{equation*}
S_0(x\Omega):=x^*\Omega\text{ and }F_0(x'\Omega):=x'^*\Omega
\end{equation*}
for all $x\in\M$ and $x'\in\M'$. 
These are antilinear operators having their dense domains $D(S_0)=\M\Omega$ and $D(F_0)=\M'\Omega$, respectively. 
The closures of these operators 
$$S:=\overline{S_0}\text{ and }F:=\overline{F_0}$$ 
have the polar decompositions 
\begin{equation}\label{JD}
S=J\Delta^{1/2}\text{ and }F=J\Delta^{-1/2}
\end{equation}
where $\Delta$ is a positive self-adjoint operator called the modular operator, while $J$ is an anti-unitary involution, that is, $J=J^*$ and $J^2=\1$. 
One also has 
\begin{equation}
\Delta=FS,\ \Delta^{-1}=SF,\ \Delta^{-1/2}=J\Delta^{1/2}J.
\end{equation}
More details can be found for instance in \cite[\S 2.5]{LP2}. 

One has by the theorem of Tomita-Takesaki 
\begin{equation}
j(\M)=\M'
\end{equation}
where $j(x):=JxJ$ for all $x\in\M$, and 
\begin{equation}
\Delta^{it}\M\Delta^{-it}=\M
\end{equation}
for all $t\in\R$. 
In the modular theory of Tomita-Takesaki, 
one associates to the pair $(\H,\Omega)$ the natural positive cone~$\P$. 
 This is a self-dual cone defined by 
\begin{equation}
\P:=\overline{\{xj(x)\Omega\mid x\in\M\}}
\end{equation}
and has the following properties: 
\begin{align}
	& \P = \overline{\Delta^{1/4}\M^+\Omega}
	=\overline{\Delta^{-1/4}\M'^+\Omega},\\
	& \Delta^{it}\P=\P \text{ for all }t\in\R,\\
	& J\xi=\xi\text{ for all }\xi\in\P,\\
	& \P\cap (-\P)=\{0\}.
\end{align}
Any $\xi\in\H$ with $J\xi=\xi$ has a unique decomposition $\xi=\xi_+-\xi_-$ with $\xi_\pm\in\P$ and $\xi_+\perp\xi_-$. 
See \cite[Prop. 2.5.28]{LP2}. 
%for all this. 

The following universality property of $\P$ holds: 
If $\xi\in\P$ is a cyclic vector for $\M$, with its corresponding modular conjugation $\J_\xi$ and natural positive cone $\P_\xi$, then one has $J_\xi=J$ and $\P_\xi=\P$ by \cite[Prop. 2.5.30]{LP2}. 

Replacing $\M$ by $\M'$ in the above considerations, we note that for the modular objects associated to $\M'$ are given by $J'=J$, $\P'=\P$, and $\Delta'=\Delta^{-1}$.

Without any assumption on existence of cyclic separating vectors, a  4-tuple $(\M,\H,J,\P)$ is called a \emph{standard form} \cite[Def. 2.1]{Ha75} of the von Neumann algebra $\M\subseteq L^\infty(\H)$ if 
$J\colon\H\to\H$ is a conjugation and $\P\subseteq\H$ is a self-dual cone satisfying  $J\M J=\M'$, $JxJ=x^*$ if $x\in\M\cap\M'$, 
$J\rho=\rho$ for all $\rho\in\P$, and $x j(x)\P\subseteq\P$ for all $x\in\M$, where $j(x):=JxJ\in\M'$ as above.  

With this terminology, we now reformulate in the following way one of the basic theorems of modular theory.  
See \cite[Th. 2.17]{Ha73} or, alternatively, \cite[Lemma 2.10]{Ha75}. 

\begin{prop}\label{basic_homeo}
	If $(\M,\H,J,\P)$ is a standard form, then 
	the expectation mappings $E\colon\H\to\M_*^+$ and $E'\colon\H\to{\M'}_*^+$ when restricted to the natural positive cone $\P\subseteq\H$ give homeomorphisms 
	\begin{equation}\label{ar1}
	\xymatrix{
		\M_*^+ & \ar[l]_{E\vert_{\P}}  \P  \ar[r]^{E'\vert_{\P}} & {\M'}_*^+	}
	\end{equation}	
where $\P$, $\M_*^+$, and ${\M'}_*^+$ are endowed with their topologies  inherited from the norm topologies of $\H$, $\M_*$, and $\M'_*$, respectively. 
\end{prop}

Now define $j_*\colon \M'_*\to\M_*$ by 
$$\langle j_*(\rho'),x\rangle=\overline{\langle\rho',j(x)\rangle}$$
for all $\rho'\in\M'_*$ and $x\in\M$. 
Then one has for every $\gamma\in\H$ and $x\in\M$
\begin{align*}
\langle j_*(\vert\gamma\rangle \langle\gamma\vert),x\rangle
& =\overline{\langle\vert\gamma\rangle \langle\gamma\vert,j(x)\rangle} 
 =\langle j(x)\gamma\mid\gamma\rangle 
 =\langle JxJ\gamma\mid\gamma\rangle 
 =\langle J\gamma\mid xJ\gamma\rangle \\
&=\langle \vert J\gamma\rangle \langle J\gamma\vert, x\rangle 
\end{align*}
hence 
\begin{equation}\label{j*}
j_*\circ E'=E\circ J\qquad\text{ and }\qquad j'_*\circ E=E'\circ J.
\end{equation}
By 
Proposition \ref{basic_homeo} one has the homeomorphisms
\begin{equation}
\label{PHP}
\xymatrix{
	\M_*^+ \ar[r]^{\bepsilon} &   \P   & {\M'}_*^+ \ar[l]_{\bepsilon'}	}
\end{equation}	
inverse to the ones presented in \eqref{ar1}. 
Using \eqref{PHP}, one obtains surjective continuous maps
	\begin{equation}	
	\xymatrix{
		\P & \ar[l]_{\tilde \bt}  \H  \ar[r]^{\tilde \bs} & \P	}
	\end{equation}
		defined by 
$$\tilde \bs:=\bepsilon'\circ E'\text{ and }\tilde \bt:=\bepsilon\circ E.$$ 
		It follows by \eqref{j*} that
		\begin{equation}\label{ts}
		\tilde \bt=\tilde \bs\circ J\qquad {\rm and}\qquad \tilde \bs=\tilde \bt\circ J.
		\end{equation}
		From now on,  
		the following notation will be used: 
		\begin{equation}\label{gst} 
		\vert\gamma\vert:=\tilde \bs(\gamma)\qquad {\rm and}\qquad \vert\gamma\vert':=\tilde \bt(\gamma),
		\end{equation}
		where $\gamma\in \H$. As a consequence of Proposition \ref{prop:2} one obtains two polar decompositions
		\begin{equation}\label{g} \gamma=u|\gamma|\qquad{\rm and }\qquad \gamma=u'|\gamma|'
		\end{equation}
		of $\gamma\in \H$, where the partial isometries $u\in \U(\M)$ and $u'\in \U(\M')$ are defined in a unique way by $$u^*u=[\M'|\gamma|]=:p_{|\gamma|}$$ 
		and 
		$$u'^*u'=[\M|\gamma|']=:p'_{|\gamma|'}.$$ 
		Note also that $uu^*=[\M'\gamma]$ and $u'u'^*=[\M\gamma]$. 
		The components of these polar decompositions are related by 
		\begin{align} 
		\label{ju1} 
		u'& =j(u^*) \\
		\label{ju2} 
		|\gamma|' & =u j(u)|\gamma|=:\beta(u)|\gamma|.
		\end{align}
		In order to show \eqref{ju1}--\eqref{ju2} we note that substituting $u'$ and $|\gamma|'$ given by \eqref{ju1}--\eqref{ju2} into the second formula of \eqref{g} one obtains 
		\begin{equation} 
		u'|\gamma|'=j(u^*)u j(u)|\gamma|=uj(u^*u)|\gamma|=uJu^*u|\gamma|=uJ|\gamma|=u|\gamma|=\gamma.
		\end{equation}
		So, \eqref{ju1}--\eqref{ju2} follow by the uniqueness  of the polar decomposition $\gamma=u'|\gamma|'$. 
		From \eqref{ts}, \eqref{gst}, and \eqref{ju2}, one obtains that 
		\begin{equation}\label{ga} 
		|\gamma|'=|J\gamma|,\qquad |\gamma|=|J\gamma|'
		\qquad {\rm and}\qquad  
		J\gamma=u^*|J\gamma|.
		\end{equation}
		
		\begin{lem}\label{lem:52}
		\begin{enumerate}[{\rm(i)}]
		\item\label{lem:52_item1} 
		The expectation map $E:\P\to \M_*^+$ satisfies
		\begin{align}
		\label{Egam} 
		\sigma_*(E(|\gamma|))& =[\M'|\gamma|]\\
		\label{Egamma} E(\beta(u)|\gamma|)& =E(uj(u)|\gamma|)=uE(|\gamma|)u^*
		\end{align}
		where $u^*u=\sigma_*(E(|\gamma|))$.
		\item\label{lem:52_item2} 
		One has the action $\beta:\U(\M)\ast_\mu \P\to \P$ of $\U(\M)\tto \Ll(\M)$ on $\P$ defined by \eqref{ju2} with its corresponding momentum map 
		$$\mu:\P\to\Ll(\M),\quad \mu(|\gamma|):=[\M'\gamma],$$ 
		see also \eqref{mom}.
		\item\label{lem:52_item3} 
		The momentum maps $\mu:\P\to\Ll(\M)$ and $\sigma_*:\M_*^+\to \Ll(\M)$ satisfy
	\begin{equation}\label{diagdual31}
\xymatrix
{\P 	 \ar[dr]_{\mu }  	\ar[rr]^E_\sim& &  {\M^+_*} \ar[dl]^{\sigma_*}
	    \\
  	&{\Ll(\M)}&  }
\end{equation}		
	hence, the expectation map $E:\P\to \M_*^+$ intertwines the actions $\Ad_*:\U(\M)\ast_{\sigma_*}\M^+_*\to \M^+_*$ and $\beta:\U(\M)\ast_\mu \P\to \P$  of $\U(\M)\tto \Ll(\M)$ defined in \eqref{coaction} and \eqref{ju2}, respectively.
		\end{enumerate}
		\end{lem}
	
\begin{proof} 
%(i) 
\eqref{lem:52_item1}			
Equality~\eqref{Egam} follows from Proposition~\ref{prop:44}. For every $x\in \M$ one has 
\allowdisplaybreaks
\begin{align*} 
\langle E(\beta(u)\vert\gamma\vert),x\rangle
& =\langle E(uj(u)\vert\gamma\vert),x\rangle 
=\langle uj(u)\vert\gamma\vert \mid xuj(u)\vert\gamma\vert\rangle  \\
& =\langle \vert\gamma\vert \mid u^*j(u^*)xuj(u)\vert\gamma\vert\rangle  
=\langle \vert\gamma\vert \mid u^*xuj(u^*u)\vert\gamma\vert\rangle  \\
&=\langle \vert\gamma\vert\mid u^*xuJu^*u\vert\gamma\vert\rangle  
 =\langle \vert\gamma\vert\mid u^*xuJ[\M'\vert\gamma\vert]\vert\gamma\vert\rangle  \\
& =\langle \vert\gamma\vert\mid u^*xuJ\vert\gamma\vert\rangle  
 =\langle \vert\gamma\vert \mid u^*xu\vert\gamma\vert\rangle  \\
&=\langle uE(\vert\gamma\vert)u^*,x\rangle. 
\end{align*}
The above proves \eqref{Egamma}.
		
%(ii) 
\eqref{lem:52_item2}
Let us take $(u_1,j(u_2)u_2\vert\gamma\vert),(u_2,\vert\gamma\vert)\in \U(\M)\ast_\mu\P$, i.e., 
$$u_1^*u_1=[\M'j(u_2)u_2\vert\gamma\vert]\text{ and } u_2^*u_2=[\M'\vert\gamma\vert].$$ 
Assuming $u_1^*u_1=u_2u_2^*$ and using 
$$j(u_1)u_1(j(u_2)u_2\vert\gamma\vert)=(j(u_1u_2)u_1u_2)\vert\gamma\vert$$
we obtain that $(u_1u_2,j(u_1u_2)u_1u_2\vert\gamma\vert)\in \U(\M)\ast_\mu\P$, i.e., 
$$(u_1u_2)^*u_1u_2=[\M'j(u_1u_2)u_1u_2\vert\gamma\vert].$$
		
%(iii) 
\eqref{lem:52_item3}
The commutativity of the diagram \eqref{diagdual31} follows from \eqref{Egam}. The second statement of~
%(iii) 
\eqref{lem:52_item3} 
follows from \eqref{Egamma} and \eqref{diagdual31}.
\end{proof}

Now, let us define the product
\begin{equation}
\label{prod} 
\H*\H\ni(\gamma_1,\gamma_2)\mapsto \gamma_1\bullet\gamma_2\in \H,
\end{equation}
where $\H*\H:=\{(\gamma_1,\gamma_2)\in \H\times\H: \tilde\bs(\gamma_1)=\tilde \bt(\gamma_2)\}$ as follows. 
We take the polar  decompositions $\gamma_1=u_1|\gamma_1|$ and $\gamma_2=u_2|\gamma_2|$. 
From $\tilde\bs(\gamma_1)=\tilde \bt(\gamma_2)$ we have $|\gamma_1|=|\gamma_2|'=u_2j(u_2)|\gamma_2|$. 
%Thus and from point (i) of 
Then, by 
Lemma \ref{lem:52}\eqref{lem:52_item1}, see \eqref{Egam} and \eqref{Egamma}, it follows that $u_1^*u_1=u_2u_2^*$. 
Hence, the right hand side of 
\begin{equation}\label{ast1} 
\gamma_1\bullet\gamma_2:=u_1u_2|\gamma_2|
\end{equation}
is well defined.
		
Finally, by 
\begin{equation}
\label{incl}\tilde\bepsilon:\P\to \H
\end{equation}
we denote the inclusion of $\P$ into $\H$.

\begin{thm}\label{thm:53}%[polar decomposition of vectors]
		If $(\M,\H,J,\P)$ is a standard form of  $\M$, then $\H$ is the space of arrows (morphisms) of the groupoid $\H\tto\P$ on base $\P$, having the following structure maps:
\begin{itemize}
	\item the inverse map $J:\H\to \H$
	
	\item the source map $\tilde\bs:\H\to \P$
	
	\item the target map $\tilde\bt:\H\to \P$
	
	\item the multiplication $\H*\H\to \H$
	
	\item the object inclusion map $\tilde\bepsilon:\P\hookrightarrow \H$
\end{itemize}
defined in \eqref{JD}, \eqref{ts}, \eqref{prod},  \eqref{incl}, respectively. 
	\end{thm} 

	\begin{proof} 
	To prove that $\H\tto \P$ is a groupoid, we must check the following conditions (cf. \cite[Def. 1.1.1]{Ma05}):
\begin{enumerate}
	\item\label{grpd_lemma_proof_item1} 
	If $(\gamma_1,\gamma_2)\in\H\ast\H$, then $(J\gamma_2,J\gamma_1)\in\H\ast\H$ and 
	$J(\gamma_1\bullet\gamma_2)=(J\gamma_2)\bullet(J\gamma_1)$. 
	\item\label{grpd_lemma_proof_item2} $\widetilde\bs(\gamma_1\bullet\gamma_2)=\widetilde\bs(\gamma_2)$ 
	and $\widetilde\bt(\gamma_1\bullet\gamma_2)=\widetilde\bt(\gamma_1)$ 
	if $(\gamma_1,\gamma_2)\in\H\ast\H$. 
	\item\label{grpd_lemma_proof_item3} $(\gamma_1\bullet\gamma_2)\bullet\gamma_3=\gamma_1\bullet(\gamma_2\bullet\gamma_3)$ if $(\gamma_1,\gamma_2),(\gamma_2,\gamma_3)\in\H\ast\H$.
	\item\label{grpd_lemma_proof_item4} $\widetilde\bs(\widetilde\bepsilon(\rho))=\widetilde\bt(\widetilde\bepsilon(\rho))=\rho$ if $\rho\in\P$. 
	\item\label{grpd_lemma_proof_item5} $\gamma\bullet\widetilde\bepsilon(\widetilde\bs(\gamma))
	=\widetilde\bepsilon(\widetilde\bt(\gamma))\bullet\gamma=\gamma$ for all $\gamma\in\H$. 
	\item\label{grpd_lemma_proof_item6} 
	One has $(\gamma,J\gamma),(J\gamma,\gamma)\in\H\ast\H$, and 
	moreover $J\gamma\bullet\gamma=\widetilde\bepsilon(\widetilde\bs(\gamma))$ 
	and $\gamma\bullet J\gamma=\widetilde\bepsilon(\widetilde\bt(\gamma))$, for all $\gamma\in\H$. 
\end{enumerate}
	In order to do that, we will use the properties of the structural maps mentioned above.  
	We will use also the polar decompositions $\gamma_i=u_i\vert\gamma_i\vert$ for $i=1,2,3$ and $\gamma=u\vert\gamma\vert$.

\eqref{grpd_lemma_proof_item1} 
If $(\gamma_1,\gamma_2)\in\H\ast\H$, 
then $\vert\gamma_1\vert=\vert J\gamma_2\vert$, 
that is, $\vert J\gamma_2\vert=\vert J(J\gamma_1)\vert$, 
hence $(J\gamma_2,J\gamma_1)\in\H\ast\H$. 
Moreover, by \eqref{ast1}, one has the polar decomposition 
$\gamma_1\bullet\gamma_2=(u_1u_2)\vert\gamma_2\vert$, hence, 
using \eqref{ju1}, \eqref{ju2}, and \eqref{ga}, we obtain
\begin{align*}
J(\gamma_1\bullet\gamma_2)
&= J((u_1u_2)\vert\gamma_2\vert) 
=(u_1u_2)^*\vert J((u_1u_2)\vert\gamma_2\vert)\vert 
=(u_1u_2)^*(u_1u_2)j(u_1u_2)\vert\gamma_2\vert\\
&=u_2^*u_2j(u_1u_2)\vert\gamma_2\vert 
=u_2^*u_2j(u_1) j(u_2)\vert\gamma_2\vert 
=u_2^* j(u_1) u_2 j(u_2)\vert\gamma_2\vert \\
&=u_2^* j(u_1) \vert J\gamma_2\vert 
 =u_2^* j(u_1) \vert \gamma_1\vert 
=u_2^* J\gamma_1 
=(J\gamma_2)\bullet(J\gamma_1),
\end{align*}
where we also used the compatibility assumption $\vert J\gamma_2\vert=\vert\gamma_1\vert$.  
Note that this assumption implies $u_1^*u_1=u_2u_2^*$.

\eqref{grpd_lemma_proof_item2} 
Since $(\gamma_1,\gamma_2)\in\H\ast\H$, 
one has $\vert\gamma_1\vert=\vert J\gamma_2\vert$. 
Then, by \eqref{ast1} as above, 
one has the polar decomposition 
$\gamma_1\bullet\gamma_2=(u_1u_2)\vert\gamma_2\vert$, hence 
$$\widetilde\bs(\gamma_1\bullet\gamma_2)
=\vert\gamma_1\bullet\gamma_2\vert=\vert\gamma_2\vert
=\widetilde\bs(\gamma_2).$$
Moreover, using the above \eqref{grpd_lemma_proof_item1}, 
and the equality $\widetilde\bt=\widetilde\bs\circ J\colon\H\to\P$, one obtains 
$$\widetilde\bt(\gamma_1\bullet\gamma_2)
=\widetilde\bs(J(\gamma_1\bullet\gamma_2))
=\widetilde\bs((J\gamma_2)\bullet(J\gamma_1))
=\widetilde\bs(J\gamma_1)
=\widetilde\bt(\gamma_1)$$
as required.

\eqref{grpd_lemma_proof_item3}
One has the polar decompositions 
$\gamma_1\bullet\gamma_2=(u_1u_2)\vert\gamma_2\vert$ 
and $\gamma_2\bullet\gamma_3=(u_2u_3)\vert\gamma_3\vert$, hence 
$$(\gamma_1\bullet\gamma_2)\bullet\gamma_3
=(u_1u_2)u_3\vert\gamma_3\vert
=u_1(u_2u_3)\vert\gamma_3\vert
=\gamma_1\bullet(\gamma_2\bullet\gamma_3)$$
as claimed. 

\eqref{grpd_lemma_proof_item4}
For any $\rho\in\P$ one has $\vert\rho\vert=\rho$ and $J\rho=\rho$, 
hence $\widetilde\bs(\rho)=\widetilde\bt(\rho)=\rho$.  
The assertion then follows since $\widetilde\epsilon(\rho)=\rho$. 

\eqref{grpd_lemma_proof_item5}
One has 
$\bepsilon(\widetilde\bs(\gamma))=\bepsilon(\vert\gamma\vert)
=\vert\gamma\vert\in\P$ 
hence
$$\gamma\bullet\widetilde\bepsilon(\widetilde\bs(\gamma))
=u\widetilde\bepsilon(\widetilde\bs(\gamma))
=u\vert\gamma\vert=\gamma.$$
Replacing $\gamma$ by $J\gamma$ in the above equality we obtain 
$(J\gamma)\bullet\widetilde\bepsilon(\widetilde\bs(J\gamma))
=J\gamma$ hence, using 
$\widetilde\bt=\widetilde\bs\circ J$, 
it follows that 
$(J\gamma)\bullet\widetilde\bepsilon(\widetilde\bt(\gamma))
=J\gamma$. 
Now, applying $J$ to both sides of this equality and using \eqref{grpd_lemma_proof_item1}, one obtains 
$\widetilde\bepsilon(\widetilde\bt(\gamma))\bullet\gamma=\gamma$, 
as required. 

\eqref{grpd_lemma_proof_item6} 
It is clear that 
$\widetilde\bs(\gamma)=\vert\gamma\vert=\widetilde\bt(J\gamma)$ 
and $\widetilde\bs(J\gamma)=\vert J\gamma\vert=\widetilde\bt(\gamma)$, 
hence  $(\gamma,J\gamma),(J\gamma,\gamma)\in\H\ast\H$. 
Moreover, using the polar decomposition $J\gamma=u^*\vert J\gamma\vert$
one obtains 
$$J\gamma\bullet\gamma
=u^*u\vert\gamma\vert
=\vert\gamma\vert
=\widetilde\bepsilon(\widetilde\bs(\gamma)).$$
Replacing $\gamma$ by $J\gamma$ in the above equality, 
one has 
$$\gamma\bullet J\gamma
=\widetilde\bepsilon(\widetilde\bs(J\gamma))
=\widetilde\bepsilon(\widetilde\bt(\gamma))$$
and this completes the proof. 
\end{proof}

	\begin{rem}\label{rem:54}
	Since $\H\tto\P$ is a groupoid  and 
	$E\vert_{\P}\colon\P\to\M_*^+$ is bijective, it then follows that $\H\tto\M_*^+$ is a groupoid for which  $J$ is  the inversion map, while the groupoid product is given by  \eqref{prod}. 
	However, the source, target and objects inclusion maps are defined by 
\begin{align}
\bs:= & (E|_\P)\circ \widetilde\bs\\
\bt:= & (E|_\P)\circ \widetilde\bt\\
\bepsilon:= & \widetilde\bepsilon \circ  (E|_\P)^{-1},
\end{align}
respectively.
\end{rem}

	In the following, we will call $\H\tto\M_*^+$ the \emph{standard groupoid}. 
	The following proposition summarizes the various realizations of the standard groupoid.
	
	\begin{prop}\label{isomorphisms}
	One has the natural isomorphisms of groupoids 
	\begin{equation}\label{531}
\xymatrix{
		\U(\M)\ast_{\sigma_*} \M_*^+ \ar@<-.5ex>[d] \ar@<.5ex>[d]  
		\ar[r]^{(\id,\bepsilon)} 
	&\U(\M)\ast_\mu \P \ar@<-.5ex>[d] \ar@<.5ex>[d]  \ar[r]^{\quad{\Theta}}
	\ar[r]
	& \H \ar@<-.5ex>[d] \ar@<.5ex>[d] \ar[r]^{\id}
	& \H \ar@<-.5ex>[d] \ar@<.5ex>[d]
	 \\
	  \M_*^+  \ar[r]^{\bepsilon} 
	& \P \ar[r]^{\id}
	& \P \ar[r]^{E\vert_{\P}}
	& \M_*^+}
\end{equation}
where $\Theta:\U(\M)\ast_\mu\P\stackrel{\sim}{\rightarrow}\H$ is defined by $\Theta(u,\rho):=u\rho$ for $u^*u=\mu(\rho)=[\M'\rho]$.
\end{prop}

\begin{proof} 
By straightforward verification.
\end{proof}

 \subsection{Lie groupoid structure of $\widetilde\H\tto\M_*^+$}
\label{Subsect6.2}

In order to prove Theorem \ref{grpd_act}, which is crucial for the investigation of the Lie groupoid structure of $\H\tto\M_*^+$, we formulate  the following two lemmas. 
We recall from \eqref{ju2} the notation 
$\beta(u)=u j(u)$ for any partial isometry $u\in\U(\M)$. 

\begin{lem}\label{lem:56}
If $\gamma_0\in\P$, $\rho_0=E(\gamma_0)\in\M_*^+$, $p_0=\sigma_*(\rho_0)\in\Ll(\M)$, and $P_0=\br^{-1}(p_0)\subseteq\U(\M)$, 
then the following assertions hold: 
\begin{enumerate}[{\rm(i)}]
	\item\label{SQ1} 
	If $u\in\M$ and $u^*u\gamma_0=\gamma_0$, 
	then $\vert u\gamma_0\vert'=\beta(u)\gamma_0\in\P$. 
	\item\label{SQ2_item1} 
	If $a\in\M$ satisfies $j(a)\gamma_0=a^*\gamma_0$,  then 
	$\rho_0 a=a\rho_0$. 
	\item\label{SQ2_item2} One has $U_{\rho_0}=\{u\in P_0: \vert u\gamma_0\vert=\gamma_0\}
	=\{u\in P_0: \beta(u)\gamma_0=\gamma_0\}
	=\{u\in P_0: j(u)\gamma_0=u^*\gamma_0\}$.
	\item\label{SQ2_item3} 
	The mapping 
	$\beta_{\gamma_0}\colon P_0\to\H$,   $\beta_{\gamma_0}(u):=\beta(u)\gamma_0$, 
	satisfies 
	\begin{equation}\label{b}\beta_{\gamma_0}^{-1}(\beta_{\gamma_0}(u))=\{ug: g\in U_{\rho_0}\} \text{ for all }u\in P_0.
	\end{equation}
	\item\label{SQ2_item4} 
	For every $u\in P_0$ one has 
	$\Ker T_u(\beta_{\gamma_0})=\{\dot{u}\in T_u P_0: u^*\dot{u}\rho_0=\rho_0u^*\dot{u}\}$. 
\end{enumerate}
\end{lem}

\begin{proof}
	\eqref{SQ1}
For every $x\in\M$,  
\begin{align*}
\langle E(\beta(u)\gamma_0),x\rangle
&=\langle uj(u)\gamma_0\mid xuj(u)\gamma_0\rangle 
 =\langle j(u)^*uj(u)\gamma_0\mid xu\gamma_0\rangle \\
&=\langle uj(u)^*j(u)\gamma_0\mid xu\gamma_0\rangle 
=\langle uj(u^*u)\gamma_0\mid xu\gamma_0\rangle \\
&=\langle uJu^*uJ\gamma_0\mid xu\gamma_0\rangle 
=\langle u\gamma_0\mid xu\gamma_0\rangle \\
&=\langle E(u\gamma_0),x\rangle
\end{align*}
where we used $J\gamma_0=\gamma_0$. 
We thus obtain $ E(\beta(u)\gamma_0)=E(u\gamma_0)$ and then, since $\beta(u)\gamma_0\in\P$, it follows that $\vert u\gamma_0\vert'=\beta(u)\gamma_0\in\P$. 

\eqref{SQ2_item1} 
Since $J\gamma_0=\gamma_0$, one has 
$$j(a)\gamma_0=a^*\gamma_0
\iff Ja\gamma_0=a^*\gamma_0\iff a\gamma_0=Ja^*\gamma_0\iff a\gamma_0=j(a^*)\gamma_0.$$
For every $x\in\M$ one then obtains by the hypothesis on $a$, 
\begin{align*}
\langle \rho_0a,x\rangle
&=\langle \rho_0,ax\rangle 
=\langle\gamma_0\mid ax\gamma_0\rangle 
=\langle a^*\gamma_0\mid x\gamma_0\rangle 
=\langle j(a)\gamma_0\mid x\gamma_0\rangle \\
&=\langle \gamma_0\mid xj(a^*)\gamma_0\rangle 
=\langle \gamma_0\mid xa\gamma_0\rangle 
=\langle a\rho_0,x\rangle
\end{align*}
where we also used $j(a)^*=j(a^*)$ and $j(a)^*\in\M'$.

\eqref{SQ2_item2} 
By $\U(\M)$-equivariance of $E\colon\P\to\M_*^+$ and the hypothesis $\gamma_0\in\P$ one has for any $u\in P_0$, 
\begin{align*}
u\in U_{\rho_0} 
&\iff u\rho_0 u^*=\rho_0 
\iff E(u\gamma_0)=E(\gamma_0)  
 \iff \vert u\gamma_0\vert '=\gamma_0 \\
&\iff \beta(u)\gamma_0=\gamma_0
\end{align*}
where the last equivalence follows by \eqref{SQ1}. 
Moreover, using the relations $u^*u\gamma_0=\gamma_0$ and $j(u)\in\M'$, 
it is easily checked that for every $u\in P_0$ one has $\beta(u)\gamma_0=\gamma_0$ if and only if $j(u)\gamma_0=u^*\gamma_0$. 

\eqref{SQ2_item3} The inclusion ``$\supseteq$'' in \eqref{b} follows by the above equality $U_{\rho_0}=\{u\in P_0: \beta(u)\gamma_0=\gamma_0\}$ along with the property $\beta(xy)=\beta(x)\beta(y)$ which holds for all $x,y\in\M$. 

For the opposite inclusion ``$\subseteq$'', if $u_1,u_2\in P_0$ satisfy  $\beta_{\gamma_0}(u_1)=\beta_{\gamma_0}(u_2)$, then $u_1j(u_1)\gamma_0=u_2j(u_2)\gamma_0$, which easily implies $u_2^*u_1\gamma_0=j(u_1^*u_2)\gamma_0$. 
Denoting $g:=u_2^*u_1$, we then obtain $g\in U_\rho$ by \eqref{SQ2_item2}, 
and on the other hand $u_2g=u_1$. 

\eqref{SQ2_item4} 
Since $\beta_{\gamma_0}(u)=uj(u)\gamma_0$, one has 
$$T_u(\beta_{\gamma_0})\colon T_uP_0\to\H,\quad 
T_u(\beta_{\gamma_0})\dot{u}=\dot{u}j(u)\gamma_0+u j(\dot{u})\gamma_0,$$
where we recall that 
$$T_uP_0=\{\dot{u}\in\M p_0\mid u^*\dot{u}\in ip_0\M^h p_0\},$$
see \eqref{TuP0}.
Since $j(\M)=\M'$, we obtain for any $\dot{u}\in T_uP_0$
\begin{align*}
\dot{u}\in \Ker T_u(\beta_{\gamma_0}) 
& \iff 
j(u)\dot{u}\gamma_0=-u j(\dot{u})\gamma_0 
 \iff 
-u^*\dot{u}\gamma_0=j(u^*\dot{u})\gamma_0 \\
& \iff 
u^*\dot{u}\rho_0=\rho_0u^*\dot{u}
\end{align*}
where the last equivalence follows by \eqref{SQ2_item1} since 
$u^*\dot{u}\in i\M^h$. 
\end{proof}

In order to prove the next theorem 
we first  
prove that the ``square-root homeomorphism'' $(E\vert_{\P})^{-1}\colon\M_*^+\to\P$ is a weak immersion (in particular, is smooth) along the coadjoint groupoid orbits:

\begin{lem}\label{SQ3}
	If $(\M,\H,J,\P)$ is a standard form, $\gamma_0\in\P$, $\rho_0:=E(\gamma_0)\in\M_*^+$,  and 
	we define  $\bepsilon:=(E\vert_{\P})^{-1}\colon\M_*^+\to\P$, 
	then the injective mapping 
	$\bepsilon\vert_{\Oc_{\rho_0}}\colon\Oc_{\rho_0}\to\H$
	is a weak immersion. 
\end{lem}

\begin{proof}
	For $p_0:=\sigma_*(\rho_0)$, $P_0:=\br^{-1}(p_0)\subseteq\U(\M)$, 
	$\pi_0\colon P_0\to\Oc_{\rho_0}$, $\pi_0(u)=u\rho_0u^*$, 
	and $\beta_{\gamma_0}\colon P_0\to\H$, $\beta_{\gamma_0}(u):=\beta(u)\gamma_0$, 
	we note that the diagram 
	\begin{equation}
	\label{SQ3_proof_eq1}
	\xymatrix{P_0 \ar[d]_{\pi_0} \ar[dr]^{\beta_{\gamma_0}}& \\
		\Oc_{\rho_0} \ar[r]_{\bepsilon\vert_{\Oc_{\rho_0}}} & \H
	}
	\end{equation}
	is commutative since, 
	by Lemma~\ref{lem:56}\eqref{SQ1} and $\U(\M)$-equivariance of $E$, 
	$$E(\beta_{\gamma_0}(u))=E(\vert u\gamma_0\vert')
	=E(u\gamma_0)=uE(\gamma_0)u^*=u\rho_0u^*=\pi_0(u).$$
	Then, in the commutative diagram~\eqref{SQ3_proof_eq1}, 
	the mapping $\pi_0$ is a submersion while $\beta_{\gamma_0}$ is clearly smooth, hence $\bepsilon\vert_{\Oc_{\rho_0}}\colon\Oc_{\rho_0}\to\H$
	is smooth. 
	
	To prove that $\bepsilon\vert_{\Oc_{\rho_0}}$ is a weak immersion, it follows by the above commutative diagram that it suffices to prove that $\Ker T_u(\pi_0)=\Ker T_u(\beta_{\gamma_0})$ for arbitrary $u\in P_0$. 
	But this follows by Lemma~\ref{lem:56}\eqref{SQ2_item4} along with the proof of Theorem~\ref{thm:32}\eqref{thm:32_item3}. 
\end{proof}

The following theorem relates the \Banach Lie groupoid  $\U(\M)*\M_*^+\tto \M_*^+$ from Theorem~\ref{coadj_grpd} to a standard form of the $W^*$-algebra~$\M$. 

\begin{thm}\label{grpd_act}
	Let $(\M,\H,J,\P)$ be a standard form.
		Then the groupoid isomorphism 
			\begin{equation}\label{grpd_act_item1}
			\xymatrix{
		\U(\M)\ast\M_*^+ \ar@<-.5ex>[d] \ar@<.5ex>[d] \ar[r]^{\Phi} 
		& \H \ar@<-.5ex>[d] \ar@<.5ex>[d]\\
		\M_*^+ \ar[r]^{\id} & \M_*^+,}
	\end{equation}
where 
\begin{equation}
\Phi(v,\varphi):=v\bepsilon(\varphi)
\end{equation}
is equal to the composition of the groupoid isomorphisms from \eqref{531}, defines a  bijective weak immersion $\Phi\colon \U(\M)\ast\M_*^+\to\H$. 
	\end{thm}

\begin{proof} 
Taking into account the structural maps of the groupoid $\H\tto\M_*^+$ given in Theorem~\ref{thm:53} and Remark \ref{rem:54}, it is straightforward to check that 
	the mapping $\Phi$ from the statement is a groupoid isomorphism, with its inverse 
	$$\Phi^{-1}\colon\H\to \U(\M)\ast\M_*^+,\quad 
	\gamma\mapsto 
	(v_\gamma, E(\vert\gamma\vert)) $$
	where we recall the bijective map $E\vert_{\P}\colon\P\to\M_*^+$ given by 
	Proposition~\ref{basic_homeo}, whose inverse
	gives the object inclusion map $\bepsilon\colon\M_*^+\to\P\hookrightarrow\H$ of the  groupoid $\H\tto\P$,  
	and the polar decomposition of an arbitrary vector $\gamma\in\H$ is written as $\gamma=v_\gamma\vert\gamma\vert$. 
	Here we note that for arbitrary $\gamma\in\H$ one has 
	$(v_\gamma, E(\vert\gamma\vert))\in \U(\M)\ast\M_*^+$ since 
	$$\br(v_\gamma)=v_\gamma^* v_\gamma=p_{\vert\gamma\vert}
	=\sigma_*(E(\vert\gamma\vert))$$
	by  \eqref{suppvect}. 

	We recall that the Lie groupoid $\U(\M)\ast\M_*^+\tto\M_*^+$ is the disjoint union of its transitive Lie subgroupoids $\U_{p_0}(\M)\ast\Oc_{\rho_0}\tto\Oc_{\rho_0}$ 
parameterized by the coadjoint groupoid orbits $\Oc_{\rho_0}$. 
It then follows by Lemma~\ref{SQ3} that $\Phi$ is smooth, taking into account the smooth structure of $\U_{p_0}(\M)\ast\Oc_{\rho_0}$. 

To prove that $\Phi$ is a weak immersion, let $\gamma_0\in\P$ arbitrary and denote $\rho_0:=E(\gamma_0)\in\M_*^+$, $p_0:=\sigma_*(\rho_0)\in\Ll(\M)$, and $P_0:=\br^{-1}(p_0)\subseteq\U(\M)$, as usual. 
One then has the commutative diagram 
\begin{equation}
\label{SQ4_proof_eq1}
\xymatrix{P_0 \times P_0\ar[d]_{\Psi} \ar[dr]^{\Psi_{\gamma_0}}& \\
	\U_{p_0}(\M)\ast\Oc_{\rho_0} \ar[r]^{\ \ \Phi} & \H
}
\end{equation}
where the bottom arrow is 
$\Phi\vert_{\U_{p_0}(\M)\ast\Oc_{\rho_0}}\colon \U_{p_0}(\M)\ast\Oc_{\rho_0} \to \H$, 
and moreover 
\begin{align}
\Psi & \colon P_0\times P_0\to\U_{p_0}(\M)\ast\Oc_{\rho_0}, 
\quad \Psi(u,v)=(uv^*,v\rho_0 v^*), \nonumber\\
\label{SQ4_proof_eq1.5}
\Psi_{\gamma_0} & \colon P_0\times P_0\to\H, 
\quad \Psi_{\gamma_0}(u,v)=u j(v)\gamma_0
\end{align} 
for $(u,v)\in P_0\times P_0$.
The diagram \eqref{SQ4_proof_eq1} is commutative since 
\begin{align*}
\Phi(\Psi(u,v))
&=\Phi(uv^*,v\rho_0 v^*) 
=uv^*\bepsilon(v\rho_0 v^*) 
=uv^*\bepsilon(E(v\gamma_0)) 
=uv^*\bepsilon(E(\vert v\gamma_0\vert')) \\
&=uv^*\vert v\gamma_0\vert' 
=uv^*\beta(v)\gamma_0 
=u j(v)\gamma_0 
=\Psi_{\gamma_0}(u,v)
\end{align*}
where we used the equalities $\vert v\gamma_0\vert'=\beta(v)\gamma_0=vj(v)\gamma_0$ given by \eqref{g} and \eqref{ju2}. 

It follows by the commutative diagram \eqref{SQ4_proof_eq1} that, in order to prove that its bottom arrow 
$\Phi\vert_{\U(\M)\ast\Oc_{\rho_0}}\colon \U(\M)\ast\Oc_{\rho_0} \to \H$ is a weak immersion, it suffices to prove the equality
\begin{equation}\label{SQ4_proof_eq2}
\Ker (T_{(u,v)}\Psi)=\Ker (T_{(u,v)}(\Psi_{\gamma_0}))\subseteq T_{(u,v)}(P_0\times P_0)
\end{equation}
for arbitrary $(u,v)\in P_0\times P_0$. 
Recalling the diagonal action of $U_{\rho_0}$ from the right on $P_0\times P_0$ and the isomorphism of the gauge groupoid $\frac{P_0\times P_0}{U_{\rho_0}}$ onto $\U(\M)\ast\Oc_{\rho_0}$, it is easily seen that 
\begin{equation}\label{SQ4_proof_eq3}
\Ker (T_{(u,v)}\Psi)=\{(ux,vx) \in T_u P_0\times T_v P_0: x\in T_{p_0}(U_{\rho_0})\}
\end{equation}
where the Lie algebra $T_{p_0}(U_{\rho_0})$ of the  Lie group $U_{\rho_0}$ is given by 
\begin{equation}\label{SQ4_proof_eq4}
T_{p_0}(U_{\rho_0})=\{x\in i p_0\M^h p_0: x\gamma_0=-j(x)\gamma_0\}
\end{equation}
by Lemma~\ref{lem:56}(\eqref{SQ2_item1}--\eqref{SQ2_item2}). 
On the other hand, it directly follows by Lemma~\ref{lem:56}\eqref{SQ2_item2} that the mapping~$\Psi_{\gamma_0}$ is constant on the orbits of the aforementioned diagonal action  of $U_{\rho_0}$ on $P_0\times P_0$, 
and this implies $\Ker (T_{(u,v)}\Psi)\subseteq \Ker (T_{(u,v)}(\Psi_{\gamma_0}))$.

To prove the remaining inclusion $\Ker (T_{(u,v)}\Psi)\supseteq \Ker (T_{(u,v)}(\Psi_{\gamma_0}))$ for \eqref{SQ4_proof_eq2}, 
we note that, by the definition of $\Psi_{\gamma_0}$, one has for arbitrary $u,v\in P_0$, 
\begin{equation}\label{SQ4_proof_eq5}
T_{(u,v)}(\Psi_{\gamma_0})\colon T_u P_0\times T_v P_0\to\H, \quad 
(T_{(u,v)}(\Psi_{\gamma_0}))(\dot{u},\dot{v})
=\dot{u}j(v)\gamma_0+uj(\dot{v})\gamma_0
\end{equation}
where $\dot{u},\dot{v}\in\M p_0$, $j(v),j(\dot{v})\in\M'$, and $u^*u\gamma_0=j(v^*v)\gamma_0=\gamma_0$. 
Therefore
\begin{align*}
(T_{(u,v)}(\Psi_{\gamma_0}))(\dot{u},\dot{v})=0
& \implies u^*j(v^*)\dot{u}j(v)\gamma_0+u^*j(v^*)uj(\dot{v})\gamma_0=0 \\
& \iff u^*\dot{u}\gamma_0+j(v^*\dot{v})\gamma_0=0.
\end{align*}
Denoting $x:=u^*\dot{u}$ and $y:=v^*\dot{v}$, 
one has $x,y\in i p_0\M^h p_0$ since $\dot{u}\in T_u P_0$ and $\dot{v}\in T_v P_0$. 
On the other hand, the above equality $x\gamma_0+j(y)\gamma_0=0$ is equivalent to $x\gamma_0+Jy\gamma_0=0$, which further implies $j(x)\gamma_0+y\gamma_0=0$, and substracting this from $x\gamma_0+j(y)\gamma_0=0$
we obtain 
$$a\gamma_0=j(a)\gamma_0.$$
were $a:=x-y\in i\M^h$. 
Then, using the fact that $0\le \langle\gamma_1\mid\gamma_2\rangle$ for all $\gamma_1,\gamma_2\in\P$ 
and on the other hand $\gamma_0,aj(a)\gamma_0\in\P$,   
one has 
\begin{equation*}
0\le  \langle \gamma_0\mid aj(a)\gamma_0\rangle 
=\langle \gamma_0\mid a^2\gamma_0\rangle 
=\langle \rho_0,a^2\rangle\le0 
\end{equation*}
where the last inequality holds true since $a^2\le 0$ as $a\in i \M^h$. 
Consequently $\langle \rho_0,a^2\rangle=0 $ and then, using 
$a\in i p_0\M^h p_0$ and $p_0=\sigma_*(\rho_0)$, we obtain  $a^2=0$, hence $a=0$, that is, $x=y$. 
Recalling from the above that $x\gamma_0+j(y)\gamma_0=0$, 
and taking into account \eqref{SQ4_proof_eq4}, we finally obtain 
$$\Ker (T_{(u,v)}(\Psi_{\gamma_0}))\subseteq 
\{(ux,vx) \in T_u P_0\times T_v P_0: x\in T_{p_0}(U_{\rho_0})\}.$$
That is, by \eqref{SQ4_proof_eq3}, one has 
$\Ker (T_{(u,v)}\Psi)\supseteq \Ker (T_{(u,v)}(\Psi_{\gamma_0}))$. 
This completes the proof of \eqref{SQ4_proof_eq2}, and we are done. 
\end{proof}

Theorem~\ref{grpd_act} 
shows that the original manifold structure of the Hilbert space $\H$ can be refined to a manifold structure to be denoted by 
$$\widetilde{\H},$$ 
transported from the Lie groupoid $\U(\M)\ast\M_*^+\tto\M_*^+$ via the bijective mapping $\Phi\colon \U(\M)\ast\M_*^+\to\H$. 
{\bf The standard groupoid $\H\tto\M_*^+$ is thus endowed with a unique Lie groupoid structure, to be denoted by} 
$$\widetilde{\H}\tto\M_*^+,$$  for which the diagram \eqref{grpd_act_item1} is an isomorphism of Lie groupoids if $\H$ is replaced by~$\widetilde{\H}$, 
where $\U(\M)\ast \M_*^+\tto \M_*^+$ is a Lie groupoid by Theorem~\ref{coadj_grpd}. 
The manifold $\widetilde{\H}$ should be regarded as a singular foliation of the Hilbert space $\H$, whose leaves are diffeomorphic to the transitive subgroupoids of the Lie groupoid $\U(\M)\ast \M_*^+\tto \M_*^+$. 
In particular, the topology and the manifold structure of $\widetilde{\H}$ are richer than the original topology and manifold structure of $\H$, and the identity mapping of $\H$ gives a bijective immersion $\widetilde{\Phi}\colon\widetilde{\H}\to\H$.

\begin{cor}\label{grpd_proj}
	One has a Lie groupoid morphism
	$$\xymatrix{
		\widetilde{\H} \ar@<-.5ex>[d] \ar@<.5ex>[d] \ar[r]^{\widetilde{\pr}_1} 
		& \U(\M) \ar@<-.5ex>[d] \ar@<.5ex>[d]\\
		\M_*^+ \ar[r]^{\sigma_*} & \Ll(\M)}$$
	where $\widetilde{\pr}_1(\gamma)=v_\gamma$ if $\gamma=v_\gamma\vert\gamma\vert$ is the polar decomposition of any $\gamma\in\H$, see~\eqref{g}.
\end{cor}

\begin{proof}
	The assertion follows by Proposition~\ref{grpd_act} 
	along with the fact that one has a Lie-groupoid morphism 
	$$\xymatrix{
		\U(\M)\ast\M_*^+\ar@<-.5ex>[d] \ar@<.5ex>[d] \ar[r]^{\pr_1} 
		& \U(\M) \ar@<-.5ex>[d] \ar@<.5ex>[d]\\
		\M_*^+ \ar[r]^{\sigma_*} & \Ll(\M)}$$
	defined by the Cartesian projection 
	$\pr_1\colon \U(\M)\ast\M_*^+\to \U(\M)$, $(v,\varphi)\mapsto v$, 
	as noted in \cite[Appendix]{OS}.
\end{proof}

 \subsection{Presymplectic structure of the standard groupoid% $\H\tto\M_*^+$
}
\label{Subsect6.3}
We recall the refined manifold structure $\widetilde{\H}$, the bijective immersion $\widetilde{\Phi}\colon\widetilde{\H}\to\H$, and the strongly symplectic form $\omega\in\Omega^2(\H,\R)$ of the Hilbert space~$\H$. 
Then one has the Lie groupoid $\widetilde{\H}\tto\P$, 
whose underlying abstract groupoid is the groupoid $\H\tto\P$ from Theorem~\ref{thm:53}. 
One can then define the pullback differential form 
\begin{equation}\label{tildeomega}
\widetilde{\bomega}:=\widetilde{\Phi}^*\omega\in\Omega^2(\widetilde{\H},\R)
\end{equation}
which is a 2-form on the Lie groupoid $\widetilde{\H}\tto\M_*^+$. 
Since the symplectic form $\omega$ is closed, it follows that the 2-form $\widetilde{\bomega}$ is also closed, however it is degenerate, as we will see below, and therefore we call it a \emph{presymplectic form}. 

The main point of this subsection is to study the compatibility between $\widetilde{\bomega}$ and the Lie groupoid structure  $\widetilde{\H}\tto\P$:  
We prove that this presymplectic form is multiplicative in the usual sense of finite-dimensional Lie groupoid theory and moreover we give a rather precise description of the foliation defined by the kernel of this closed 2-form. 
It will thus turn out that the 
augmented standard groupoid $(\widetilde{\H},\widetilde{\bomega})\tto\M_*^+$ is an (infinite-dimensional) Lie groupoid that shares some of the features of 
finite-dimensional Lie groupoids with multiplicative symplectic structure as in 
\cite{BuCWZ04}. 
(See also \cite{DZ05}.)

We now prepare for the proof of Proposition~\ref{graph} below, which  shows the aforementioned multiplicativity property of the presymplectic form~$\widetilde{\bomega}$. 
We denote by $\Delta$ 
the graph of the groupoid multiplication map $\mu\colon\widetilde{\H}\ast\widetilde{\H}\to\widetilde{\H}$, 
that is, 
$$\Delta:=\{(\gamma_1,\gamma_2,\gamma_1\bullet\gamma_2)\in (\widetilde{\H}\ast\widetilde{\H})\times\widetilde{\H}: (\gamma_1,\gamma_2)\in \widetilde{\H}\ast\widetilde{\H}\}.$$ 
Let us now define the complex-valued differential 1-form 
\begin{equation}\label{(5)}
\widetilde{\Gamma}^{++-}:={\rm pr}_1^*\widetilde{\Gamma}+{\rm pr}_2^*\widetilde{\Gamma}-\bmu^*\widetilde{\Gamma} \in\Omega^1(\widetilde{\H}\ast\widetilde{\H},\C)
\end{equation}
where $\widetilde{\Gamma}:=\widetilde{\Phi}^*(\Gamma)\in\Omega^1(\widetilde{\H},\C)$ with 
$\Gamma\in\Omega^1(\H,\C)$ being the complex-valued 1-form defined 
in~\eqref{Gamma}, and one uses pull-backs of $\widetilde{\Gamma}$ with respect to the Cartesian projections ${\rm pr}_1,{\rm pr}_2\colon\widetilde{\H}\ast\widetilde{\H}\to\widetilde{\H}$ and with respect to the groupoid multiplication $\bmu\colon\widetilde{\H}\ast\widetilde{\H}\to\widetilde{\H}$. 

\begin{lem}\label{18May2019}
	The differential 1-form $\widetilde{\Gamma}^{++-}$  
	is an exact 1-form, more specifically
	\begin{equation}\label{(7)}
	\widetilde{\Gamma}^{++-}=\de(\frac{1}{2}\Vert \tilde\bs\circ{\rm pr}_2(\cdot)\Vert^2).
	\end{equation}
\end{lem}

\begin{proof}
In polar coordinates, an arbitrary point of $\Delta$ assumes the form 
\begin{equation}\label{(1)}
(\gamma_1,\gamma_2,\gamma_1\bullet\gamma_2)=(u_1\xi_1,u_2\xi_2,u_1u_2\xi_2)
\end{equation}
where $\gamma_1=u_1\xi_1$ and $\gamma_2=u_2\xi_2$ with $u_1,u_2\in\U(\M)$ and $\xi_1,\xi_2\in\P\subseteq\H$ satisfy the constraints 
\begin{align}
\label{(2)}
\xi_1& =u_2 j(u_2)\xi_2 \\
\label{(3)} 
u_1^*u_1& =u_2u_2^*,\ u_1^*u_1\xi_1=\xi_1, \ u_2^*u_2\xi_2=\xi_2. 
\end{align} 
Using \eqref{(2)} we obtain 
\begin{equation}\label{(4)}
(\gamma_1,\gamma_2,\gamma_1\bullet\gamma_2)=
(u_1u_2 j(u_2)\xi_2,u_2\xi_2,u_1u_2\xi_2)
\end{equation}
which allows us to parameterize the graph $\Delta$ by $(u_1,u_2,\xi_2)$ 
for $u_1,u_2\in\U(\M)$ and $\xi_2\in\P$ satisfying in particular~\eqref{(3)}. 

Expressing $\widetilde{\Gamma}^{++-}$ in the coordinates $(\gamma_1,\gamma_2)\in\H\ast\H$ we find 
\begin{equation}\label{(6)}
\widetilde{\Gamma}^{++-}(\gamma_1,\gamma_2)
=\langle\gamma_1\mid\de\gamma_1\rangle+\langle\gamma_2\mid\de\gamma_2\rangle-\langle\gamma_1\bullet\gamma_2\mid\de(\gamma_1\bullet\gamma_2)\rangle.
\end{equation}	
Using \eqref{(4)} and \eqref{(6)} we obtain 
	\allowdisplaybreaks
	\begin{align}
	\widetilde{\Gamma}^{++-}(u_1\xi_1,u_2\xi_2)
	=&\langle u_1\xi_1\mid\de(u_1\xi_1)\rangle+\langle u_2\xi_2\mid\de(u_2\xi_2)\rangle 
	-\langle u_1u_2\xi_2\mid\de(u_1u_2\xi_2)\rangle 
	\nonumber \\
	=&\langle u_1 u_2j(u_2)\xi_2\mid\de(u_1u_2j(u_2)\xi_2)\rangle+\langle u_2\xi_2\mid\de(u_2\xi_2)\rangle \nonumber \\
	&
	-\langle u_1u_2\xi_2\mid\de(u_1u_2\xi_2)\rangle 
	\nonumber \\
	=&\langle j(u_2)\xi_2\mid(u_1u_2)^*\de(u_1u_2)j(u_2)\xi_2\rangle \nonumber \\
	&
	+\langle (u_1u_2)^*u_1u_2j(u_2)\xi_2\mid\de(j(u_2)\xi_2) \rangle 
	\nonumber \\
	&
	+\langle \xi_2\mid (u_2^*\de u_2)\xi_2\rangle
	+\langle u_2^*u_2\xi_2\mid \de\xi_2\rangle  
	\nonumber \\
	& -\langle \xi_2\mid(u_1u_2)^*\de(u_1u_2)\xi_2 \rangle
	-\langle (u_1u_2)^*(u_1u_2)\xi_2\mid \de\xi_2\rangle
	\nonumber \\
	=&\langle j(u_2^*u_2)\xi_2\mid (u_1u_2)^*\de(u_1u_2)\xi_2\rangle
	+ \langle j(u_2)\xi_2\mid \de(j(u_2)\xi_2)\rangle
	\nonumber \\
	&+ \langle \xi_2\mid u_2^*\de u_2\xi_2\rangle
	+\langle \xi_2\mid \de\xi_2\rangle 
	\nonumber \\ 
	&-\langle \xi_2\mid (u_1u_2)^*\de(u_1u_2)\xi_2\rangle
	-\langle \xi_2\mid \de\xi_2\rangle 
	\nonumber \\
	=& \langle Ju_2^*u_2\xi_2\mid (u_1u_2)^*\de(u_1u_2)\xi_2\rangle
	+\langle Ju_2\xi_2\mid J\de(u_2\xi_2)\rangle \nonumber \\
	&
	+\langle \xi_2\mid u_2^*\de u_2\xi_2\rangle 
	\nonumber \\ 
	&
	-\langle \xi_2\mid (u_1u_2)^*\de(u_1u_2)\xi_2\rangle 
	\nonumber \\
	=&\overline{\langle\xi_2\mid(u_2^*\de u_2)\xi_2\rangle}
	+\overline{\langle\xi_2\mid\de\xi_2\rangle} 
	+\langle \xi_2\mid u_2^*\de u_2\xi_2\rangle 
	\nonumber \\
	=&
	\label{(8)}
	\overline{\langle\xi_2\mid\de\xi_2\rangle}. 
	\end{align}
	To obtain the above equalities we used the conditions~\eqref{(3)} and $J\xi_1=\xi_1$, $J\xi_2=\xi_2$. 
	We used also the commutation relation 
	$$(u_1u_2)^*\de(u_1u_2)j(u_2)=j(u_2)(u_1u_2)^*\de(u_1u_2).$$
	For $\xi_1,\xi_2\in{\mathcal L}:=\{\gamma\in\H : J\gamma=\gamma\}$ one has 
	\begin{align}
	\de\langle\xi_2\mid\xi_2\rangle
	&=\langle\de\xi_2\mid\xi_2\rangle
	+\langle\xi_2\mid\de\xi_2\rangle 
	=\langle\de(J\xi_2)\mid\xi_2\rangle
	+\langle\xi_2\mid\de\xi_2\rangle
	\nonumber \\
	&\label{(9)}
	=\overline{\langle\de\xi_2\mid J\xi_2\rangle}
	+\langle\xi_2\mid\de\xi_2\rangle 
	%\nonumber \\
	%\label{(9)}
	%&
	=2\langle\xi_2\mid\de\xi_2\rangle. 
	%\nonumber
	\end{align}
	Now \eqref{(7)} follows by \eqref{(8)} and \eqref{(9)}. 
\end{proof}

For the following proposition we define 
$$\widetilde{\bomega}^{++-}:={\rm pr}_1^*\widetilde{\bomega}+{\rm pr}_2^*\widetilde{\bomega}-{\rm pr}_3^*\widetilde{\bomega}
\in\Omega^2(\widetilde{\H}\times\widetilde{\H}\times\widetilde{\H}).$$
%{\rm pr}_1^*\de\Gamma+{\rm pr}_2^*\de\Gamma-{\rm pr}_3^*\de\Gamma$$
where ${\rm pr}_j\colon \widetilde{\H}\times\widetilde{\H}\times\widetilde{\H}\to\widetilde{\H}$ for $j=1,2,3$ are the natural Cartesian projections. 
%For these reasons we obtain the following result. 

\begin{prop}\label{graph}
The presymplectic form $\widetilde{\bomega}\in\Omega^2(\widetilde{\H})$ from \eqref{tildeomega} is multiplicative on the Lie groupoid $\widetilde{\H}\tto\P$, in the sense that 
\begin{equation}\label{graph_eq1}
{\rm pr}_1^*\widetilde{\bomega}+{\rm pr}_2^*\widetilde{\bomega}=\bmu^*\widetilde{\bomega} \in\Omega^2(\widetilde{\H}\ast\widetilde{\H},\R).
\end{equation}
Moreover, the graph of the groupoid multiplication $\Delta$ is an isotropic submanifold of the presymplectic manifold $(\widetilde{\H}\times\widetilde{\H}\times\widetilde{\H},\widetilde{\bomega}^{++-})$. 
\end{prop}

\begin{proof}
It folows by Lemma~\ref{18May2019} that $\de\widetilde{\Gamma}^{++-}=0$ 
hence, by \eqref{(5)},  one has 
${\rm pr}_1^*\de\widetilde{\Gamma}+{\rm pr}_2^*\de\widetilde{\Gamma}=\bmu^*\de\widetilde{\Gamma}$. 
Then, since $\de\widetilde{\Gamma}=\widetilde{\bomega}$, 
one obtains~\eqref{graph_eq1}. 
Finally, since $\widetilde{\H}\tto\P$ is a Lie groupoid, 
it is straightforward to prove that $\Delta$ is a submanifold of $\widetilde{\H}\times\widetilde{\H}\times\widetilde{\H}$, 
and then \eqref{graph_eq1} implies that $\Delta$ is moreover an isotropic submanifold of the presymplectic manifold $(\widetilde{\H}\times\widetilde{\H}\times\widetilde{\H},\widetilde{\bomega}^{++-})$. 	
\end{proof}

We now investigate the foliation corresponding to the degeneracy kernel of the multiplicative  presymplectic form $\widetilde{\bomega}\in\Omega^2(\widetilde{\H})$ on the Lie groupoid $\widetilde{\H}\tto\P$. 

\begin{prop}\label{SQ5}
	Let $\gamma_0\in\P$ be arbitrary, 
	denote $\rho_0:=E(\gamma_0)\in\M_*^+$, $p_0:=\sigma_*(\rho_0)\in\Ll(\M)$, 
and $P_0:=\br^{-1}(p_0)\subseteq\U(\M)$,  
	and use $\Psi_{\gamma_0}\colon P_0\times P_0\to\H$ from \eqref{SQ4_proof_eq1.5} to define the skew-symmetric bilinear form 
	$$%\varpi
(\Psi_{\gamma_0}^*\omega)(u,v)\colon (T_uP_0\times T_v P_0)\times(T_uP_0\times T_v P_0)\to\R$$ 
by 
	$$((\Psi_{\gamma_0}^*\omega)(u,v))((\dot{u}_1,\dot{v}_1),(\dot{u}_2,\dot{v}_2))
	:=\Im\langle (T_{(u,v)}(\Psi_{\gamma_0}))(\dot{u}_1,\dot{v}_1)\mid 
	(T_{(u,v)}(\Psi_{\gamma_0}))(\dot{u}_2,\dot{v}_2)\rangle.$$
	Then one has 
	\begin{equation}\label{SQ5_eq1}
	((\Psi_{\gamma_0}^*\omega)(u,v))((\dot{u}_1,\dot{v}_1),(\dot{u}_2,\dot{v}_2))
	=\Im\langle \dot{u}_1\gamma_0\mid 
	\dot{u}_2\gamma_0\rangle
	-\Im\langle \dot{v}_1\gamma_0\mid 
	\dot{v}_2\gamma_0\rangle
	\end{equation}
	and 
	\begin{equation}\label{SQ5_eq2}
	(T_uP_0\times T_v P_0)^{\perp_{\Psi_{\gamma_0}^*\omega}}
	=\{(ux,vy)\in T_uP_0\times T_v P_0: x,y\in T_{p_0}(U_{\rho_0})\}.
	\end{equation}
\end{prop}

\begin{proof}
	One has by \eqref{SQ4_proof_eq5}, 
	\begin{equation}\label{SQ5_proof_eq1}
	((\Psi_{\gamma_0}^*\omega)(u,v))((\dot{u}_1,\dot{v}_1),(\dot{u}_2,\dot{v}_2))
	=\Im\langle \dot{u}_1j(v)\gamma_0+uj(\dot{v}_1)\gamma_0\mid 
	\dot{u}_2j(v)\gamma_0+uj(\dot{v}_2)\gamma_0\rangle.
	\end{equation}
	Moreover 
	\begin{align}
	\langle \dot{u}_1j(v)\gamma_0+uj(\dot{v}_1)\gamma_0\mid 
	& \dot{u}_2j(v)\gamma_0+uj(\dot{v}_2)\gamma_0\rangle 
	\nonumber\\
	=& 
	\langle \dot{u}_1j(v)\gamma_0\mid 
	\dot{u}_2j(v)\gamma_0\rangle
	+\langle uj(\dot{v}_1)\gamma_0\mid 
	uj(\dot{v}_2)\gamma_0\rangle
	\nonumber \\
	& +\langle uj(\dot{v}_1)\gamma_0\mid 
	\dot{u}_2j(v)\gamma_0\rangle
	+ \langle \dot{u}_1j(v)\gamma_0\mid 
	uj(\dot{v}_2)\gamma_0\rangle\nonumber \\
	= 
	&\langle \dot{u}_1\gamma_0\mid 
	\dot{u}_2\gamma_0\rangle
	+\langle \dot{v}_2\gamma_0\mid 
	\dot{v}_1\gamma_0\rangle \nonumber \\
	& 
	\label{SQ5_proof_eq2}
	+\langle j(v^*\dot{v}_1)\gamma_0\mid 
	u^*\dot{u}_2\gamma_0\rangle
	+\langle u^*\dot{u}_1\gamma_0\mid 
	j(v^*\dot{v}_2)\gamma_0\rangle
	\end{align}
	Here we used the equalities 
	\begin{align*}
	\langle uj(\dot{v}_1)\gamma_0\mid 
	uj(\dot{v}_2)\gamma_0\rangle
	& =\langle j(\dot{v}_1)u^*u\gamma_0\mid 
	j(\dot{v}_2)\gamma_0\rangle 
	=\langle j(\dot{v}_1)\gamma_0\mid 
	j(\dot{v}_2)\gamma_0\rangle \\
	&=\langle J\dot{v}_1J\gamma_0\mid 
	J\dot{v}_2J\gamma_0\rangle 
	=\langle \dot{v}_2\gamma_0\mid 
	\dot{v}_1\gamma_0\rangle.
	\end{align*}
	On the other hand, denoting $x_1:=u^*\dot{u}_1$ and $y_2:=v^*\dot{v}_2$, 
	one has $x_1^*=-x_1$, $y_2^*=-y_2$, $j(y_2)^*=-j(y_2)$, 
	and $x_1j(y_2)=j(y_2)x_1$, hence $(x_1j(y_2))^*=x_1j(y_2)$, 
	and then 
	$$\langle u^*\dot{u}_1\gamma_0\mid 
	j(v^*\dot{v}_2)\gamma_0\rangle
	=\langle x_1\gamma_0\mid 
	j(y_2)\gamma_0\rangle=-\langle \gamma_0\mid 
	x_1j(y_2)\gamma_0\rangle\in\R.$$
	Similarly $\langle j(v^*\dot{v}_1)\gamma_0\mid 
	u^*\dot{u}_2\gamma_0\rangle\in\R$, 
	and we then obtain by \eqref{SQ5_proof_eq1} and \eqref{SQ5_proof_eq2}, 
	\begin{equation*}
	%\label{SQ5_proof_eq3}
	((\Psi_{\gamma_0}^*\omega)(u,v))((\dot{u}_1,\dot{v}_1),(\dot{u}_2,\dot{v}_2))
	=\Im\langle \dot{u}_1\gamma_0\mid 
	\dot{u}_2\gamma_0\rangle
	+\Im\langle \dot{v}_2\gamma_0\mid 
	\dot{v}_1\gamma_0\rangle
	\end{equation*}
	which is equivalent to the equality \eqref{SQ5_eq1} from the statement.
	
	Furthermore, it follows by \eqref{SQ5_eq1} and 
	Theorem~\ref{thm:orbit} 
	on obtaining the symplectic form of $\Oc_{\rho_0}$ by reduction of the symplectic form of $\H$ that \eqref{SQ5_eq2} holds true as well. 
\end{proof}

\begin{rem}
	Proposition~\ref{SQ5} shows that, in the diagram \eqref{SQ4_proof_eq1}, 
	if we denote by $\omega$ the canonical symplectic structure of the Hilbert space $\H$, then $\Psi_{\gamma_0}^*(\omega)$ is a closed 2-form on $P_0\times P_0$ whose degeneracy $\Ker \Psi_{\gamma_0}^*(\omega) \subseteq T(P_0\times P_0)$ satisfies 
	$$\underbrace{\{(ux,vx):u,v\in P_0, x\in T_{p_0}(U_{\rho_0})\}}_{\Ker(T\Psi)}
	\subsetneqq 
	\underbrace{\{(ux,vy):u,v\in P_0, x,y\in T_{p_0}(U_{\rho_0})\}}_{\Ker \Psi_{\gamma_0}^*(\omega)}$$
	and then, reduction of the 2-form $\Psi_{\gamma_0}^*(\omega)$ 
	via the mapping $\Psi\colon P_0\times P_0\to \U(\M)\ast\Oc_{\rho_0}$ exists but it is always a degenerate 2-form on 
	$\U(\M)\ast\Oc_{\rho_0}$ 
	(and similarly for the gauge groupoid $\frac{P_0\times P_0}{U_{\rho_0}}$). 
	See also \cite{Ni97} and \cite{CGM17} for other instances of symplectic reduction in infinite dimensions. 
\end{rem}

We are now in a position to describe the foliation determined by the degeneracy kernel of $\widetilde{\bomega}$, which we call for short the \emph{degeneracy-foliation of  $\widetilde{\bomega}$}. 
It is important to point out that formula~\eqref{SQ6_eq1} below can be regarded as an infinite-dimensional version of the description of multiplicative 2-forms on finite-dimensional Lie groupoids given in \cite[Lemma 3.1(iv)]{BuCWZ04}.  

\begin{prop}\label{SQ6}
In the setting of Proposition~\ref{SQ5}, denote $$\Oc_{\gamma_0}:=\{uj(u)\gamma_0: u\in P_0\}\subseteq\P$$ 
and 
$$\widetilde{\H}_{\gamma_0}:=\widetilde{\bs}^{-1}(\Oc_{\gamma_0})\subseteq\widetilde{\H}.$$ 
Then $\widetilde{\H}_{\gamma_0}\tto\Oc_{\gamma_0}$ is a transitive subgroupoid of the Lie groupoid $\widetilde{\H}\tto\P$ 
and the following assertions hold. 
\begin{enumerate}[{\rm(i)}] 
 \item\label{SQ6_item1} 
 The mapping $(\widetilde{\bt},\widetilde{\bs})\colon \widetilde{\H}_{\gamma_0}\to \Oc_{\gamma_0}\times\Oc_{\gamma_0}$ is a surjective submersion, 
and its fibers are the leaves of the degeneracy-foliation of  $\widetilde{\bomega}$
\item\label{SQ6_item2} 
The orbit $\Oc_{\gamma_0}$ has a unique weakly symplectic structure $\widetilde{\bomega}_{{\gamma_0}}\in\Omega^2(\Oc_{\gamma_0},\R)$ 
that satisfies 
\begin{equation}\label{SQ6_eq1}
\widetilde{\bomega}\vert_{\widetilde{\H}_{\gamma_0}}
=(\widetilde{\bt},\widetilde{\bs})^*(\widetilde{\omega}_{{\gamma_0}}\oplus(-\widetilde{\omega}_{{\gamma_0}}))
=\widetilde{\bt}^*\widetilde{\omega}_{{\gamma_0}}-\widetilde{\bs}^*\widetilde{\omega}_{{\gamma_0}}.
\end{equation}
Moreover, the mapping $\bepsilon\vert_{\Oc_{\rho_0}}\colon \Oc_{\rho_0}\to\Oc_{\gamma_0}$ 
is a symplectomorphism from the weakly symplectic manifold $(\Oc_{\rho_0},\widetilde{\omega}_{\rho_0})$ in Theorem~\ref{thm:orbit} 
onto $(\Oc_{\gamma_0},\widetilde{\omega}_{{\gamma_0}})$. 
\end{enumerate}
\end{prop}

\begin{proof}
It follows by \eqref{ju2} that $\Oc_{\gamma_0}$ is the orbit of the groupoid $\H\tto\P$ passing through $\gamma_0\in\P$, and then $\widetilde{\H}_{\gamma_0}\tto\Oc_{\gamma_0}$ is a transitive subgroupoid of the Lie groupoid $\widetilde{\H}\tto\P$. 
Moreover, we obtain 
$$\Psi_{\gamma_0}(P_0\times P_0)=\widetilde{\H}_{\gamma_0}$$ 
by Proposition~\ref{isomorphisms} and the commutative diagram \eqref{SQ4_proof_eq1}, 
since $\U_{p_0}(\M)\ast\Oc_{\rho_0}\tto\Oc_{\rho_0}$ is a transitive subgroupoid of $\U(\M)\ast\M_*^+\tto\M_*^+$.
Furthermore, for arbitrary $u_0,v_0\in P_0$, 
the leaf of the degeneracy-foliation of  $\widetilde{\bomega}$ that passes through 
the point $\Psi_{\gamma_0}(u_0,v_0)\in\widetilde{\H}_0$ is  $\Psi_{\gamma_0}(u_0U_{\rho_0}\times v_0U_{\rho_0})$. 
This follows by \eqref{SQ5_eq2} in Proposition~\ref{SQ5}. 

Using these remarks, the equality \eqref{SQ6_eq1} follows by 
\eqref{SQ5_eq1} in Proposition~\ref{SQ5}.  
To this end we also use the principal $U_{\rho_0}$-bundle $P_0\to P/U_{\rho_0}$  
and, on the other hand,  
the fact that, by Proposition~\ref{gauge_isom}, the mapping 
$$\frac{P_0\times P_0}{U_{\rho_0}}\to\widetilde{\H}_{\gamma_0},\quad   [(u,v)]\mapsto\Psi_{\gamma_0}(u,v)=uj(v)\gamma_0$$ 
is an isomorphism of Lie groupoids. 

To prove the uniqueness assertion in \eqref{SQ6_item2}, we note that 
 the mapping 
 $$(\widetilde{\bt},\widetilde{\bs})\colon \widetilde{\H}_{\gamma_0}\to\Oc_{\gamma_0}\times\Oc_{\gamma_0}$$ 
 is a surjective submersion, 
 hence there exists at most one differential 2-form on $\Oc_{\gamma_0}\times\Oc_{\gamma_0}$ whose pull-back via $(\widetilde{\bt},\widetilde{\bs})$ is equal to $\widetilde{\bomega}\vert_{\widetilde{\H}_{\gamma_0}}$. 
\end{proof}

For the next proposition of this subsection recall the following:
\begin{enumerate} [(i)]
\item the positive cone $\P$ inherits a smooth manifold structure from $\M_*^+$ 
via the homeomorphism $E|_{\P}:\P\to \M_*^+$.
\item the tangent spaces $T_\gamma\tilde \bs^{-1}(\tilde \bs(\gamma_0))=\Ker T_\gamma\tilde \bs$  
is equal $T_u\iota_{\gamma_0}(T_uP_0)$, where $\gamma=u\gamma_0$.
\end{enumerate}
The above implies that $\iota_{\gamma_0}:P_0\to \H$ is a quasi-immersed submanifold of $\H$, i.e. for every $u\in P_0$ the tangent map $T_u\iota_{\gamma_0}:T_uP_0\to T_{\iota_{\gamma_0}(u)} \H$ is injective with closed range. 

In the next propositions we mentioned other properties of the standard groupoid $\H\tto \M_*^+$ that are shared by symplectic groupoids in the finite dimensional case. 
Here we denote by $\H^\mathbb{R}$ the space $\H$ regarded as a real Hilbert space.

\begin{prop} \label{Prop}
\begin{enumerate} [{\rm(i)}]
\item\label{Prop_item1} 
The anti-unitary involution (the inversion map) $J:\H\to \H$ is an anti-symplectomorphism, i.e.
\begin{equation} 
\label{Jomega}
J^*\omega=-\omega.
\end{equation}
\item\label{Prop_item2} 
 One has the splitting 
\begin{equation*}
 \H^\mathbb{R}=\H_+\oplus \H_-
\end{equation*} 
of $\H^\mathbb{R}$ into the real Hilbert (Lagrangian) subspaces $\H_{\pm}=P_{\pm}\H$, where 
$$P_{\pm}:=\frac{1}{2}(\id+J)$$ 
satisfy $P_{\pm}^2=P_{\pm}$ and $P_+P_-=P_-P_+=0$.
\item\label{Prop_item3} 
 For the object inclusion maps  one has 
$\widetilde{\bepsilon}\colon\P \hookrightarrow\H_+$  and
$\bepsilon\colon\M_*^+\to \H_+$.

\item\label{Prop_item7}
One has 
$T_\gamma \tilde \bs^{-1}(\tilde \bs(\gamma))\subseteq(T_\gamma \tilde \bt^{-1}(\tilde \bt(\gamma)))^{\perp_\omega}$ 
and 
$T_\gamma \tilde \bt^{-1}(\tilde \bt(\gamma))\subseteq(T_\gamma \tilde \bs^{-1}(\tilde \bs(\gamma)))^{\perp_\omega}$ 
for every $\gamma\in\H$. 
%\end{equation}
\end{enumerate}
\end{prop}

\begin{proof}
%(i) 
\eqref{Prop_item1}
For $\dot\gamma_1, \dot\gamma_2\in T_\gamma \H$ one has 
\begin{align*} 
(J^*\omega)_\gamma(\dot\gamma_1, \dot\gamma_2)
=& \omega_{J\gamma}(J\dot\gamma_1, J\dot\gamma_2) 
=\Im (\langle J\dot\gamma_1\mid J\dot\gamma_2\rangle 
=\Im \ol{\langle \dot\gamma_1\mid \dot\gamma_2\rangle} 
= -\Im \langle \dot\gamma_1\mid \dot\gamma_2\rangle \\
=&-\omega_\gamma(\dot\gamma_1, \dot\gamma_2).
\end{align*}
The above proves \eqref{Jomega}.

%(ii) 
\eqref{Prop_item2}
This follows from $J^2=\id$.

%(iii) 
\eqref{Prop_item3}
This follows from~\eqref{PHP}.

\eqref{Prop_item7} 
See Proposition~\ref{prop:4.4}. 
\end{proof}

 \section{Modular flows and Hamiltonian flows}
\label{Sect7}

This section includes a brief discussion of the way the modular flows of a von Neumann algebra~$\M$ 
define Poisson flows of its corresponding standard groupoid $\H\tto\M_*^+$ 
associated to any standard form representation $(\M,\H,J,\P)$. 
To this end, we make  some remarks on Poisson flows of general \Banach Lie-Poisson spaces. 
Then we investigate the relation between the modular Poisson flows of the \Banach Lie-Poisson space $\M_*^h$ 
and of the presymplectic structure of the standard Lie groupoid $\widetilde{\H}\tto\M_*^+$, respectively.

 \subsection{Modular flows on standard groupoids}
\label{Subsect7.1}

We recall from \cite[Th. 3.2, Def. 3.3]{Ha73} that 
if $(\M,\H,J,\P)$ is a standard form, then 
the \emph{canonical implementation of $\Aut(\M)$} 
is the group homomorphism 
$\Aut(\M)\to U(\H)$, $g\mapsto u_g$, which satisfies 
\begin{equation}\label{impl}
g(x)=u_gxu_g^{-1}\text{ for all }x\in\M\text{ and }g\in\Aut(\M)
\end{equation}
 and 
is uniquely determined by the conditions 
$$Ju_g=u_gJ\text{ and }u_g(\P)=\P\text{ for all }g\in\Aut(\M).$$
Let $\Aut(\M_*)$ be the group of all isometric invertible operators on $\M_*$, regarded as a topological group with respect to its topology of pointwise convergence. 
For every $g\in\Aut(\M)$ we denote by $g_*\in\Aut(\M_*)$ its predual map, and we endow $\Aut(\M)$ with its topology induced via the injective group homomorphism $\Aut(\M)\to\Aut(\M_*)$, $g\mapsto g_*^{-1}$, 
using \cite[Lemma 3.5]{Ha73}. 
Then the canonical implementation is an isomorphism of topological groups from $\Aut(\M)$ onto a certain closed subgroup of $U(\H)$ with respect to the strong operator topology. 
(See \cite[Prop. 3.6]{Ha73}.)

\begin{prop}\label{prel_P1}
	Let $(\M,\H,J,\P)$ be a standard form. 
	For every $\g\in\Aut(\M)$ 
	and any orbits $\Oc_1,\Oc_2\subseteq\M_*^+$ of the standard groupoid,  if  $\g_*(\Oc_1)=\Oc_2$, 
	then the restricted map $\g_*\vert_{\Oc_1}\colon\Oc_1\to\Oc_2$ is a symplectomorphism. 
\end{prop}

\begin{proof}	
	For $j=1,2$ let $\rho_j\in\Oc_j$ with 
	$\rho_2=\rho_1\circ\alpha$, which we may since $\g_*(\Oc_1)=\Oc_2$. 
	Denote $p_j:=\sigma_*(\rho_j)\in\Ll(\M)$, $P_j:=\br^{-1}(p_j)$, 
	and $U_{\rho_j}:=\{u\in P_j: u\rho_j u^*=\rho_j\}$. 
	Then the map 
	$$\pi_j\colon P_j\to\Oc_j,\quad u\mapsto u\rho_j u^*$$
	induces an $U(\M)$-equivariant diffeomorphism 
	\begin{equation}\label{prel_P1_proof_eq1}
	P_j/ U_{\rho_j}\to\Oc_j, \quad 
	u U_{\rho_j}\mapsto u\rho_j u^*.
	\end{equation}
	On the other hand, it is easily checked that the equality $\rho_1=\rho_2\circ\g$ 
	implies that for every $u\in\M$ we have 
	$$u\in P_1\iff \g(u)\in P_2$$
	and 
	$$u\in U_{\rho_1} \iff \g(u)\in U_{\rho_2}$$
	and this shows that 
	the $U(\M)$-equivariant map 
	$$\widetilde{\g}\colon P_1/ U_{\rho_1}\to P_2/ U_{\rho_2},\quad 
	u U_{\rho_1}\mapsto \g(u)U_{\rho_2}$$
	is bijective. 
	One also has the commutative diagram 
	$$
	\xymatrix{
		P_1 \ar[r]^{\g\vert_{P_1}} \ar[d] & P_2 \ar[d]\\
		P_1/ U_{\rho_1} \ar[r]^{\widetilde{\g}} & P_2/ U_{\rho_2}
	}
	$$ 
	whose vertical arrows are the submersions $P_j\to P_j/ U_{\rho_j}$, $u\mapsto uU_{\rho_j}$. 
	(See Theorem~\ref{thm:32}\eqref{thm:32_item1}.)
	Since the top horizontal arrow in the above diagram is a diffeomorphism  
	(for, $P_j$ is a submanifold of $\M$ and the $*$-automorphism $\alpha\colon \M\to\M$ 
	is in particular a diffeomorphism), it then follows that also $\widetilde{\g}$ 
	is smooth. 
	Replacing $\g$ by $\g^{-1}$ in the above reasoning, 
	we obtain that $\widetilde{\g}$ is actually a diffeomorphism. 
	Now, using the commutative diagram 
	$$
	\xymatrix{
		P_1/ U_{\rho_1}  \ar[r]^{\widetilde{\g}} \ar[d]
		& P_2/ U_{\rho_2}  \ar[d]\\
		\Oc_1 & \ar[l]_{\g_*\vert_{\Oc_2}}  \Oc_2
	}
	$$
	whose vertical arrows are the diffeomorphisms \eqref{prel_P1_proof_eq1}, 
	it follows that $\g_*\vert_{\Oc_1}\colon\Oc_1\to\Oc_2$ is a diffeomorphism. 
	
	Thus, in order to prove that $\g_*\vert_{\Oc_1}\colon\Oc_1\to\Oc_2$ is a symplectomorphism, 
	it remains to check that 
	\begin{equation}\label{prel_P1_proof_eq2}
	(\g_*\vert_{\Oc_2})^*(\tilde{\omega}_1)=\tilde{\omega}_2,
	\end{equation} 
	where $\tilde{\g}_j$ is the canonical symplectic form on $\Oc_j$ for $j=1,2$, cf. Theorem~\ref{thm:orbit}. 
	To this end let 
	$u_\g\colon\H\to\H$ be the unitary implementation of $\g\in\Aut(\M)$, 
	hence $\g(x)=u_\g xu_\g^*$ for all $x\in \M$. 
	Fix any $\gamma_1\in \H$ with $p_1(\gamma_1)=\gamma_1$, and denote $\gamma_2:=u_\g(\gamma_1)$, 
	which implies $p_2(\gamma_2)=\gamma_2$ since $p_2=\g(p_{\gamma_1})=u_\g p_1 u_\g^*$. 
	Defining $\iota_{\gamma_j}\colon P_j\to\H$, $\iota_{\gamma_j}(v):=v\gamma_j$ as in \eqref{pi}, 
	we then obtain 
	the commutative diagram 
	\begin{equation}\label{prel_P1_proof_eq3}
	\xymatrix{
		\H \ar[r]^{u_\g} & \H \\
		P_1 \ar[r]^{\g\vert_{P_1}} \ar[u]^{\iota_{\gamma_1}}  \ar[d]_{\pi_1} & P_2 \ar[d]^{\pi_2} \ar[u]_{\iota_{\gamma_2}}  \\
		\Oc_1 & \ar[l]_{\g_*\vert_{\Oc_2}}  \Oc_2 
	}
	\end{equation} 
	where $u_\g$ preserves the symplectic form $\omega$ of $\H$ since it is a unitary operator. 
	Since the symplectic form $\tilde{\omega}_j$ is obtained from $\omega$ by reduction (see Theorem~\ref{thm:orbit})
	it then follows by \eqref{prel_P1_proof_eq3} that $\g_*\vert_{\Oc_1}\colon\Oc_1\to\Oc_2$ is a symplectomorphism, and we are done. 
\end{proof}

For every orbit $\Oc\subseteq\M_*^+$ of the standard Lie groupoid $\widetilde{\H}\tto\M_*^+$ with its source mapping $\bs=E\colon\H\to\M_*^+$, we denote 
$$\widetilde{\H}_{\Oc}:=\bs^{-1}(\Oc)\subseteq\widetilde{\H}$$
hence $\widetilde{\H}_{\Oc}\tto\Oc$ is a transitive subgroupoid of the standard Lie groupoid corresponding to the orbit~$\Oc$.

\begin{prop}\label{flow}
If $(\M,\H,J,\P)$ is a standard form and $\psi$ is a normal semifinite faithful weight of $\M$ with its corresponding modular flow 
$\R\ni t\mapsto \sigma^t_\psi\in \Aut(\M)$, then the following assertions hold for every orbit $\Oc\subseteq\M_*^+$ of the standard Lie groupoid $\widetilde{\H}\tto\M_*^+$ and for all $t\in\R$.
\begin{enumerate}[{\rm(i)}]
	\item\label{flow_item1} One has $(\sigma^t_\psi)_*(\Oc)=\Oc$ 
	and the mapping  $(\sigma^t_\psi)_*\vert_{\Oc}\colon\Oc\to\Oc$ is a symplectomorphism.  
	\item\label{flow_item2} The canonical implementation $u^t_\psi\in U(\H)$ of $\sigma^t_\psi$ leaves $\H_{\Oc}$ invariant and the mapping 
	$u^t_\psi\vert_{\widetilde{\H}_{\Oc}}\colon \widetilde{\H}_{\Oc}\to \widetilde{\H}_{\Oc}$ is a diffeomorphism that leaves invariant the presymplectic form $\widetilde{\bomega}\vert_{\Oc}$. 
	\item\label{flow_item3} 
	The unitary operator $u^t_\psi$ gives a Lie groupoid automorphism $\widetilde{u^t_\psi}\colon \widetilde{\H}\to \widetilde{\H}$ 
	that leaves invariant the presymplectic form $\widetilde{\bomega}$.
\end{enumerate} 
\end{prop}

\begin{proof}
\eqref{flow_item1} 
The groupoid orbit $\Oc$ is a union of unitary orbits in $\M_*^+$ by Proposition~\ref{comp}, and every such unitary orbit is preserved by $(\sigma^t_\psi)_*$ by \cite[Prop. 12.6]{HS90a}. 
Then $(\sigma^t_\psi)_*\vert_{\Oc}\colon\Oc\to\Oc$ is a symplectomorphism by Proposition~\ref{prel_P1}

\eqref{flow_item2} 
The assertion will follow by \eqref{flow_item1} as soon as we will have shown 
 that if $g\in\Aut(\M)$ satisfies $g_*(\Oc)=\Oc$, then $u_g(\widetilde{\H}_{\Oc})=\H_{\Oc}$, and moreover the mapping 
 $u_g\vert_{\widetilde{\H}_{\Oc}}\colon \widetilde{\H}_{\Oc}\to \widetilde{\H}_{\Oc}$ is a diffeomorphism that leaves invariant the presymplectic form $\widetilde{\bomega}\vert_{\H_\Oc}$. 
  
To prove this we first note that, for arbitrary 
$g\in\Aut(\M)$ and 
$\gamma\in\H$ one has by \eqref{impl},
$$E(u_g\gamma)=g_*^{-1}(E(\gamma))\in\M_*^+.$$
Then, if $g_*(\Oc)=\Oc$, we directly obtain $u_g(E^{-1}(\Oc))=E^{-1}(\Oc)$, 
that is, $u_g(\widetilde{\H}_{\Oc})=\widetilde{\H}_{\Oc}$. 

On the other hand, every automorphism $g\in\Aut(\M)$ naturally defines the automorphism 
$$A_g\colon \U(\M)\ast\M_*^+\to \U(\M)\ast\M_*^+,\quad (\widetilde{\omega}v,\rho)\mapsto (g(v),\rho\circ g^{-1}) $$
of the Lie groupoid $\U(\M)\ast\M_*^+\tto\M_*^+$, 
which further defines via Theorem~\ref{grpd_act} the automorphism 
$$\Phi\circ A_g\circ \Phi^{-1}\colon \widetilde{\H}\to\widetilde{\H}$$
of the Lie groupoid $\widetilde{\H}\tto\M_*^+$. 
One can now show that $\Phi\circ A_g\circ \Phi^{-1}=u_g$, using the fact that $g(x)=u_gxu_g^{-1}$ for all $x\in\M$. 
This implies that the mapping 
$\widetilde{u_g}\colon\widetilde{\H}\to\widetilde{\H}$, $\widetilde{u_g}(\gamma):=u_g\gamma$, is smooth. 
Moreover, since $u_g\in U(\H)$, one can easily check that $\widetilde{u_g}^*\widetilde{\omega}=\widetilde{\omega}$, 
and this concludes the proof. 
\end{proof}

 \subsection{Infinitesimal aspects}
\label{Subsect7.2}
 
We will now take a brief look at the infinitesimal generators of the flows that we discussed so far in this section. 
Loosely speaking, the main idea here is that if one has 
a standard form representation $(\M,\H,J,\P)$, where $\H$ is a \emph{separable} Hilbert space, 
 then in general a Hamiltonian flow on $\M_*$  generates a Hamiltonian function on $\M_*^h$ 
because such a flow is given by a one-parameter unitary group in $\M$ 
whose infinitesimal generator is a self-adjoint operator that is affiliated to $\M$. 
However, in general, for a Poisson flow on $\M_*$, its corresponding infinitesimal generator may not be affiliated to $\M$, 
and for this reason it is not possible in general to define its corresponding Hamiltonian function. 
In this subsection, by \emph{Poisson flow} on the predual of a von Neumann algebra $\M$ we mean 
the natural action of 
any continuous 1-parameter subgroup of the topological group $\Aut(\M)$ on the predual $\M_*$ or on is self-adjoint part $\M_*^h$, 
where continuity is meant with respect to the topology of this group described at the beginning of Subsection~\ref{Subsect7.1}. 
If a Poisson flow on $\M_*$ consists only of inner automorphisms of $\M$, then we call it a \emph{Hamiltonian flow}. 
The motivation for this terminology is that for any $g\in\Aut(\M)$ its predual mapping $g_*\colon\M_*\to\M_*$ 
is a linear Poisson map with respect to the Lie-Poisson bracket~\eqref{bracketi}, 
while the connection with the Hamiltonian fomalism is discussed below. 

We assume the following setting: 
\begin{itemize}
	\item $\H$ separable Hilbert space
	\item $\M=\M''\subseteq L^\infty(\H)$ standard, 
	with modular operator $\Delta\ge 0$
	\item $A:=\log\Delta$ self-adjoint unbounded operator in~$\H$ 
	\item $\sigma\colon\R\to\Aut(\M)$, $\sigma(t)x:=e^{i tA}xe^{-i tA}$
	\item $\sigma_*\colon\R\to\Iso(\M_*^h)$, $\sigma_*(t)\psi:=\psi\circ\sigma(t)$
\end{itemize}

\begin{rem}\label{modular_rem}
	\normalfont
	Here we have denoted by $\Iso(\Xc)$ the group of all isometric bijective linear operators on~$\Xc$ 
for any real complete normed space~$\Xc$,  
	and we regard $\Iso(\Xc)$ as a topological group with its topology of pointwise convergence on~$\Xc$. 
	We note that the continuity of the composition and inversion mappings of the group $\Iso(\Xc)$ follow by the estimates 
	$$\Vert (T_0^{-1}-T^{-1})x\Vert=\Vert T^{-1}(T-T_0)T_0^{-1}x\Vert =\Vert (T-T_0)T_0^{-1}x\Vert$$
	and 
	$$\Vert(TS-T_0S_0)x\Vert\le \Vert T(S-S_0)x\Vert +\Vert (T-T_0)S_0x\Vert
	=\Vert (S-S_0)x\Vert +\Vert (T-T_0)S_0x\Vert$$
	%where we have the operator equations 
	%$T_0^{-1}-T^{-1}=T^{-1}(T-T_0)T_0^{-1}$ and $TS-T_0S_0=T(S-S_0)+(T-T_0)S_0$ 
	for arbitrary $T,S,T_0,S_0\in\Iso(\Xc)$ and $x\in\Xc$.

	The above map $\sigma_*\colon(\R,+)\to\Iso(\M_*^h)$ is then a continuous homomorphism 
	of topological groups by the remarks on the topology of $\Aut(\M)$ 
at the beginning of Subsection~\ref{Subsect7.2}. 
\end{rem}

Using that for every $t\in\R$ and $x\in\M$ one has $e^{i tA}x e^{-i tA}\in\M$, 
prove by differentiation with respect to $t$ that if $x=x^*\in\M$ and $[A,x]\in L^\infty(\H)$, 
then actually $[A,x]\in\M$. 

\begin{defn}
	\normalfont 
	We define 
	$$(\M_*^h)^\infty:=\{\psi\in\M_*^h\mid \sigma_*(\cdot)\psi\in\Ci(\R,\M_*^h)\}.$$ 
	This is the space of differentiable vectors with respect to the strongly continuous representation 
	$\sigma_*\colon(\R,+)\to\Iso(\M_*^h)$ 
	(see Remark~\ref{modular_rem})
	hence $(\M_*^h)^\infty$ has the natural structure of a real Fr\'echet space and 
	is a dense 
	linear subspace of $\M_*^h$ that is invariant under the representation~$\sigma_*$. 
\end{defn}

We fix $\xi_0\in (\M_*^h)^\infty$ with its unitary orbit $\Oc:=\U(\M).\xi_0\subseteq \M_*^h$, 
and define 
$$\Oc^\infty:=\Oc\cap (\M_*^h)^\infty.$$
Then $\Oc^\infty$ is a subset of $\Oc$ that is invariant under the action of the $1$-parameter group 
of symplectomorphisms defined by $\sigma_*(\cdot)\vert_{\Oc}$ and $\xi_0\in\Oc^\infty$. 

It would be interesting to establish conditions for $\Oc^\infty$ to be dense in $\Oc$, and to study what additional topological or differential properties the set $\Oc^\infty$ has.

The infinitesimal generator of the \emph{continuous} 1-parameter group of diffeomorphisms 
$\sigma_*\colon\R\to\Diff(\Oc)$ is the vector field $X$ defined at any point $\xi\in\Oc^\infty$, 
for arbitrary $f\in\Ci(\M_*^h,\R)$ satisfying $[A,f'_\xi]\in\M$. 
Here one has the commutator of the unbounded operator $A$ with the bounded operator $f'_\xi$. 
We note that $f'_\xi\colon \M_*^h\to\R$, hence $f'_\xi\in\ug(\M)\subset\M$. 

Using the duality pairing $\langle\cdot,\cdot\rangle\colon\ug(\M)\times\M_*^h\to\R$, 
$(x,\eta)\mapsto-i\eta(x)$, one has 
\begin{equation}\label{mod_eq1}
(X(f))(\xi)
:=\frac{\de}{\de t}\Bigl\vert_{t=0}f(\sigma_*(t)\xi) 
=\langle f'_\xi,\dot\sigma_*(0)\xi\rangle
=\langle \dot\sigma(0)f'_\xi,\xi\rangle
=\xi([A,f'_\xi])
\end{equation}
where we have also used the above formula of $\sigma(t)$. 
To explain the above computation, we also note that $T_\xi\Oc\hookrightarrow\M_*^h$, 
$\dot\sigma_*(0)\colon \M_*^h\to \M_*^h$, and $\dot\sigma_*(0)(T_\xi\Oc)\subseteq T_\xi\Oc$. 

On the other hand, for every $h\in\Ci(\M_*^h,\R)$ and $\xi\in\Oc$ we have 
\begin{equation}\label{mod_eq2}
\{h,f\}(\xi)=\xi([h'_\xi,f'_\xi]).
\end{equation}
If $A\in\M$, then equations \eqref{mod_eq1}--\eqref{mod_eq2} show that 
$X$ is the Hamiltonian vector field of the function $h_A\colon\M_*^h\to\R$, $h_A(\xi):=\langle \xi,A\rangle$. 
It would be interesting to study analogues of the above function $h_A$ if $A=A^*$ is an unbounded operator.

 \section{Standard groupoids in the special case of type I factors} 
\label{Sect8}

In this final section we illustrate the general results of the preceding sections by a brief discussion of standard groupoids of type~I factors 
$\M\simeq L^\infty(\H_0)$, for an arbitrary separable 
complex Hilbert space $\H_0$. 
We also point out the significance of this groupoid for quantum physics. 

Define $\H:=L^2(\H_0)$, the ideal of Hilbert-Schmidt operators on~$\H_0$, 
and for every $a\in L^\infty(\H_0)$ let $\lambda(a)\in L^\infty(\H)$ by $\lambda(a)\gamma:=a\gamma$ for all $\gamma\in\H$. 
Then $\M:=\lambda(L^\infty(\H_0))\subseteq L^\infty(\H)$ is again a type~I factor, and the mapping 
$$\lambda\colon L^\infty(\H_0)\to\M, \quad a\mapsto\lambda(a)$$ 
is a $*$-isomorphism. 
We denote by $\lambda_*\colon \M_*\to(L^\infty(\H_0))_*= L^1(\H_0)$ its corresponding predual mapping, which is an isometric isomorphism of Banach spaces, 
and its restriction and corestriction to the positive cones is a homeomorphism denoted by 
$$\lambda_*^+\colon \M_*^+\to(L^\infty(\H_0))_*^+= L^1(\H_0)^+.$$
The triple $(\M,\H,J, \P)$ is a standard form of the von Neumann algebra $\M$, where 
\begin{itemize}
	\item $J\colon L^2(\H_0)\to L^2(\H_0)$, $J(\gamma):=\gamma^*$; 
	\item $ \P:=L^2(\H_0)^+=\{\rho\in L^2(\H_0): \rho\ge 0\}$.
\end{itemize}

\begin{lem}\label{typeI_lemma}
	The above standard form $(\M,\H,J, \P)$ of the von Neumann algebra $\M=\lambda(L^\infty(\H_0))\subseteq L^\infty(\H)$ has the following properties.  
	\begin{enumerate}[{\rm(i)}]
		\item\label{typeI_lemma_item1} 
		The composition of homeomorphisms 
		$$\xymatrix{L^2(\H_0)^+= \P \ar[r]^{\ \ \ \ E\vert_{\P}} & \M_*^+ \ar[r]^{\lambda_*^+\ \ } &  L^1(\H_0)^+}$$
		is  given by $(\lambda_*^+\circ\theta)(\rho)=\rho^2$ for all $\rho\in L^2(\H_0)^+$. 
		\item\label{typeI_lemma_item2}
		If $\gamma=v_\gamma\vert\gamma\vert$ is the polar decomposition of any $\gamma\in L^2(\H_0)$ 
		with respect to the above standard form, 
		then $\vert\gamma\vert=(\gamma^*\gamma)^{1/2}\in L^2(\H_0)^+$ 
		and $\gamma=v\vert\gamma\vert$ is the usual polar decomposition of the operator $\gamma\in L^\infty(\H_0)$, involving the partial isometry $v:=\lambda^{-1}(v_\gamma)\in L^\infty(\H_0)$. 
	\end{enumerate}
\end{lem}

\begin{proof}
	\eqref{typeI_lemma_item1}	
	Let $\rho\in L^2(\H_0)^+$ be arbitrary. 
	One has $E(\rho)\colon\M\to \C$, 
	$\omega_\rho(x)=\langle\rho\mid x\rho\rangle$. 
	Therefore, using the duality pairing $$\langle\cdot,\cdot\rangle\colon\M_0\times(\M_0)_*= L^\infty(\H_0)\times L^1(\H_0)\to \C,\quad  
	\langle a,\delta\rangle=\Tr(a\delta),$$ 
	one obtains for arbitrary $a\in L^\infty(\H_0)$, 
	$$
	\langle a, \lambda_*^+(  E(\rho))\rangle
	= \langle \lambda(a),   E(\rho)\rangle 
	=\langle\rho\mid \lambda(a)\rho\rangle
	=\langle\rho\mid a\rho\rangle
	=\Tr(\rho^*(a\rho)) 
	=\Tr(a\rho^2). 
	$$
	Thus  
	\begin{equation}\label{typeI_lemma_proof_eq0}
	\langle a, \lambda_*^+(  E(\rho))\rangle
	=\langle a,\rho^2\rangle 
	\text{ for all }a\in L^\infty(\H_0).
	\end{equation}
	This implies $\lambda_*^+(  E(\rho))=\rho^2$. 
	
	\eqref{typeI_lemma_item2} 
	Let $\gamma\in L^2(\H_0)$ be arbitrary, with its usual (operator theoretic) polar decomposition 
	$$\gamma=v\rho,$$ 
	where 
	$$\rho:=(\gamma^*\gamma)^{1/2}\in L^2(\H_0)^+$$ 
	and $v\in L^\infty(\H_0)$ is the partial isometry for which $v^*v$ is the orthogonal projection onto $(\Ker\rho)^\perp$. 
	The equality  $\gamma=v\rho$ can be written 
	\begin{equation}\label{typeI_lemma_proof_eq1}
	\gamma=\lambda(v)\rho,
	\end{equation}
	where $\lambda(v)\in L^\infty(\H)$ is a partial isometry and $\rho\in \P$. 
	In order to show that \eqref{typeI_lemma_proof_eq1} 
	is the polar decomposition with respect to the standard form $(\M,\H,J, \P)$ it suffices, by the uniqueness property of that decomposition, 
	to check that $\lambda(v)^*\lambda(v)=p_\rho$. 
	That is, %by \eqref{suppvect}, 
	we must prove that 
	\begin{equation}\label{typeI_lemma_proof_eq2}
	\lambda(v)^*\lambda(v)= \sigma_*(E(\rho))\  (\in \Ll(\M)). 
	\end{equation}
	Since $\lambda\colon\M_0\to\M$ is a $*$-isomorphism, 
	it is easily seen that $\lambda^{-1}( \sigma_*(E(\rho)))= \sigma_*(E(\rho)\circ\lambda)$.  
	Therefore, also using $\lambda(v)^*\lambda(v)=\lambda(v^*v)$,  one can see that \eqref{typeI_lemma_proof_eq2} is equivalent to 
	\begin{equation}\label{typeI_lemma_proof_eq3}
	v^*v= \sigma_*(E(\rho)\circ\lambda)
	\  
	(\in \Ll(\M_0)).
	\end{equation}
	On the other hand, using the second equality in \eqref{typeI_lemma_proof_eq0}, it is straightforward to check that 
	$ \sigma_*(E(\rho)\circ\lambda)$ is equal to the orthogonal projection on $(\Ker(\rho^2))^\perp$. 
	Since $\rho=\rho^*$, one has $\Ker(\rho^2)=\Ker\rho$, 
	hence $ \sigma_*(E(\rho)\circ\lambda)$ is equal to the orthogonal projection on $(\Ker\rho)^\perp$, which is equal to $v^*v$ 
	(by the choice of $v$). 
	Thus \eqref{typeI_lemma_proof_eq3} holds true, and this completes the proof. 
\end{proof}

\begin{prop}\label{typeI_prop}
	There exists the groupoid $L^2(\H_0)\tto L^1(\H_0)^+$ having the following structural maps: 
	\begin{itemize} 
		\item the unit map $\bepsilon\colon L^1(\H_0)^+\to L^2(\H_0)$, $\bepsilon(\varphi)=\varphi^{1/2}$; 
		\item the source map $\bs\colon L^2(\H_0)\to L^1(\H_0)^+$, $\bs(\gamma):=\gamma^*\gamma$; 
		\item the target map $\bt\colon L^2(\H_0)\to L^1(\H_0)^+$, $\bt(\gamma)=\gamma\gamma^*$; 
		\item the inversion map $J\colon L^2(\H_0)\to L^2(\H_0)$, $J(\gamma):=\gamma^*$; 
		\item the multiplication map 
		$$\bmu\colon L^2(\H_0)\ast L^2(\H_0)\to L^2(\H_0),\quad  
		\bmu(\gamma_1,\gamma_2):=v_1v_2\vert\gamma_2\vert,$$ 
		where $\gamma_j=v_j\vert\gamma_j\vert$ for $j=1,2$ are the canonical polar decompositions.
	\end{itemize}
\end{prop}

\begin{proof}
	Use Theorem~\ref{thm:53} 
along with Lemma~\ref{typeI_lemma}. 
\end{proof}

For the case considered in this section, the coadjoint groupoid orbit $\Oc_{\rho_0}$ corresponding to any $\rho_0\in L^1(\H_0)^+$ can be described using its spectral decomposition 
\begin{equation}
\label{hotel1}
\rho_0=\sum_{n\ge1}\rho_n\ket{n}\bra{n}
\end{equation}
where $\rho_n\ge0$ for every $n\ge1$ and $\{\ket{n}\}_{n\ge1}$ is an orthonormal sequence consisting of eigenvectors of $\rho_0$. 
The support of $\rho_0$ is 
\begin{equation}
\label{hotel2}
p_0=\sum_{\stackrel{\scriptstyle n\ge1}{\rho_n>0}}\ket{n}\bra{n}
\in\Ll(L^\infty(\H_0))
\end{equation}
and 
\begin{equation}
\label{hotel3}
\gamma_0:=\bepsilon(\rho_0)=\sum_{n\ge1}\rho_n^{1/2}\ket{n}\bra{n}
\in L^2(\H_0)^+
\end{equation}

\begin{example}
We now discuss the simplest instance of the transitive groupoid $$\U_{p_0}(L^\infty(\H_0))\ast\Oc_{\rho_0}\tto\Oc_{\rho_0},$$ 
corresponding to the case 
\begin{equation}
\label{hotel4}
\gamma_0=r\ket{\delta_0}\bra{\delta_0}
\end{equation}
where $\delta_0\in\H_0$ satisfies $\langle\delta_0\mid\delta_0\rangle=1$. 
If 
$$\mathbb{S}_r:=\{\gamma\in L^2(\H_0):\Vert\gamma\Vert=r\}$$ 
is the sphere in $L^2(\H_0)$ having its center at $0$ and its radius $r$, one then has $\gamma_0\in\mathbb{S}_r$. 

In this case we have $p_0=\ket{\delta_0}\bra{\delta_0}$ and therefore 
$$P_0=\{\ket{\delta}\bra{\delta_0}\in L^\infty(\H_0) :\delta\in\H_0,\ \langle\delta\mid\delta\rangle=1\}.$$
The stabilizer groups of $\rho_0\in L^1(\H_0)^+$ and $p_0\in\Ll(L^\infty(\H_0))$ are 
$$U_{\rho_0}=U_{p_0}\simeq{\rm U}(1).$$
The mapping $\iota_{\gamma_0}$ defined in \eqref{pi}
is now 
\begin{equation}
\label{hotel5}
\iota_{\gamma_0}\colon P_0\to L^2(\H_0),\quad 
\iota_{\gamma_0}(\ket{\delta}\bra{\delta_0})
=r\ket{\delta}\bra{\delta_0}.
\end{equation}
That is, $\iota_{\gamma_0}(u)=ru\in\mathbb{S}_r$ for every $u\in P_0$, 
and it is clear that this mapping $\iota_{\gamma_0}\colon P_0\to\mathbb{S}_r$ is an injective immersion. 
By Theorem~\ref{thm:32} one has 
\begin{equation}
\label{hotel7}
\Oc_{\rho_0}\simeq P_0/U_{\rho_0}\simeq P_0/{\rm U}(1). 
\end{equation}
Plugging \eqref{hotel5} into \eqref{Gamma0} we obtain 
\begin{equation}
\label{hotel6}
(\iota_{\gamma_0}^*\Gamma)(\ket{\delta}\bra{\delta_0})
=i\Tr(r(\ket{\delta}\bra{\delta_0})^*\ket{\de \delta}\bra{\delta_0})
=ir\langle \delta\mid\de \delta\rangle
\end{equation}
where we recall that $\langle\delta_0\mid\delta_0\rangle=1$ and 
\begin{equation}
\label{hotel8}
\Gamma(\gamma)=i\Tr(\gamma^*\de\gamma).
\end{equation}
The mapping $\iota_{\gamma_0}(P_0)\ni r\ket{\delta}\bra{\delta_0}\mapsto \ket{\delta}\in\H_0$ allows us to identify $\iota_{\gamma_0}(P_0)$ with the unit sphere $\mathbb{S}(\H_0):=\{\delta\in\H_0:\langle\delta_0\mid\delta_0\rangle=1\}$ in $\H_0$. 
Therefore $\Oc_{\rho_0}\simeq P_0/U_{\rho_0}\simeq \mathbb{S}(\H_0)/{\rm U}(1)\simeq\CP(\H_0)$.  
Via this identification, the symplectic form $\widetilde{\omega}_{\rho_0}$ is the Fubini-Study form scaled by $r$, 
\begin{equation}
\label{hotel9}
\omega_{\rm FS}=ir\bar{\partial}\partial\langle \delta\mid\delta\rangle 
\end{equation}
defined on the complex projective space $\CP(\H_0)$ 
(cf. for instance \cite[\S XV.1]{Ne00} and also \cite[Ex. 4.32]{Be06}). 
See \eqref{hotel6}. 

We now recall that 
$$\CP(\H_0)\simeq\Ll_{p_0}(L^\infty(\H_0)),$$ 
since any point of the projective space 
$$[\delta]:=\{z\delta:z\in{\rm U}(1)\}\in\CP(\H_0)$$ 
can be identified with with the rank-one orthogonal projection $$\ket{\delta}\bra{\delta}\in \Ll_{p_0}(L^\infty(\H_0)),$$ 
where $\delta\in\mathbb{S}(\H_0)$. 

Using the maps defined in \eqref{SQ4_proof_eq1}--\eqref{SQ4_proof_eq1.5}, we find  
as in Proposition~\ref{gauge_isom}
$$\U_{p_0}(L^\infty(\H_0))\simeq\frac{P_0\times P_0}{{\rm U}(1)},$$ 
consisting of elements $\ket{\delta}\bra{\tau}$, where $\delta,\tau\in\mathbb{S}(\H_0)$ and, for 
$$u=\ket{\delta}\bra{\delta_0}\text{ and }v=\ket{\tau}\bra{\delta_0},$$ 
we have 
\begin{align*}
\Psi_{\gamma_0}^*\Gamma
&=i\Tr((u\gamma_0v^*)^*\de(u\gamma_0v^*)) \\
&=ir\Tr((\ket{\delta}\langle\delta_0\mid\delta_0\rangle\langle\delta_0\mid\delta_0\rangle\bra{\tau})^*\de(\ket{\delta}\langle\delta_0\mid\delta_0\rangle\langle\delta_0\mid\delta_0\rangle\bra{\tau})) \\
&=ir\Tr((\ket{\delta}\rangle\bra{\tau})^*\de(\ket{\delta}\rangle\bra{\tau})) \\
&=ir\Tr(\ket{\tau}\rangle\bra{\delta}
(\ket{\de\delta}\rangle\bra{\tau}+\ket{\delta}\rangle\bra{\de\tau})) \\
&=ir\langle\delta\mid \de\delta\rangle 
+ir\langle\de\tau\mid\tau\rangle.
\end{align*}
We can now draw the following conclusions from the above discussion. 
\begin{enumerate}[{\rm(i)}]
	\item The coadjoint action groupoid $\U_{p_0}(L^\infty(\H_0))\ast\Oc_{\rho_0}\tto\Oc_{\rho_0}$ can be identified with the action groupoid 
	$\U_{p_0}(L^\infty(\H_0))\ast\Ll_{p_0}(L^\infty(\H_0))\tto \Ll_{p_0}(L^\infty(\H_0))$. 
	\item The presymplectic form $\Psi_{\gamma_0}^*\omega=\de(\Psi_{\gamma_0}^*\Gamma)$ is expressed in terms of the scaled Fubini-Study form as
\begin{equation}
\label{hotel11}
\Psi_{\gamma_0}^*\omega=\bt^*\omega_{\rm FS}-\bs^*\omega_{\rm FS}
\end{equation} 
\item Applying the reduction procedure to the multiplicative presymplectic structure of the \Banach Lie groupoid $$(\U_{p_0}(L^\infty(\H_0))\ast\Oc_{\rho_0},\Psi_{\gamma_0}^*\omega)$$ 
we obtain the symplectic pair groupoid 
$$\CP(\H_0)\times\overline{\CP(\H_0)}\tto\CP(\H_0).$$ 
\end{enumerate}
The last two of these conclusions actually illustrate Proposition~\ref{SQ6}. 
\end{example}

\begin{rem}
The occurrence of the complex projective space  $\CP(\H_0)$ in the above discussion is not incidental. 
It was proved in \cite{BR05} that the 
characteristic distributions of 
 several natural classes of infinite-dimensional real \Banach Lie–Poisson
spaces are integrable, the integral manifolds
being symplectic leaves just as in finite dimensions. 
This holds true in particular for  the real \Banach Lie–Poisson spaces
given by the self-adjoint parts of preduals of arbitrary $W^*$-algebras, as
well as of certain Banach ideals of compact operators on Hilbert spaces, 
including, but not limited to, the Schatten ideals.
In this context, one found that the symplectic leaves sometimes have structures of complex manifolds modeled on Banach spaces, 
and these complex structures are compatible with the symplectic structures of these leaves. 
Thus, one actually obtains K\"ahler structures on some symplectic leaves, and a very special instance of this phenomenon is the above complex projective space~$\CP(\H_0)$. 
\end{rem}

\begin{rem}
We finally point out a physical interpretation of the  groupoid $$\U_{p_0}(L^\infty(\H_0))\ast\Oc_{\rho_0}\tto\Oc_{\rho_0}.$$ 
The partial isometry $\ket{\delta}\bra{\tau}\in\U_{p_0}(L^\infty(\H_0))$ 
realizes the transition 
$$\ket{\tau}\bra{\tau}\mapsto \bigl(\ket{\delta}\bra{\tau}\bigr) \bigl(\ket{\tau}\bra{\tau}\bigr) \bigl(\ket{\delta}\bra{\tau}\bigr)^*=\ket{\delta}\bra{\delta}
$$
between the pure states $\ket{\tau}\bra{\tau},\ket{\delta}\bra{\delta}\in
\Ll_{p_0}(L^\infty(\H_0))\simeq \CP(\H_0)$ of a physical system. 
The scalar product $\langle\delta\mid\tau\rangle$ in $\H_0$ 
and the square of its absolute value $\vert\langle\delta\mid\tau\rangle\vert^2$ describe the amplitude and probability for that transition, respectively. 
One can also consider a sequence of transitions between pure states $\ket{\delta_k}\bra{\delta_k}\mapsto\ket{\delta_{k+1}}\bra{\delta_{k+1}}$ realized by the partial isometries $\ket{\delta_{k+1}}\bra{\delta_k}$, where $k=0,1,\dots,F$. 
Then, according to Feynman's composition rules for the transition amplitudes, the amplitude of the transition from the initial state $\ket{\delta_0}\bra{\delta_0}$ to the final state $\ket{\delta_F}\bra{\delta_F}$ through this sequence of transitions is given by the product $\prod\limits_{k=0}^F\langle\delta_k\mid\delta_{k+1}\rangle$. 
We refer to \cite{Od88} for a more exhaustive discussion of these aspects. 
See also \cite{SW16} for other applications of Lie groupoids to problems of prequantization. 
\end{rem}

 \subsection*{Acknowledgment} 
Thanks are due to C\'edric Arhancet for pointing out a relevant reference on modular theory
and to Aneta Sli\.zewska and Tomasz Goli\'nski for their assistance in preparing this paper. 
The first-named author acknowledges financial
support from a grant of the Ministry of Research, Innovation and Digitization, CNCS--UEFISCDI, project number PN-IV-P1-PCE-2023-0264, within PNCDI IV. 

%\end{document}

%\backmatter

%\thebibliography{9999}


\begin{thebibliography}{999999}

\bibitem%[ACL18]
{ACL18}
E.~Andruchow, E.~Chiumiento, G.~Larotonda, 
{\it Geometric significance of Toeplitz kernels}. 
J. Funct. Anal. {\bf 275} (2018), no. 2, 329--355.

\bibitem%[ACM05]
{ACM05}
E.~Andruchow, G.~Corach, M.~Mbekhta, 
{\it On the geometry of generalized inverses}. 
Math. Nachr. {\bf 278} (2005), no. 7--8, 756--770. 


\bibitem%[ACL18]
{ALR10}
E.~Andruchow, G.~Larotonda, L.~Recht, 
{\it Finsler geometry and actions of the $p$-Schatten unitary groups}. Trans. Amer. Math. Soc. {\bf 362} (2010), no. 1, 319--344. 

\bibitem%[AV05]
{AV05}
E.~Andruchow, A.~Varela, 
{\it $C^*$-modular vector states}. 
Integral Operators Operator Theory {\bf 52} (2005), 149--163. 


\bibitem%[B10]
{Be06}
D.~Belti\c t\u a, 
``Smooth homogeneous structures in operator theory''. 
Chapman \& Hall/CRC Monographs and Surveys in Pure and Applied Mathematics, 137. Boca Raton, FL, 2006. 

\bibitem%[B10]
{B10}
D.~Belti\c t\u a, 
{\it Lie theoretic significance of the measure topologies associated with a finite trace}. 
Forum Math. {\bf 22} (2010), no. 2, 241--253. 

\bibitem%[BGJP19]
{BGJP19}
D.~Belti\c t\u a, T.~Goli\'nski, G.~Jakimowicz, F.~Pelletier,  
{\it Banach-Lie groupoids and generalized inversion}. 
J. Funct. Anal. {\bf 276} (2019), no. 5, 1528--1574.

\bibitem{BGT18}
D.~Belti\c t\u a, T. Goli\'nski, A.-B. Tumpach, 
{\it Queer Poisson brackets}.
J. Geom. Phys. {\bf 132} (2018), 358-362. 

\bibitem%[BR05]
{BR05}
D.~Belti\c t\u a, T.S.~Ratiu, 
{\it Symplectic leaves in real Banach Lie-Poisson spaces}. 
Geom. Funct. Anal. {\bf 15} (2005), no. 4, 753--779.


\bibitem%[BL04]
{BL04}
D.P.~Blecher, Ch.~Le Merdy, 
``Operator algebras and their modules ---an operator space approach.'' 
London Mathematical Society Monographs. New Series, 30. Oxford Science Publications. The Clarendon Press, Oxford University Press, Oxford, 2004.

\bibitem%[BR]
{Bou} 
N.~Bourbaki, 
``Vari\'et\'es diff\'erentielles et analytiques. Fascicule de resultats.''
Hermann, Paris, 1975.


\bibitem%[BR]
{LP2} 
O.~Bratteli, D.W.~Robinson, 
``Operator algebras and quantum statistical mechanics.'' Vol. 1. 
%C∗- and W∗-algebras, symmetry groups, decomposition of states. 
Second edition. Texts and Monographs in Physics. 
Springer-Verlag, New York, 1987.

\bibitem%[BuCWZ04]
{BuCWZ04}
H.~Bursztyn, M.~Crainic, A.~Weinstein, C.~Zhu, 
{\it Integration of twisted Dirac brackets}. 
Duke Math. J. {\bf 123} (2004), no. 3, 549--607.

\bibitem
{CP12} 
P.~Cabau, F.~Pelletier,
{\it Almost Lie structures on an anchored Banach bundle}. 
J. Geom. Phys. {\bf 62} (2012), 2147--2169. 

\bibitem{CGM17}
A.~Cabrera, M.~Gualtieri, E.~Meinrenken, 
Dirac geometry of the holonomy fibration. 
\textit{Comm. Math. Phys.} \textbf{355} (2017), no. 3, 865--904.

\bibitem{CW}
A.~Cannas da Silva, A.~Weinstein, 
``Geometric models for noncommutative algebras.'' 
Berkeley Mathematics Lecture Notes, 10. American Mathematical Society, Providence, RI; 
Berkeley Center for Pure and Applied Mathematics, Berkeley, CA, 1999.

\bibitem
{ChM} 
P.R.~Chernoff, J.E.~Marsden, 
``Properties of infinite dimensional Hamiltonian systems.'' 
Lecture Notes in Mathematics, 425, New York, Springer-Verlag, 1974. 

\bibitem{CM25}
E.~Chiumiento, P.~Massey, 
Geometric approach to the Moore-Penrose inverse and the polar decomposition of perturbations by operator ideals. 
\textit{Forum Math.} \textbf{37} (2025), no. 3, 871--897. 

\bibitem%[Co73]
{Co73}
A.~Connes, 
{\it Une classification des facteurs de type III}. 
Ann. Sci. \'Ecole Norm. Sup. (4) {\bf 6} (1973), 133--252. 


\bibitem%[CoStHa85]
{CoStHa85}
A.~Connes, U.~Haagerup, E.~St\o rmer, 
{\it Diameters of state spaces of type ${\rm III}$ factors}. 
In: H. Araki, C.C. Moore, \c S Str\u atil\u a and D. Voiculescu (eds.), ``Operator algebras and their connections with topology and ergodic theory (Bu\c steni, 1983).'' Lecture Notes in Math., 1132, Springer, Berlin, 1985, pp. 91--116. 

\bibitem%[CoSt78]
{CoSt78}
A.~Connes, E.~St\o rmer, 
{\it Homogeneity of the state space of factors of type ${\rm III}_1$}. 
J. Functional Analysis {\bf 28} (1978), no. 2, 187--196. 

\bibitem%[DZ05]
{DZ05}
J.-P.~Dufour, N.T.~Zung, 
``Poisson structures and their normal forms.'' 
Progress in Mathematics, 242. Birkh\"auser Verlag, Basel, 2005. 

\bibitem{FNO25}
J.~Frahm, K.-H.~Neeb, G.~\'Olafsson, 
Realization of unitary representations of the Lorentz group on de Sitter space. 
\textit{Indag. Math. (N.S.)} \textbf{36} (2025), no. 1, 61--113. 

\bibitem{GN26}
H.~Gl\"ockner, K.-H.~Neeb, 
``Infinite-dimensional Lie groups''. 
\textit{Preprint}, arXiv:2602.12362 [math.FA] (2026).

\bibitem%[H96]
{H96}
R.~Haag, 
``Local quantum physics. Fields, particles, algebras''. 
Second edition. Texts and Monographs in Physics. 
Springer-Verlag, Berlin, 1996.

\bibitem%[Ha73]
{Ha73}
U.~Haagerup, 
{\it The standard form of von Neumann algebras}. 
Kobenhavns Universitet/ Matematisk Institut Preprint Series no. 15 (1973). 

\bibitem%[Ha75]
{Ha75}
U.~Haagerup, 
{\it The standard form of von Neumann algebras}. 
Math. Scand.  {\bf 37} (1975), no. 2, 271--283.

\bibitem%[HS90a]
{HS90a}
U.~Haagerup, E.~St\o rmer, 
{\it Equivalence of normal states on von Neumann algebras and the flow of weights}. 
Adv. Math. {\bf 83} (1990), no. 2, 180--262.

\bibitem%[Ka86]
{Ka86} 
M.V.~Karas\"ev, 
{\it Analogues of objects of the theory of Lie groups for nonlinear Poisson brackets}. 
Izv. Akad. Nauk SSSR Ser. Mat. {\bf 50} (1986), no. 3, 508--538, 638.

\bibitem{KM97}
A.~Kriegl, P.W.~Michor, 
``The convenient setting of global analysis''. 
Mathematical Surveys and Monographs, 53. 
American Mathematical Society, Providence, RI, 1997. 


\bibitem
{Lang} 
S.~Lang, 
``Introduction to differentiable manifolds.'' New York, 1972.

\bibitem%[Ma05]
{Ma05}
K.C.H.~Mackenzie, 
``General theory of Lie groupoids and Lie algebroids''. 
London Mathematical Society Lecture Note Series, 213. Cambridge Univ.\ Press, Cambridge, 2005.

\bibitem{MN24}
V.~Morinelli, K.-H.~Neeb, 
From local nets to Euler elements. 
\textit{Adv. Math.} \textbf{458} (2024), part A, Paper No. 109960, 87 pp.

\bibitem{MNOl24}
V.~Morinelli, K.-H.~Neeb, G.~\'Olafsson, 
{\it Modular geodesics and wedge domains in non-compactly causal symmetric spaces}. 
Ann. Global Anal. Geom. {\bf 65} (2024), no. 1, Paper No. 9, 50 pp. 

\bibitem%[Ne00]
{Ne00}
K.-H.~Neeb, 
``Holomorphy and convexity in Lie theory''. 
De Gruyter Expositions in Mathematics, 28. 
Walter de Gruyter \& Co., Berlin, 2000. 

\bibitem%[Ne02]
{Ne02}
K.-H.~Neeb, 
{\it A Cartan-Hadamard theorem for Banach-Finsler manifolds}. 
Geom. Dedicata {\bf 95} (2002), 115--156. 

\bibitem%[Ne18]
{Ne18}
K.-H. Neeb, 
{\it On the geometry of standard subspaces}. 
In: J.G. Christensen, S. Dann, and M. Dawson (eds.), 
``Representation theory and harmonic analysis on symmetric spaces'',  
Contemp. Math., 714, Amer. Math. Soc., Providence, RI, 2018, 199--223.


\bibitem%[NOOl20]
{NOOl20}
K.-H.~Neeb, B.~\O rsted, G. \'Olafsson, 
{\it Standard subspaces of Hilbert spaces of holomorphic functions on tube domains}. 
Comm. Math. Phys. {\bf 386} (2021), no. 3, 1437--1487.. 

\bibitem%[NO98]
{NO98}
K.-H.~Neeb, B.~\O rsted, 
{\it Unitary highest weight representations in Hilbert spaces of holomorphic functions on infinite-dimensional domains}. 
J. Funct. Anal.{\bf 156} (1998), no. 1, 263--300.

\bibitem%[NOl19]
{NOl19}
K.-H. Neeb, G. \'Olafsson, 
{\it KMS conditions, standard real subspaces and reflection positivity on the circle group}. 
Pacific J. Math. 299 (2019), no. 1, 117--169. 

\bibitem%[NOl21]
{NOl21}
K.-H. Neeb, G. \'Olafsson, 
Nets of standard subspaces on Lie groups. 
\textit{Adv. Math.} \textbf{384} (2021), Paper No. 107715, 69 pp. 


\bibitem%[NOl19]
{NOl23}
K.-H. Neeb, G. \'Olafsson, 
{\it Algebraic quantum field theory and causal symmetric spaces}. 
In: P. Kielanowski, A. Dobrogowska, G. A. Goldin, T. Goli\'nski (eds.), 
``Geometric methods in physics XXXIX'', Trends Math., Birkhäuser/Springer, Cham, 2023, pp.~207--231.

\bibitem{Neeb} 
K.-H.~Neeb, H.~Sahlmann, T.~Thiemann, 
{\it Weak Poisson structures on infinite dimensional manifolds and Hamiltonian actions}. 
In: V. Dobrev (ed.), ``Lie theory and its applications in physics.'' Springer Proc. Math. Stat., 111, Springer, Tokyo, 2014, pp.~105--135.  
%	DOI $10.1007/978-4-431-55285-7-8$

\bibitem%[Ni97]
{Ni97}
L.I. Nicolaescu, 
{\it Generalized symplectic geometries and the index of families of elliptic problems}. Mem. Amer. Math. Soc. {\bf 128} (1997), no. 609, {\rm xii}+80 pp.

\bibitem%[Od88]
{Od88}
A.~Odzijewicz, 
{\it On reproducing kernels and quantization of states}. 
Comm. Math. Phys. {\bf 114} (1988), no. 4, 577--597. 

\bibitem{OJS1} A. Odzijewicz, G. Jakimowicz, A. Sli\.{z}ewska, 
{\it Banach-Lie algebroids associated to the groupoid of partially invertible elements of a $W^*$-algebra}.  
J. Geom. Phys. {\bf 95} (2015), 108--126.

\bibitem{OJS2} A. Odzijewicz, G. Jakimowicz, A. Sli\.{z}ewska, 
{\it Fibre-wise linear Poisson structures related to $W^*$-algebras}. 
J. Geom. Phys. {\bf 123} (2018), 385--423

\bibitem%[OR03]
{OR03}
A.~Odzijewicz, T.S.~Ratiu, 
{\it Banach Lie-Poisson spaces and reduction}. 
Comm. Math. Phys. {\bf 243} (2003), no. 1, 1--54. 

\bibitem{OS} A. Odzijewicz, A. Sli\.{z}ewska, 
{\it Banach-Lie groupoids associated to $W^*$-algebras}.  
J. Symplectic Geom. {\bf 14} (2016), no. 3, 687--736. 

\bibitem{Pe12}
F.~Pelletier, 
{\it Integrability of weak distributions on Banach manifolds}. 
Indag. Math. (N.S.) {\bf 23} (2012), no. 3, 214--242.


\bibitem
{Sakai}
S.~Sakai, 
``$C^*$-algebras and $W^*$-algebras''. 
Springer-Verlag, Berlin, 1998.

\bibitem{S23}
A.~Schmeding, 
``An introduction to infinite-dimensional differential geometry''. 
Cambridge Studies in Advanced Mathematics, 202. Cambridge University Press, Cambridge, 2023.

\bibitem%[SW15]
{SW15}
A.~Schmeding, C.~Wockel, 
The Lie group of bisections of a Lie groupoid. 
Ann. Global Anal. Geom. 48 (2015), no. 1, 87--123.

\bibitem%[SW16]
{SW16}
A.~Schmeding, C.~Wockel, 
{\it (Re)constructing Lie groupoids from their bisections and applications to prequantisation}. 
Differential Geom. Appl. {\bf 49} (2016), 227--276.


\bibitem
{SS}
H.J.~Sussman, 
{\it Orbits of families of vector fields and integrability of distributions}. 
Trans. Amer. Math. Soc. {\bf 180} (1973), 171--188.

\bibitem%[Ta02]
{Ta02}
M.~Takesaki, 
``Theory of operator algebras.'' I. 
Encyclopaedia of Math. Sciences, 124. 
Operator Algebras and Non-commutative Geometry, 5. Springer-Verlag, Berlin, 2002. 

\bibitem%[Ta03a]
{Ta03a}
M.~Takesaki, 
``Theory of operator algebras''. II. 
Encyclopaedia of Math. Sciences, 127. Operator Algebras and Non-commutative Geometry, 8. 
Springer-Verlag, Berlin, 2003.

\bibitem%[Up85]
{Up85}
H.~Upmeier, ``Symmetric Banach manifolds and Jordan $C^*$-algebras''. North-Holland Math. Studies, 104. Notas de Mat., 96. North-Holland Publishing Co., Amsterdam, 1985.

\bibitem%[We97]
{We87}
A.~Weinstein, 
{\it Symplectic groupoids and Poisson manifolds}. 
Bull. Amer. Math. Soc. (N.S.) {\bf 16} (1987), no. 1, 101--104.

\bibitem%[We97]
{We97}
A.~Weinstein, 
{\it The modular automorphism group of a Poisson manifold}. 
J. Geom. Phys. {\bf 23} (1997), no. 3--4, 379--394.


\end{thebibliography}
\end{document}